\documentclass[reqno]{amsart}
\usepackage[top=1.1in, bottom=1in, left=1in, right=1in]{geometry}
\usepackage{amsfonts}
\usepackage{amssymb}
\usepackage{amsthm}
\usepackage{amsmath}
\usepackage{mathrsfs}
\usepackage{color}
\usepackage{enumerate}
\usepackage[numbers,sort&compress]{natbib}
\usepackage{pdfsync}
\usepackage{esint}
\usepackage{graphicx}
\usepackage{float}
\usepackage{caption}
\usepackage{subfigure}
\usepackage{enumitem}
\usepackage{appendix}
\usepackage{bbm}

\allowdisplaybreaks

\usepackage{tikz} % Use TikZ for drawing
\usetikzlibrary{calc}
\usetikzlibrary{intersections}
\usetikzlibrary{patterns}
\usepackage[colorlinks,
            linkcolor=black,       %%�޸Ĵ˴�Ϊ����Ҫ����ɫ
            anchorcolor=blue,      %%�޸Ĵ˴�Ϊ����Ҫ����ɫ
            citecolor=blue,        %%�޸Ĵ˴�Ϊ����Ҫ����ɫ�������޸�blueΪred
            ]{hyperref}
\makeatother

\numberwithin{equation}{section}

\newtheorem{theorem}{Theorem}[section]
\newtheorem{definition}[theorem]{Definition}

\newtheorem{lemma}[theorem]{Lemma}

\newtheorem{proposition}[theorem]{Proposition}
\newtheorem{corollary}[theorem]{Corollary}
\numberwithin{equation}{section}

\theoremstyle{remark}
\newtheorem{remark}[theorem]{Remark}

\newcommand{\R}{{\mathbb R}}

\def\wh{\widehat}
\def\9{{\infty}}
\def\ve{{\varepsilon}}
\def\na{{\nabla}}
\def\bbe{{\mathbb{E}}}
\def\bbr{{\mathbb{R}}}

\def\bbp{{\mathbb{P}}}
\def\({\left(}
\def\){\right)}
\def\<{\left<}
\def\>{\right>}

\newcommand{\one}{\mathbbm{1}}

\newcommand{\imu}{\mathrm{i}}
\newcommand{\dd}{\,\mathrm{d}}
\renewcommand{\epsilon}{\varepsilon}

\renewcommand{\Im}{\operatorname{Im}}
\renewcommand{\Re}{\operatorname{Re}}

\newcommand{\N}{{\mathbb{N}}}
\newcommand{\Q}{{\mathbb{Q}}}

\newcommand{\C}{{\mathbb{C}}}
\newcommand{\Z}{{\mathbb{Z}}}

\newcommand{\PP}{{\mathbb{P}}}
\newcommand{\G}{{\mathbb{G}}}

\newcommand{\X}{{\mathbb{X}}}
\newcommand{\Y}{{\mathbb{Y}}}

\newcommand{\cF}{{\mathcal{F}}}

\newcommand{\cB}{{\mathcal{B}}}

\newcommand{\cI}{{\mathcal{I}}}
\newcommand{\cJ}{{\mathcal{J}}}

\newcommand{\cP}{{\mathcal{P}}}
\newcommand{\cS}{{\mathcal{S}}}
\newcommand{\cT}{{\mathcal{T}}}
\newcommand{\cU}{{\mathcal{U}}}

\newcommand{\Schw}{\mathcal{S}}

\newcommand{\Sp}{{\mathbb S}}
\newcommand{\vece}{{\textnormal{\textbf{e}}}}

\newcommand{\ModN}{C_{\leq (\frac{\lambda}{2^8})^2}}

\newcommand{\TempN}{P^{(t)}_{\leq (\frac{\lambda}{2^8})^2}}

\newcommand{\SOne}{S^{1,\frac{1}{4}}}
\newcommand{\SHalf}{S^{\frac{1}{2},0}}
\newcommand{\Ssa}{S^{s,a}}
\newcommand{\NOne}{N^{1,\frac{1}{4}}}
\newcommand{\NHalf}{N^{\frac{1}{2},0}}
\newcommand{\Nsa}{N^{s,a}}
\newcommand{\XOne}{\mathbb{X}^1}

\newcommand{\Xs}{\mathbb{X}^{s}}
\newcommand{\GOne}{\mathbb{G}^1}
\newcommand{\GHalf}{\mathbb{G}^{\frac{1}{2}}}
\newcommand{\Gs}{\mathbb{G}^{s}}

\newcommand{\WEner}{W^{0,\frac{1}{4},\frac{1}{2}}}
\newcommand{\WEndp}{W^{0,0,0}}

\begin{document}
	
\title[]{The energy-critical stochastic Zakharov system}

\author{Sebastian Herr}
\address{Fakult\"at f\"ur Mathematik,
Universit\"at Bielefeld, D-33501 Bielefeld, Germany}
\email[Sebastian Herr]{herr@math.uni-bielefeld.de}
\thanks{}

\author{Michael R\"ockner}
\address{Fakult\"at f\"ur Mathematik,
Universit\"at Bielefeld, D-33501 Bielefeld, Germany, 
Academy of Mathematics and Systems Science, 
Chinese Academy of Sciences, 
Beijing, China.}
\email{roeckner@math.uni-bielefeld.de}
\thanks{}

\author{Martin Spitz}
\address{Fakult\"at f\"ur Mathematik,
Universit\"at Bielefeld, D-33501 Bielefeld, Germany}
\email[Martin Spitz]{mspitz@math.uni-bielefeld.de}
\thanks{}

\author{Deng Zhang}
\address{School of Mathematical Sciences, MOE-LSC,
CMA-Shanghai, Shanghai Jiao Tong University, China.}
\email[Deng Zhang]{dzhang@sjtu.edu.cn}
\thanks{}

%
%\subjclass[2020]{60H15, 60H5035Q55.}

\begin{abstract}
   This work is devoted to the stochastic Zakharov system in dimension four, which is the energy-critical dimension.
   First, we prove local well-posedness in the energy space $H^1\times L^2$
   up to the maximal existence time and a blow-up alternative.
 Second,
  we prove that for large data solutions exist globally
  as long as energy and wave mass are below the ground state threshold.
  Third, we prove a regularization by noise phenomenon:  the probability of global existence and scattering goes to one if the strength of the (non-conservative) noise goes to infinity.
  The proof is based on the refined rescaling approach and a new functional framework, where both Fourier restriction and local smoothing norms are used as well as a (uniform)  double endpoint Strichartz and local smoothing inequality for the Schrödinger equation with certain rough and time dependent lower order perturbations.
\end{abstract}

\maketitle

 \tableofcontents

%%%%%%%%%%%%%%%%%%%%%%%%%%

\section{Introduction and main results}   \label{Sec:Intro}

The present work is devoted to 
the energy-critical stochastic Zakharov system in dimension four  
\begin{equation}   \label{eq:StoZak}
	\left\{\aligned
	 \imu \dd X + \Delta X \dd t &= \Re(Y) X \dd t 
   - \imu \mu X \dd t + \imu X \dd W_1(t),  \\
	 \frac{1}{\alpha} \imu \dd Y + |\na|Y \dd t &= -|\na| |X|^2 \dd t + \dd W_2(t).
	\endaligned
	\right.
\end{equation}
Here the initial datum is in the energy space $(X(0), Y(0)) = (X_0, Y_0) \in H^1(\R^4) \times L^2(\R^4)$, 
$\alpha>0$ represents the ion sound speed, 
$X$ is the complex envelope of the electric field, 
and $Y$ is the (reduced) ion density fluctuation.  
The noises $W_1$ and $W_2$ represent, 
respectively, 
fluctuations in the plasma density and 
temperature, 
which are independent Wiener processes 
\begin{align} \label{Noise}
   W_j(t,x) = \sum\limits_{k=1}^\infty \imu \phi^{(j)}_k(x) \beta^{(j)}_k(t),
   \ \ x\in \bbr^4,\ \ t\geq 0
\end{align}
for $j=1,2$,  
where 
$\{\phi^{(1)}_k\}_k \subseteq H^4(\mathbb{R}^4)$  
and $\{\phi^{(2)}_k\}_k \subseteq H^2(\mathbb{R}^4$) are real-valued functions,
$\{\beta^{(j)}_k\}$ are real-valued independent Brownian motions
on a stochastic basis $(\Omega, \mathscr{F}, \{\mathscr{F}_t\}, \mathbb{P})$, 
and 
\begin{align*}
    \mu = \frac 12 \sum\limits_{k=1}^\infty 
    |\phi_k^{(1)}|^2 <\infty. 
\end{align*}
In 1972 Zakharov introduced the deterministic system (i.e.\ $W_1=W_2=0$)
to model the dynamics of Langmuir waves. 
A heuristic derivation of the stochastic noise can be found in \cite[Section 2]{HRSZ23}. 
It is worth noting that 
the Schr\"odinger component is driven by multiplicative noise, 
while the wave component is driven by additive noise. 
In particular, 
in the case where $\{\phi^{(1)}_k\}$ 
are real valued, 
$-\imu \mu X  \dd t + \imu X  \dd W_1(t) = \imu X \circ \dd W_1(t)$ is the standard Stratonovich differential, 
so that the mass of the Schr\"odinger component is conserved.

One remarkable relationship is that, 
via the subsonic limit $\alpha \to \infty$, the
Zakharov system \eqref{eq:StoZak} converges to the
focusing cubic nonlinear Schr\"odinger equations 
(NLS for short) 
\begin{align}   \label{equa-NLS}
   \imu \partial_t u + \Delta u  &= - |u|^2 u. 
\end{align}
We refer to ~\cite{SW86, OT92, KPV95, MN08}
for rigorous results in this direction.

Recently, 
substantial progress has been achieved towards
the understanding of 
solvability and long-term dynamics of 
the Zakharov system in the critical energy space. 
One key feature is that 
the sharp threshold for global well-posedness, blowup and scattering of the energy-critical Zakharov system is played by the ground state, 
that is, the famous Aubin-Talenti function 
\begin{align} \label{AT-funct}
    W(x)= \Big(1+\frac{|x|^2}{8}\Big)^{-1}. 
\end{align}   
This phenomenon was previously proved by Kenig and Merle  in the seminal paper \cite{KM06} 
for the energy-critical NLS. 
For the 4D Zakharov system,  
a sub-threshold Kenig-Merle dichotomy was derived by Guo-Nakanishi~\cite{GN21} for the radial Zakharov system: under the energy constraint 
\begin{align*}
    e_Z(u,v) < e_Z(W, - W^2), 
\end{align*} 
the energy space $H^1(\R^4)\times L^2(\R^4)$ 
topologically splits into two regimes 
\begin{align*}
    \{\|v\|_{L^2_x} < \|W^2\|_{L^2_x}\}\ \ \   
    \text{and}\ \ \  
    \{\|v\|_{L^2_x} > \|W^2\|_{L^2_x}\},
\end{align*} 
and all radial solutions in the former domain 
exist globally and scatter, 
but can not be global and bounded in the energy space in the latter regime. 
For  non-radial data of finite energy below the threshold, 
global well-posedness was proved by Candy-Herr-Nakanishi~\cite{CHN21} by developing a new type of adapted spaces and a uniform Strichartz estimate for the Schrödinger equation with a  potential. Also, 
the local regularity theory has been clarified in \cite{CHN23}. 
Very recently,  
in the case slightly above the threshold energy, 
finite time type II blow-up solutions have been constructed by 
Krieger-Schmid \cite{KS24a,KS24b},   
where the method is inspired by matched asymptotic regions and approximation procedures introduced by Krieger, Schlag and Tataru \cite{KST08,KST09a,KST09b}  
for critical nonlinear wave equations 
and subsequently developed methods 
by Perelman \cite{Pe14} and Ortoleva-Perelman \cite{OP14} for critical Schr\"odinger equations.  
Up to now, 
the scattering below the ground state for general data still remains a challenging problem.

In contrast to the above, 
very few results are known for {\it critical stochastic dispersive equations}. 
For the typical stochastic NLS, 
well-posedness and scattering were proved in the recent papers 
\cite{FX21, Zh23, OO20}. 
However,  
the theory for the {\it energy-critical}  stochastic Zakharov system is largely open.  
In the subcritical case where $d\leq 3$, well-posedness results were very recently proved 
in \cite{Ts22,Ba22,HRSZ23,BBD24}. 

In the 4D case, 
the energy regularity $H^1\times L^2$ lies at the boundary of the well-posedness regime for the Zakharov system, 
and cannot be treated directly by the normal form 
method \cite{BGHN15,HRSZ23}. 
The noise makes the situation worse by destroying the energy conservation law, 
so that the indirect method from \cite{BGHN15} does not apply to the stochastic case. 
Moreover, 
the low temporal $C^{1/2-}$-regularity of  the Wiener process does not allow for standard $X^{s,b}$-techniques.
For proving global well-posedness  for the stochastic Zakharov system, another main challange 
is to derive a uniform control 
for Schr\"odinger type equations with free-wave potentials and problematic derivative terms caused by the noise. 
This is in contrast to the stochastic NLS and the stability arguments for critical stochastic NLS 
developed in \cite{Zh23} 
are not applicable to the Zakharov system.

From the perspective of probability theory, 
it is widely expected that 
noise has regularizing effects on deterministic systems, 
such as preventing blowup and improving the regularity theory.
Regularization by noise phenomena have been observed for 
various SPDE models, including 
transport equations \cite{FGP10}, 
stochastic Hamilton-Jacobi equations \cite{GG19}, 
3D vorticity stochastic Navier-Stokes equations \cite{FL21}, 
and stochastic NLS \cite{BRZ17-1,HRZ19}.  
Recently, it has been shown in \cite{HRSZ23} that
norm explosion is prevented 
with high probability for the 3D Zakharov system driven by {\it non-conservative noise}  on bounded 
time intervals.
But  whether noise can prevent blowup for all times,  or even enforce scattering behavior, remains open.

The purpose of this work is to solve the 4D stochastic Zakharov system at the critical energy regularity
and to investigate the noise regularization effects on large time dynamics.
We mainly address the following three problems:
\begin{enumerate}
  \item[(i)] local well-posedness in the energy space including a blow-up alternative
  \item[(ii)] global well-posedness for large data below the ground state threshold
  \item[(iii)] noise regularization effects on global well-posedness and scattering
\end{enumerate}

We next present the precise formulation of the main results. 

Throughout this paper the spatial functions of the noise \eqref{Noise}
satisfy the following hypothesis.

\paragraph{\bf Hypothesis (H)}
The spatial functions $\{\phi^{(j)}_k\}$, $j=1,2$, satisfy 
the following summability conditions:
   \begin{align} \label{phik-condition}
   \sum\limits_{k=1}^\infty 
   \|\phi^{(1)}_k\|_{H^4}^2 +
   \sum\limits_{j=1}^4  \sum\limits_{k=1}^\infty \int  \sup_{y\in \bbr^3} |\nabla \phi^{(1)}_k(r \vece_j+y)|\dd r  <\infty, \ \
   \sum\limits_{k=1}^\infty 
   \|\phi^{(2)}_k\|_{H^2}^2 <\infty,
\end{align}
where $\vece_1, \vece_2, \vece_3, \vece_4$ denote the standard orthonormal basis of $\bbr^4$.

Theorem \ref{Thm-LWP}  shows that, 
for general initial data, 
the energy-critical stochastic Zakharov system 
 is locally well-posed in the energy space  
and satisfies a blow-up alternative.

\begin{theorem}    [LWP and blow-up alternative] \label{Thm-LWP}
	\label{thm:LocalWP}
Assume Hypothesis (H).
Then, given any deterministic initial data  $(X_0,Y_0) \in H^1(\R^4) \times L^2(\R^4)$, 
there exists a stopping time $\tau^* > 0$ such that the system~\eqref{eq:StoZak} has a unique $\{\cF_t\}$-adapted solution $(X,Y)$ in $C([0,\tau^*), H^1(\R^4) \times L^2(\R^4))$.
	
	Moreover, if $\tau^* <\infty$, then $\PP$-a.s.
	$$(i)\quad  \limsup_{t \rightarrow \tau^*} (\|X(t)\|_{H^1(\R^4)} + \|Y(t)\|_{L^2(\R^4)}) = \infty \qquad \text{or}\qquad (ii)\quad \|X\|_{L^2_t W^{\frac{1}{2},4}_x([0,\tau^*) \times \R^4)} = \infty.$$
\end{theorem}

The uniqueness statement in the above theorem means that the solution is unique in a suitable subspace of $C([0,\tau^*), H^1(\R^4) \times L^2(\R^4))$. We refer to Remark~\ref{rem:Uniqueness} for the precise formulation.

We remark that $L^2_t W^{\frac 12, 4}_x$ is the endpoint Strichartz space at the minimal regularity for the Schr{\"o}dinger component for which the deterministic Zakharov system is well-posed, see \cite{CHN23}.

Our proof introduces a {\it new functional framework}, a combination of the spaces in \cite{CHN21,CHN23} and local smoothing norms, which can deal with more general Zakharov systems with {\it derivative perturbations}.

\begin{theorem} [Zakharov system with derivative perturbations] \label{Thm-Zakharov-Low}
Consider the Zakharov system with lower order perturbations
\begin{equation}   \label{eq:StoZak-low}
	\left\{\aligned
	 \imu \partial_t u + \Delta u  &= \Re(v) u  - a_1\cdot \na u - a_0 u,  \\
	 \frac{1}{\alpha} \imu \partial_t v + |\na|v  &= -|\na| |u|^2,
	\endaligned
	\right.
\end{equation}
where the coefficients $a_1$ and $a_0$ are of the form
\begin{align}
 & a_1(t,x)= 2 \imu \sum\limits_{k=1}^\infty \na \phi_k(x)h_k(t),   \label{a1-loworder} \\
 & a_0(t,x)= - \sum\limits_{j=1}^4 \Big(\sum\limits_{k=1}^\infty \partial_j \phi_k(x)h_k(t)\Big)^2
          + \imu \sum\limits_{k=1}^\infty \Delta \phi_k(x) h_k(t), \label{a0-loworder}
\end{align}
and $\{\phi_k\}$ satisfy Hypothesis $(H)$, 
$\{h_k\} \subset C(\bbr^+; \bbr)$ 
and $h_k(0)=0$. 
Then, for any initial data $(u_0, v_0) \in H^1(\R^4) \times L^2(\R^4)$,
there exists a unique solution $(u,v)$ in $C([0,\tau^*), H^1(\R^4) \times L^2(\R^4))$ of \eqref{eq:StoZak-low}
up to a maximal existence time $\tau^*$.
Moreover,
one has the blow-up alternative as in Theorem \ref{Thm-LWP}
for system \eqref{eq:StoZak-low}.
\end{theorem}
As in Theorem~\ref{Thm-LWP}, uniqueness holds in a suitable subspace of $C([0,\tau^*), H^1(\R^4) \times L^2(\R^4))$.

Moreover, 
based on a {\it new uniform estimate}, 
we derive 
the global well-posedness below the ground state
for the energy-critical stochastic Zakharov system.
More precisely,
let $e_Z$ denote the Zakharov energy 
\begin{align} \label{energy-Zakh}
   e_Z(u,v) := 
   \int\limits_{\bbr^4} 
   \bigg(\frac 12 |\na u|^2
               + \frac 14 |v|^2 + \frac 12 \Re(v) |u|^2 
    \bigg) dx.
\end{align} 
The ground state for the Zakharov system 
is  $(W, -W^2)$ 
where $W$ is the Aubin-Talenti function  
given by~\eqref{AT-funct}.

\begin{theorem} [GWP below the ground state]  \label{Thm-GWP-Ground}
Assume (H). Let $(X_0, Y_0) \in H^1(\R^4) \times L^2(\R^4)$ 
be deterministic initial data 
satisfying 
\begin{align*}
	e_Z(X_0, Y_0) < e_Z(W, -W^2), \qquad \|Y_0\|_{L^2_x} \leq \|W^2\|_{L^2_x}.
\end{align*}
Let $(X,Y)$ be the corresponding unique solution of \eqref{eq:StoZak} on $[0,\tau^*)$ from Theorem~\ref{Thm-LWP},
where $\tau^*$ is the maximal existence time.
Define the $\{\mathscr{F}_t\}$-stopping times
\begin{align} \label{sigma*-def}
   \sigma_n^* := \inf\left\{t \in [0,\tau^*): e_Z\Big(e^{-W_1(t)} X(t),Y(t) + \imu \int_0^t e^{\imu (t-s) |\nabla|} \dd W_2(s)\Big) \geq e_Z(W, - W^2) - \frac{1}{n} \right\}
\end{align}
for all $n \in \N$ and the stopping time $\sigma^*$ as the pointwise limit of the monotonically increasing sequence $(\sigma_n^*)$.

We then have  $\tau^*\geq \sigma^*$, $\bbp$-a.s., i.e.
$(X,Y)$ exists at least up to the stopping time $\sigma^*$.
\end{theorem}

In the deterministic case, 
one indeed has 
$\sigma^*=\infty$ due to the energy conservation law. 
But the presence of noise destroys the energy conservation. 
Moreover, 
the noise has large fluctuations at infinity, 
which may even push the energy to exceed the ground state energy. 
Thus, 
it may happen that $\sigma^*<\infty$ 
with positive probability. 

However, 
the next result shows that, 
driven by a suitable  {non-conservative noise}, 
stochastic solutions to the Zakharov system 
exist globally and scatter at infinity 
with high probability, 
even for general data above the ground state energy.

\begin{theorem}[Noise regularization effects on 
blowup and scattering]
	\label{thm:RegNoise}
	Consider the stochastic Zakharov system \eqref{eq:StoZak}
with  a one-dimensional Brownian motion $W_1$
with non-zero imaginary part as the driving noise,
i.e., $\phi_1^{(1)}$ is a constant with $\Im \phi_1^{(1)} \not = 0$,
$\phi_k^{(1)} = 0$ for $2\leq k<\infty$,
and $W_2 \equiv 0$.
Then, for any deterministic initial data
$(X_0, Y_0) \in H^1(\R^4) \times L^2(\R^4)$,
we have
\begin{align}  
\bbp ((X(t), Y(t)) \text{ scatters as } t \rightarrow \infty) \longrightarrow 1,\ \ \text{as} \ \Im \phi_1^{(1)}\to \infty,
\end{align}
where $(X,Y)$ denotes the solution of~\eqref{eq:StoZak}, 
and ``scatters'' means that there exists $(z_+,v_+)\in H^1(\R^4) \times L^2(\R^4)$ such that 
\begin{align}  \label{scatter-XY} 
   \lim_{t \rightarrow \infty} \|e^{-\imu t \Delta} e^{\wh \mu t - W_1(t)} X(t) - z_+\|_{H^1} = 0\ \  \text{and} \ \ \lim_{t \rightarrow \infty} \|e^{-\imu t |\nabla|} Y(t) - v_+\|_{L^2} = 0 
\end{align}
with
\begin{align} \label{def-whmu}
	\wh \mu := \frac{1}{2} (|\phi_1^{(1)}|^2 - (\phi_1^{(1)})^2).
\end{align}
\end{theorem}

\begin{remark}  
$(i)$ For the deterministic 4D Zakharov system, 
finite time type II blow-up solutions exist
if the energy is slightly above the threshold energy \cite{KS24a,KS24b}. 

The noise regularization effect in 
Theorem \ref{thm:RegNoise} shows that, with high probability,  
the non-conservative noise destroys the  dichotomy and finite time blow-up dynamics \cite{GN21,CHN21} in that
the corresponding stochastic solutions 
exist globally and scatter at infinity 
for general data, even above the 
threshold energy.  

$(ii)$ 
It is worth noting that, 
via It\^o's calculus, 
the mass of the Schr\"odinger component of solutions satisfies 
the evolution formula 
\begin{align*}
    \|X(t)\|_{L^2_x}^2 
    = \|X_0\|_{L^2_x}^2  
      - 2 \sum\limits_{k=1}^\infty \int_0^t \int |X(s)|^2 \text{Im} \phi_k^{(1)}  \dd \beta_k^{(1)}(s).   
\end{align*} 
Hence, 
in the case where $\{\phi_k^{(1)}\}$ are real valued as in Theorems \ref{thm:LocalWP} 
and \ref{Thm-GWP-Ground}, the mass of the Schr{\"o}dinger component of
solutions is conserved pathwisely. 
But in the non-conservative case  where $\text{Im}\phi_k^{(1)} \not =0$ 
as in Theorem \ref{thm:RegNoise}, 
the mass is a continuous positive martingale and hence conserved under expectation
rather than in the pathwise sense. 

The key observation in the proof of Theorem~\ref{thm:RegNoise} is that after a rescaling transform  
the non-conservative noise 
gives rise to a geometric Brownian motion 
in the wave nonlinearity, 
see \eqref{eq:RanZakNoncons}. 
Intuitively, 
the geometric Brownian motion decays exponentially fast at large time, 
which weakens the wave nonlinearity 
and the resulting solutions scatter at infinity.  A rigorous analysis involves an intricate trilinear estimate 
and a novel global-in-time $V^p$ type control of 
the geometric Brownian motion, see Subsection \ref{Subsec-Novelties} for more details.  
\end{remark}

\subsection{Background  and motivation}

\subsubsection{Deterministic Zakharov system} 

The coupling between Schr{\"o}dinger and wave equation in the Zakharov system leads to a rich local and global regularity theory, which has attracted a lot of interest over the years. Several key results in the theory have been established only recently. We refer to~\cite{BC96,GTV97,Sa22} and the references therein for dimensions less than $4$ and to~\cite{CHN23} and the references therein for higher dimensions and concentrate on discussing the $4D$ energy-critical case in the following.

In dimension four, 
the Zakharov system  is energy critical in the sense that the focusing NLS \eqref{equa-NLS}, which is the subsonic limit, is scale invariant in $\dot{H}^1(\R^4)$, which is the  energy-regularity of~\eqref{energy-Zakh}.  
The kinetic and potential energy have the same scaling,
and the sign-indefinite term of the energy
is controlled in the energy space by the critical Sobolev embedding
$\dot{H}^1(\bbr^4) \subset L^4(\bbr^4)$.

In the seminal work \cite{KM06}, 
Kenig-Merle proved the dynamical dichotomy into
scattering and blowup  for the 
radial case
by developing the concentration-compactness and rigidity method. Dodson extended this to the full energy space in \cite{Do19}.

A similar dichotomy for the 4D Zakharov system in the radial energy space was proved by Guo-Nakanishi \cite{GN21}. 
The key role to characterize the threshold 
of the dichotomy 
is played by the ground state, 
that is, the Aubin-Talenti function $W$
defined in \eqref{AT-funct}. 
It is an extremiser of the energy-critical Sobolev inequality 
\begin{align*}
    \|\varphi\|_{L^4(\mathbb{R}^4)} 
    \leq \frac{\|W\|_{L^4(\mathbb{R}^4)} }{\|\na W\|_{L^2(\mathbb{R}^4)} } 
         \|\na \varphi\|_{L^2(\mathbb{R}^4)}, 
         \ \ \varphi \in \dot{H}^1(\mathbb{R}^4). 
\end{align*}
Moreover, 
the family $W_\lambda(x):= \lambda W(\lambda x)$, $\lambda>0$,   
solves the static cubic NLS  
\begin{align} \label{AT-equa}
   -\Delta W = W^3. 
\end{align}
Correspondingly, 
$(W_\lambda, - W^2_\lambda)$, $\lambda>0$, 
are static non-dispersing solutions to the  Zakharov system \eqref{eq:StoZak}.  
One also has for the Zakharov energy 
\begin{align*} 
    e_Z(W, -W^2) = e_S(W) = 
    \frac 14 \|W^2\|^2_{L^2(\mathbb{R}^4)}, 
\end{align*} 
where $e_S$ denotes the energy of the cubic NLS 
\begin{align*}
    e_S(u) = \int_{\mathbb{R}^4} 
    \bigg( \frac 12 |\na u|^2  - \frac 14 |u|^4\bigg)  dx. 
\end{align*}

For the 4D
energy-critical Zakharov system, 
the difficulty for the global analysis of solutions 
is mainly due to the wave component 
with the low $L^2_x$ regularity. 
Smallness helps to control the Schr\"odinger-wave interaction in view of
\begin{align*}
    \|\Re(v) u \|_{L^{\frac{4}{3}}(\R^4)} 
    \leq \|v\|_{L^2(\R^4)} \|u\|_{L^4(\R^4)}. 
\end{align*} 
For small initial data, the global well-posedness and scattering in the energy space was derived in \cite{BGHN15} 
by an indirect weak compactness argument. 
Global well-posedness and scattering for radial data below the ground state was recently shown in~\cite{GN21}.
For general data below the ground state threshold, 
global well-posedness was proved 
in \cite{CHN21} by using adapted spaces, bilinear Fourier restriction estimates and a profile decomposition argument.
If scattering fails, then the existence of a minimal energy non-scattering solution
was proved in~\cite{Ca24}. However, scattering for general data remains an open problem. 

Above the ground state threshold it is known that there is grow-up from \cite{GN21}. Very recently, 
finite time blow-up solutions to the 4D Zakharov system 
have been constructed by Krieger-Schmid in \cite{KS24a,KS24b}.

\subsubsection{Stochastic NLS} 
As a closely related stochastic dispersive model, 
there have been many results  on stochastic NLS in the subcritical regime. 
It was first proved by de Bouard and Debussche \cite{DD99, DD03} 
that the stochastic NLS is globally well-posed 
in some subcritical regime. 
Afterwards, Millet-Brze\'{z}niak \cite{BM14} obtained global well-posedness of subcritical stochastic NLS on manifolds.  
The key tools there are  stochastic Strichartz estimates. 
Moreover, the existence of martingale solutions 
was developed for stochastic NLS 
in general geometrical manifolds in \cite{BHW19}. 
In \cite{BRZ14, BRZ16},   
based on the rescaling approach, 
global pathwise well-posedness 
was proved for stochastic NLS in the whole subcritical regime in general dimensions. In \cite{HRZ19}, global well-posedness and scattering have been established in energy-subcritical cases.

In the critical defocusing regime,   
global well-posedness was proved for the 1D mass-critical case in \cite{FX21}, where the arguments also can be generalized to the 3D case.  
For mass- and energy-critical cases in general dimensions, 
global well-posedness, scattering and a Stroock-Varadhan type support theorem were obtained in \cite{Zh23}, 
based on a different refined rescaling approach. 
We also refer to \cite{OO20} for the global well-posedness of critical stochastic NLS with additive noise, 
and to \cite{BRZ18, Zh20} for 
the existence of optimal controllers.  

In the critical focusing regime, 
several results have been obtained recently for 
stochastic blow-up and soliton dynamics.  
Based on numerical methods,  
stable stochastic blow-up solutions 
were investigated in \cite{MRRY21,MRY21}.  
Stochastic blow-up solutions with the ground state mass 
or with loglog blow-up rate 
were constructed in \cite{SZ23,FSZ22}.
Concerning multi-bubble blow-up solutions, we refer to the recent papers
\cite{SZ20,RSZ24}. Also, in spite of the breakdown of the pseudo-conformal symmetry in the stochastic case, 
stochastic multi-solitons have been constructed directly in \cite{RSZ23}.

\subsubsection{Stochastic Zakharov system}
In contrast to the above, very few results have been obtained for the stochastic Zakharov system. 
In \cite{Ts22}, 
Tsutsumi first proved the global well-posedness for the stochastic 1D Zakharov system with additive noise. 
Different to the deterministic case, 
the solutions are constructed in $X^{s,b}$ spaces 
with $b<1/2$, due to the temporal $C^{1/2-}$-irregularity of the Wiener process. 
The subsonic limit problem was analyzed in \cite{Ba22,BBD24}
in the 1D setting with additive noise in the wave component.

In the 3D case,  well-posedness in the energy space 
was proved by the authors in~\cite{HRSZ23}.   
Unlike in \cite{Ts22, Ba22,BBD24}, 
the proof there is based on the normal form method 
and the refined rescaling approach. 
The normal form is crucially used to recover the necessary regularity in the Schr\"odinger-wave interaction.  

We point out that in the 4D case one cannot solve the problem in the energy space by the normal form method directly, see~\cite{BGHN15}. 
In~\cite{BGHN15}, this problem was circumvented by a compactness argument based on the energy conservation of the deterministic Zakharov system.
However, the energy is no longer conserved in the stochastic case. 
In this paper we use a direct method based on adapted Fourier restriction spaces and lateral Strichartz spaces to treat the 4D energy space, which avoids the normal form and builds on the approach devised in \cite{CHN21,CHN23} instead.

\subsubsection{Noise regularization effects} 
Noise regularization phenomena have been observed for various stochastic models. 
In the finite dimensional case,  
it is well-known that 
noise can improve well-posedness properties for differential equations with irregular drifts, 
see, e.g., \cite{KR05, Ve80}.   
This kind of regularization effect was also proved for 
infinite dimensional SPDEs 
with non-regular drifts \cite{DFR16}.  
Moreover, 
in \cite{FGP10}, Flandoli-Gubinelli-Priola 
showed that transport type noise improves the uniqueness of transport equations, 
even in the case where deterministic solutions lose 
uniqueness. 
Recently, Flandoli-Luo \cite{FL21} proved that transport noise can prevent blowup with high probability for 3D vorticity stochastic Navier-Stokes equations. 

Regarding stochastic dispersive equations, 
it was investigated numerically 
that multiplicative noise has a regularization effect in the sense that it 
delays blowup \cite{DDi02a, DDi02b}. 
In \cite{BRZ17-1}, 
it was found that norm explosion can be prevented 
for mass-(super)critical stochastic NLS  
by non-conservative noise, 
for which solutions conserve the mass on average 
rather than in the pathwise sense. 
The effect of non-conservative noise 
on scattering for stochastic NLS 
was analyzed in \cite{HRZ19}. 
Very recently, 
the effect of superlinear noise on non-explosion was proved in \cite{BFMZ24} for stochastic NLS with arbitrary power nonlinearity. 
Moreover, a regularization-by-noise effect on preventing blowup on any bounded time interval was proved by the authors for the 3D Zakharov system \cite{HRSZ23}.

\subsection{Novelties of the present work} 
\label{Subsec-Novelties}
The present paper mainly investigates the 4D energy-critical stochastic Zakharov system. 

The novelties of the present paper  
can be summarized as follows 
\begin{enumerate}
  \item[(i)] construction of a new functional framework
  which includes adapted Fourier restriction and lateral Strichartz spaces and is compatible with the refined rescaling transformations; 
  \item[(ii)]  
  uniform Strichartz type estimates for Schr\"odinger equations with free-wave potentials below the ground state 
  and with first order perturbations; 
  \item[(iii)] 
   a noise regularization effect on global scattering dynamics, via a global-in-time $V^p$ control of geometric Brownian motions. 
\end{enumerate}

More detailed explanations are presented in the following  subsections. 

\subsubsection{A new functional framework}  
One of the main difficulties in the analysis of the Zakharov system arises from the regularity of the Schr{\"o}dinger-wave interaction. This is already the case in the deterministic setting and only becomes more challenging in the stochastic one. We thus apply a rescaling transform, which was introduced in~\cite{HRSZ23} for the stochastic Zakharov system by the authors, in order to transform the stochastic Zakharov system to an equivalent system of random PDEs, see Section~\ref{Sec-Rescaling}. The advantage of this approach is that we can treat the resulting system pathwisely. Besides the benefit that our results such as local well-posedness will be in a pathwise sense, we can now use sophisticated analytical tools to address the regularity issue in the Schr{\"o}dinger-wave interaction. The price one has to pay is that the rescaling transform gives rise to random first order perturbation terms, see~\eqref{eq:RanZakbc} below. These first order terms are at the critical regularity level for a perturbative approach and lead to many difficulties not present in the deterministic system, both in the local and the global analysis.

In particular, we have to develop a new functional framework which 
can control, simultaneously, the Schr\"odinger-wave interaction 
and the critical derivative terms caused by the noise. 
The new function space~$\X^s$, which we introduce for the Schrödinger part (see Subsection~\ref{Subsec:FunctFrame}),
consists of two main ingredients: 
the adapted Fourier restriction spaces very recently developed in \cite{CHN21,CHN23}, and lateral Strichartz spaces established in the context of Schr\"odinger maps~\cite{BIKT11}.

Compared to the theory of adapted spaces in \cite{CHN21,CHN23}, 
two new contributions 
are the following: 
\begin{enumerate}
    \item[$\bullet$] Compatibility between adapted spaces and lateral Strichartz spaces, in particular the linear inhomogenoeus estimate in Lemma \ref{lem:LinEstimates}.
  \item[$\bullet$] Compatibility between adapted spaces and the refined rescaling transformations, in particular a product type estimate in Lemma \ref{lem:ProductNoiseInX}.
\end{enumerate}
Let us also mention that the new functional spaces 
admit the {\it decomposability property}, 
that permits to glue together solutions 
in the refined rescaling procedure, 
see Appendix \ref{Sec-Decom}. 
In addition, a new argument involving 
two sequences of stopping times is used 
in the fixed point argument 
in Subsection \ref{Subsec-LWP}.

\subsubsection{Uniform estimates below the ground state} 
In view of the blow-up alternative,  
in order to prove the global existence of solutions, 
it is crucial to derive a global bound for solutions in the critical endpoint space. 

For critical stochastic NLS,  
the global bounds were derived by stability estimates 
with respect to the deterministic NLS, 
together with the refined rescaling approach \cite{Zh23}.
But for the Zakharov system, 
because of the Schr\"odinger-wave interaction 
one has to derive a uniform global bound for solutions of Schr\"odinger equations with free-wave potentials below the ground state as in \cite{GN21,CHN21}. 

For the stochastic Zakharov system,  
one needs to 
derive a uniform estimate   
in the endpoint space $L^2_tW^{\frac 12, 4}_x$ 
for Schr\"odinger equations with both a free-wave potential $v_L$ and extra lower order perturbations 
\begin{align} \label{equa-low}
   \imu \partial_t u + \Delta u  - \Re(v_L) u + b\cdot \na u + c u - \Re(\cT_t(W_2)) u = f, 
\end{align}  
where the coefficients $b$ and $c$ are random and
arise from the noise via the rescaling transformation, 
and $\cT_t(W_2)$ is the stochastic convolution of the wave noise.  
In fact, we do not only apply one rescaling transform, but the extension to the maximal existence time requires the refined rescaling approach, which is a sequence of rescaling transformations, see Proposition~\ref{prop:RefinedRescaling}.
When implementing the refined rescaling approach, restarting at subsequent stopping times
gives rise to different free-wave potentials, 
which forces us to derive uniform estimates 
for the equation~\eqref{equa-low}. In~\cite{CHN21}, uniform estimates for~\eqref{equa-low} without the terms arising from the noise have been established using concentration compactness arguments.
It is not clear to us how to extend this approach to lateral Strichartz spaces in order to control the additional lower order terms, in particular, the derivative 
term $b\cdot \na u$.

Proceeding differently, we present a simplified argument to derive   
\begin{align*}
		\|u\|_{\SHalf(I)}
     \leq C (\|u_0\|_{H^{\frac{1}{2}}_x(\R^4)} +   \|f\|_{\NHalf(I)} +  |I|^{\frac{1}{2}})
	\end{align*} 
for solutions of~\eqref{equa-low}, where the constant $C$ depends only on the mass of the free wave, the energy norm of $u$, and the noise, see Proposition~\ref{prop:UniformEstimateBelowGroundState}. (We refer to Subsection~\ref{Subsec:FunctFrame} for the definition of the adapted spaces $\SHalf(I)$ and $\NHalf(I)$.)
Besides the uniform estimate in \cite{CHN21}, 
one key ingredient of the proof is 
a double expansion of Duhamel operators, 
i.e., 
	\begin{equation*}
		\cI_{v_L} = [I + \cI_{v_L} \Re(v_L)] \cI_0,
	\end{equation*} 
where $\cI_{v_L}$ denotes the Duhamel operator 
of the Schr{\"o}dinger equation with free-wave potential $v_L$.   
See more details in the proof of Proposition \ref{prop:UniformEstimateBelowGroundState}.

\subsubsection{Noise regularization effects on blowup and scattering} 

We note that the non-conservative noise 
is structurally different from the conservative case studied in the previous subsections. 

More precisely, we use a different rescaling transform than before. We set
\begin{align}
\label{eq:RescalingTransformNoncons}
    z := e^{\wh \mu t - W_1(t)} X,  \qquad	 v := Y
\end{align}
with $\wh \mu$ given by \eqref{def-whmu}.  
Note that
in the non-conservative case considered in Theorem~\ref{thm:RegNoise}, 
\begin{align}  \label{wtmu-Imphi2}
	 \Re \wh \mu
	 =  (\Im \phi_1^{(1)})^2 >0,
\end{align}
while $\Re \wh \mu = 0$ in the conservative case as in Theorem~\ref{thm:LocalWP} and Theorem~\ref{Thm-GWP-Ground}. 
The rescaling transform~\eqref{eq:RescalingTransformNoncons} converts the stochastic Zakharov system~\eqref{eq:StoZak} into the equivalent random system
\begin{equation}   \label{eq:RanZakNoncons}
	\left\{\aligned
	  \imu \partial_t z + \Delta z &= \Re(v)z,  \\
	 \imu \partial_t v + |\na |v  &= - h_{\Im\phi_1^{(1)}}|\na||z|^2, \\
    (u(0), v(0)) &= (X_0, Y_0),
	\endaligned
	\right.
\end{equation}
with $(X_0, Y_0)\in H^1(\R^4) \times L^2(\R^4)$ being deterministic initial data and $h_{\Im\phi_1^{(1)}}$ the geometric Brownian motion
\begin{align}  \label{h-W1-def}
	h_{\Im \phi_1^{(1)}}(t):= e^{2 \Re (W_1(t)- \wh \mu t)}
    = e^{-2 \Im \phi_1^{(1)} \beta^{(1)}_1(t) - 2 (\Im \phi_1^{(1)})^2 t}.
\end{align}
The crucial observation here is that because of the law of the iterated logarithm
\begin{align} \label{log-BM}
   \limsup\limits_{t\to \infty} \frac{\beta^{(1)}_1(t)}{\sqrt{2t \log\log t}} =1, \ \
    \liminf\limits_{t\to \infty} \frac{\beta^{(1)}_1(t)}{\sqrt{2t \log\log t}} = -1, \ \ \bbp\text{-a.s.},
\end{align}
the geometric Brownian motion $h_{\Im \phi_1^{(1)}}$ decays exponentially fast at infinity. Heuristically, one thus expects that the geometric Brownian motion weakens the nonlinearity in the wave equation and hence stabilizes the system.
The key step in order to exploit this exponential decay is the derivation of a suitable trilinear estimate for the wave nonlinearity in~\eqref{eq:RanZakNoncons}, which is the content of  Theorem~\ref{thm:TrilinearEstWave} below.
It should be mentioned that, 
although the geometric Brownian motion is independent of the spatial variable, there are new effects in the trilinear interactions which are not present in the deterministic setting \cite{CHN23}. 

We overcome this difficulty by uncovering the subtle {\it non-resonance identity} \eqref{eq:NonresonanceId}, which allows to transfer spatial regularity to temporal regularity of the geometric Brownian motion,   
and therefore requires a~\emph{global bound on some fractional derivative} of the geometric Brownian motion.   
This is achieved
by proving a new {\it global-in-time $V^p$ control} of the geometric Brownian motion  
\begin{align*}
     h \in V^p_0,\ \ \bbp\text{-a.s.},\ \ \text{for}\ \text{every}\ p>2,
\end{align*}
where $V^p_0$ is the space of functions of bounded $p$-variation \cite{Wi24}, 
but over $[0,\infty)$. 

We stress that for the global dynamics
it is crucial to work on unbounded time intervals, where Brownian motion has infinite $V^p$ variation. 
The idea to control the geometric Brownian motion 
globally in $V^p$ is to 
exploit its exponential decay, which has to be balanced carefully  
with the pathwise growth of the 
H\"older-norm 
$\|\beta(\cdot, \omega)\|_{C^{1/p}(n,n+1)}$ 
 of Brownian motions 
(see Proposition \ref{prop:GeometricBrownianMotionVp}).

\subsection{Notation} \label{sec:Notation}
Take an even function $\eta_0 \in C_c^\infty(\R)$ such that $0 \leq \eta_0 \leq 1$, $\eta_0(r) = 1$ for $|r| \leq \frac{5}{4}$, and $\eta_0(r) = 0$ for $|r| \geq \frac{8}{5}$. For every dyadic number $\lambda \in 2^\Z$ we set
\begin{align*}
	\chi_\lambda(r) = \eta_0(r/\lambda) - \eta_0(2r/\lambda), \qquad \chi_{\leq \lambda}(r) = \eta_0(r/\lambda)
\end{align*}
for all $r \in \R$. We define the standard Littlewood-Paley projectors as the spatial Fourier multipliers
\begin{align*}
    P_\lambda = \chi_\lambda(|\nabla|) \quad \text{if } \lambda > 1, \qquad P_1 = \chi_{\leq 1}(|\nabla|).
\end{align*}
Hence, $P_\lambda$ localizes the spatial Fourier support to the set $\{\lambda/2 < |\xi| < 2 \lambda\}$ if $\lambda > 1$ and to the set $\{|\xi| < 2\}$ if $\lambda = 1$.

Similarly, we define temporal frequency and modulation projectors by
\begin{align*}
    P^{(t)}_\lambda = \chi_\lambda(|\partial_t|), \qquad C_\lambda = \chi_\lambda(| \imu \partial_t + \Delta|)
\end{align*}
for all $\lambda \in 2^\Z$, i.e., $P^{(t)}_\lambda$ localizes temporal frequencies around $\lambda$ and $C_\lambda$ localizes the space-time Fourier support to distances of size $\lambda$ from the paraboloid. We also set
\begin{align*}
    P_{\leq \lambda} = \chi_{\leq \lambda}(|\nabla|), \qquad P^{(t)}_{\leq \lambda} = \chi_{\leq \lambda}(|\partial_t|), \qquad C_{\leq \lambda} = \chi_{\leq \lambda}(| \imu \partial_t + \Delta|),
\end{align*}
as well as $P_{> \lambda} = I - P_{\leq \lambda}$, $P^{(t)}_{> \lambda} = I - P^{(t)}_{\leq \lambda}$, and $C_{> \lambda} = I - C_{\leq \lambda}$. We also write
\begin{align*}
    \tilde{P}_\lambda = P_{\frac{\lambda}{2}} + P_\lambda + P_{2 \lambda}
\end{align*}
for the fattened Littlewood-Paley projectors and correspondingly for the temporal frequency and the modulation projectors. Sometimes, we also write $f_\lambda = P_\lambda f$ for the sake of brevity.

Besides the more sophisticated function spaces we introduce in Subsection~\ref{Subsec:FunctFrame} below, we employ the standard Besov and Sobolev spaces. These are defined as the sets of tempered distributions for which the following norms are finite. The inhomogeneous and homogeneous Sobolev spaces $W^{s,p}$ and $\dot{W}^{s,p}$ are defined via the norms
\begin{align*}
    \| f \|_{W^{s,p}} = \| \langle \nabla \rangle^s f \|_{L^p} \qquad \text{and} \qquad \| f \|_{\dot{W}^{s,p}} = \| | \nabla |^s f \|_{L^p},
\end{align*}
respectively. The defining norms for the inhomogeneous and homogeneous Besov spaces $B^{s}_{p, q}$ and $\dot{B}^{s}_{p,q}$ are
\begin{align*}
    \| f \|_{B^s_{p,q}} = \Big(\sum_{\lambda \in 2^{\N_0}} \lambda^{s q} \| P_\lambda f\|_{L^p}^q \Big)^{\frac{1}{q}} \qquad \text{and} \qquad 
    \| f \|_{\dot{B}^s_{p,q}} = \Big(\sum_{\lambda \in 2^{\Z}} \lambda^{s q} \| \dot{P}_\lambda f\|_{L^p}^q \Big)^{\frac{1}{q}},
\end{align*}
respectively, where we have written $\dot{P}_\lambda = \chi_\lambda(|\nabla|)$ $(\lambda \in 2^\Z)$ for the homogeneous Littlewood-Paley projectors.

The endpoint Strichartz space at endpoint regularity $L^2_t W^{\frac{1}{2},4}(I \times \R^4)$ will play a distinguished role in our analysis and will be denoted by $D(I)$ for any interval $I \subseteq \R$. In particular, we set
\begin{align*}
    \| u \|_{D(I)} = \| u \|_{L^2_t W^{\frac{1}{2},4}(I \times \R^4)}.
\end{align*}
We also note that $C_\lambda P_\lambda$, $C_{\leq \lambda} P_\lambda$, etc. are convolution operators with kernels bounded in $L^1(\R \times \R^4)$ independent of $\lambda$ so that these operators are bounded on all $L^q_t L^p_x$, $L^q_t W^{s,p}_x$, and $L^q_t B^s_{p,r}$ spaces uniformly in $\lambda$.

To distinguish different frequency interactions, we also introduce the standard paraproduct decomposition
\begin{align*}
    f g = (f g)_{LH} + (f g)_{HH} + (f g)_{HL},
\end{align*}
where the low-high, high-high, and high-low interactions are defined as
\begin{align*}
    (f g)_{LH} = \sum_{\lambda} P_{\leq \frac{\lambda}{2^8}} f P_\lambda g, \qquad 
    (f g)_{HH} = \sum_{\lambda_1 \sim \lambda_2} P_{\lambda_1} f P_{\lambda_2} g, \qquad
    (f g)_{HL} = (g f)_{LH}.
\end{align*}
Here the first sum runs over $\lambda \in 2^{\N}$ with $\lambda \geq 2^8$ and the second sum runs over all $\lambda_1, \lambda_2 \in 2^{\N_0}$ such that $| \log(\lambda_1/\lambda_2) | \leq 7$.

We finally introduce some notations concerning the solution operators of linear Schrödinger and wave equations.
        We write $\cI_0[g]$ for the inhomogeneous Schr\"odinger solution of
\begin{align*}
	(\imu \partial_t + \Delta) u = g, \qquad u(t_0) = 0,
\end{align*}
and $\cJ_0[h]$ for the inhomogeneous wave solution of
\begin{align*}
	(\imu \partial_t + |\nabla|) v = h, \qquad v(t_0) = 0,
\end{align*}
i.e., in the Duhamel form
\begin{equation}
	\label{eq:DefPropOp}
	\cI_0[g](t) = -\imu \int_{t_0}^t e^{\imu (t-s) \Delta} g(s) \dd s, \qquad
	\cJ_0[h](t) = -\imu \int_{t_0}^t e^{\imu (t-s)|\nabla|} h(s) \dd s.
\end{equation}
We will also consider the Schrödinger equation with potential $V$, i.e.,
\begin{align}
\label{eq:SchrPotentialIntro}
    (\imu \partial_t + \Delta - V) u = g, \qquad u(t_0) = f.
\end{align}
If they exist, we will denote the homogeneous propagation operator, i.e. the solution of~\eqref{eq:SchrPotentialIntro} with $g = 0$, by $\cU_V[f]$, and the inhomogeneous propagation operator, i.e. the solution of~\eqref{eq:SchrPotentialIntro} with $f = 0$, by $\cI_V[g]$. In order to ease the notation, we did not include the dependence on $t_0$ in the labeling of these operators and we shall take care that the considered $t_0$ will always be clear from the context.

\medskip
\paragraph{\bf Organization}
The remaining part of the present paper is orgarnized as follows. 
Section \ref{Sec-Rescaling} is concerned with the refined rescaling transformations, 
which permit to reduce the originial stochastic Zakharov system 
to random systems on different random intervals. 
Then, in Sections \ref{Sec:MultilinEst}-\ref{Sec-Product}, 
we construct the new functional framework 
and derive key estimates to show its compatibility
with lateral Strichartz spaces and refined rescaling transformations. 
Section \ref{Sec-LWP} is devoted to the proof of local well-posedness in Theorem \ref{thm:LocalWP}. 
The global well-posedness result below the ground state 
is proved in Section \ref{sec:WellPoBelowGrSt}. 
At last, Section \ref{Sec-Noise-Regular} is concerned with the noise regularization effects on blow-up and scattering. 
It also contains the key global $V^p$ control 
of geometric Brownian motions. In order to not disturb the flow of the main part of the paper, we prove some rather technical but important properties of our function spaces in the appendices.

\section{Refined rescaling transforms} \label{Sec-Rescaling}

Here we give the definition of a solution of~\eqref{eq:StoZak}. 
Without loss of generality, 
we take $\alpha=1$. 

\begin{definition}   \label{def:Solution}
Fix $T\in (0,\infty)$.
We say that $(X,Y)$ is a probabilistic strong solution to \eqref{eq:StoZak}
on $[0,\tau]$,
where $\tau\in(0,T]$ is an $\{\mathscr{F}_t\}$-stopping time,
if $(X,Y)$ is an $H^1\times L^2$-valued
$\{\mathscr{F}_t\}$-adapted process which belongs to $C([0,\tau],H^1 \times L^2)$
and satisfies $\mathbb{P}$-a.s.
for any $t\in [0,\tau]$,
\begin{equation}   \label{equa-stoZak-def}
	\left\{\aligned
	 X(t) &= \int_0^t \imu \Delta X \dd s - \int_0^t  \imu \Re(Y) X \dd s - \int_0^t \mu X \dd s + \int_0^t X \dd W_1(s),    \\
	Y(t) &= \int_0^t \imu |\na|Y \dd s + \int_0^t \imu |\na| |X|^2 \dd s - \imu W_2(t),
	\endaligned
	\right.
\end{equation}
as equations in $H^{-1} \times H^{-1}$.

Given an $\{\mathscr{F}_t\}$-stopping time $\tau^*$, we also call $(X,Y)$ a probabilistic strong solution to \eqref{eq:StoZak}
on $[0,\tau^*)$ if $(X,Y)$ is an $\{\mathscr{F}_t\}$-adapted process belonging to $C([0,\tau^*),H^1 \times L^2)$ such that for any $T \in (0,\infty)$ and any $\{\mathscr{F}_t\}$-stopping time $\tau < \tau^*$, $(X,Y)$ is a probabilistic strong solution to \eqref{eq:StoZak} on $[0,\tau \wedge T]$.
\end{definition}

\begin{remark}
	\label{rem:Uniqueness}
In the statement of Theorem~\ref{Thm-LWP} uniqueness means that for any $T \in (0,\infty)$ and any $\{\cF_t\}$-adapted stopping time $\tau < \tau^*$ the process $(X,Y)$ is the unique solution of~\eqref{eq:StoZak} in the sense of Definition~\ref{def:Solution} satisfying
\begin{align*}
    (X,Y) \in C([0, \tau \wedge T], H^1 \times L^2), \qquad e^{-W_1} X \in \XOne([0, \tau \wedge T]),
\end{align*}
where the space $\XOne$ is introduced in~\eqref{eq:DefRestrictionNorm}.
\end{remark}

Via the rescaling or Doss-Sussman type transforms
\begin{align}
     & u(t):= e^{-W_1(t)}X(t),   \label{rescal.1}  \\
     & v(t):= Y(t) - \cT_{t}(W_2) \ \
	\text{with} \ \cT_t(W_2):= - \imu \int_0^t e^{\imu (t-s)|\na|} \dd W_2(s).   \label{rescal.2}
\end{align}
we reduce~\eqref{eq:StoZak} to the random system
\begin{equation}   \label{eq:RanZakW1W2}
	\left\{\aligned
	  \imu \partial_t u + e^{-W_1}\Delta (e^{W_1} u) &= \Re(v) u + \Re(\cT_t(W_2)) u,  \\
	  \imu \partial_t v + |\na |v  &= - |\na||u|^2, \\
    (u(0), v(0)) &= (X_0, Y_0),
	\endaligned
	\right.
\end{equation}
or equivalently,
\begin{equation}   \label{eq:RanZakbc}
	\left\{\aligned
	  \imu \partial_t u + \Delta u
	   &= \Re(v) u - b\cdot \na u - cu + \Re(\cT_t(W_2)) u,  \\
	 \imu \partial_t v + |\na |v  &= - |\na||u|^2, \\
    (u(0), v(0)) &= (X_0, Y_0),
	\endaligned
	\right.
\end{equation}
where the coefficients $b$ and $c$ of the lower order perturbations are of the form
\begin{align}
   b &= 2 \na W_1 = 2 \imu \sum\limits_{k=1}^\infty \na \phi^{(1)}_k \beta^{(1)}_k,  \label{eq:Defb}  \\
   c &= |\na W_1|^2 + \Delta W_1
     = - \sum\limits_{j=1}^4 \Big(\sum\limits_{k=1}^\infty
	     \partial_j \phi_k^{(1)} \beta^{(1)}_k \Big)^2
        + \imu \sum\limits_{k=1}^\infty \Delta \phi^{(1)}_k \beta^{(1)}_k.   \label{eq:Defc}
\end{align}

This rescaling transform was introduced for the Zakharov system in dimension $3$ in~\cite{HRSZ23}. The equivalence of~\eqref{eq:StoZak} and~\eqref{eq:RanZakbc} is independent of the spatial dimension and we obtain~\cite[Theorem~3.1]{HRSZ23} also in dimension $d = 4$.

\begin{theorem} [Equivalence via rescaling transformations] \label{thm:Rescaling}
	\begin{enumerate}
\item[]
\item \label{it:EquivRescStochToRandom} Let $(X,Y)$ be a solution to \eqref{eq:StoZak} on $[0,\tau]$
in the sense of Definition~\ref{def:Solution},
where $\tau$ is an $\{\mathscr{F}_t\}$-stopping time
and $(X,Y) \in C([0,\tau]; H^1 \times L^2)$ $\bbp$-a.s.
Set $u: =e^{-W_1}X$ and $v:= Y - \cT_t(W_2)$.
Then,
$(u,v)$ is an analytically weak solution to~\eqref{eq:RanZakW1W2} on $[0,\tau]$
as equations in $H^{-1} \times H^{-1}$.

\item \label{it:EquivRescRandomToStoch} Let $(u,v)$ be an analytically weak solution to \eqref{eq:RanZakW1W2} on $[0,\tau]$
as equations in $H^{-1} \times H^{-1}$,
where $\tau$ is an $\{\mathscr{F}_t\}$-stopping time,
and $(u,v)$ is $\{\mathscr{F}_t\}$-adapted and continuous in $H^1\times L^2$.
Set $(X, Y): =(e^{W_1}u, v+\cT_t(W_2))$.
Then,
$(X, Y)$ is a solution of~\eqref{eq:StoZak} on $[0,\tau]$
in the sense of Definition~\ref{def:Solution}.
\end{enumerate}
\end{theorem}

The above results permit to construct local solutions up to a possibly very small stopping time, 
but are not sufficient to extend solutions to the maximal existence time. 
The key ingredient in the extension is the  refined rescaling approach 
for the Zakharov system introduced in~\cite{HRSZ23}. 
Since the statement and the proof are independent of the spatial dimension, we obtain~\cite[Proposition~3.2]{HRSZ23} also for $d = 4$.

\begin{proposition} [Refined rescaling transformations] \label{prop:RefinedRescaling}
Let $\sigma, \tau \colon \Omega \rightarrow [0,T]$ such that $\sigma+\tau\leq T$.
\begin{enumerate}
\item \label{it:RefResc0ToSig}
Let $(u_\sigma, v_\sigma) \in C([0,\tau], H^1 \times L^2)$
be an analytically weak solution of the system
\begin{equation}   \label{eq:RanZakSigma}
	\left\{\aligned
	  \partial_t u_\sigma(t) &=  \imu e^{-W_{1,\sigma}(t)} \Delta (e^{W_{1,\sigma}(t)} u_\sigma(t))
	-   \imu  \Re v_\sigma(t)\, u_\sigma(t)
	- \imu u_\sigma(t) \Re \cT_{\sigma+t, \sigma}(W_2),   \\
	  \partial_t v_\sigma(t) &=   \imu  |\na| v_\sigma(t) + \imu |\na| |u_\sigma(t)|^2,
	\endaligned
	\right.
\end{equation}
as equations in $H^{-1} \times H^{-1}$,
where the incerements of noises $W_{1,\sigma}$ and $ \cT_{\sigma+t, \sigma} (W_2)$ are defined by
\begin{align}
   & W_{1,\sigma}(t):= W_1(\sigma+t) - W_1(\sigma),   \label{eq:DefWsigma}  \\
   & \cT_{\sigma+t, \sigma} (W_2)
     := -\imu \int_\sigma^{\sigma+t} e^{\imu (\sigma+t-s)|\na|} \dd W_2(s)  \label{eq:DefTsigmaW2}
\end{align}
for all $t\in [0,\tau]$.
For any $t\in [\sigma, \sigma + \tau]$, we set
\begin{align}
	 u(t)
	&:= e^{-W_1(\sigma)} u _\sigma(t-\sigma),    \label{eq:DefusigmaRescal} \\
     v(t) &:= v_\sigma(t-\sigma)- e^{\imu (t-\sigma) |\na|}\cT_\sigma(W_2).    \label{eq:DefvsigmaRescal}
\end{align}
Then, $(u, v)$ is an analytically weak solution of system \eqref{eq:RanZakW1W2}
on $[\sigma, \sigma+\tau]$
with
\begin{align}
   & u(\sigma) =  e^{-W_1(\sigma)} u_\sigma(0),   \label{usigma-vsigma-initial.1}\\
   & v(\sigma) =   v_\sigma(0)-\cT_\sigma(W_2).    \label{usigma-vsigma-initial.2}
\end{align}

\item \label{it:RefRescSigTo0} If $(u,v) \in C([\sigma, \sigma+\tau], H^1\times L^2)$
is an analytically weak solution of system \eqref{eq:RanZakW1W2} on $[\sigma, \sigma+\tau]$
as equations in $H^{-1}\times H^{-1}$,
then
\begin{align}
	 & u_\sigma(t)
    :=  e^{W_1(\sigma)} u(\sigma+t),    \label{eq:u-sigma-v-sigma-rescal.1} \\
     & v_\sigma(t) := v(\sigma+t) + e^{\imu t|\na|}\cT_\sigma(W_2)),   \label{eq:u-sigma-v-sigma-rescal.2}
    \ \ t\in [0,\tau],
\end{align}
is an analytically weak solution of the system \eqref{eq:RanZakSigma} on $[0,\tau]$.
\end{enumerate}
\end{proposition}

\begin{remark}
	\label{rem:RanZakRefinedRescaling}
	Defining
	\begin{align}
	\label{eq:Defbsigmacsigma}
		b_\sigma := 2 \nabla W_{1,\sigma}, \qquad
        c_\sigma := |\nabla W_{1,\sigma}|^2 + \Delta W_{1,\sigma},
	\end{align}
	in the setting of Proposition \ref{prop:RefinedRescaling},
    we note that~\eqref{eq:RanZakSigma} is equivalent to
	\begin{equation}   \label{eq:RanZakbsigmacsigma}
	\left\{\aligned
	  \imu \partial_t u_\sigma + \Delta u_\sigma
	   &= \Re(v_\sigma) u_\sigma - b_\sigma \cdot \nabla u_\sigma - c_\sigma u_\sigma + \Re(\cT_{\sigma + \cdot, \sigma}(W_2)) u_\sigma,  \\
	 \imu \partial_t v_\sigma + |\nabla |v_\sigma  &= - |\nabla||u_\sigma|^2.
	\endaligned
	\right.
\end{equation}
\end{remark}

In order to extend solutions by means of Proposition~\ref{prop:RefinedRescaling},
we also need to be able to glue together solutions.
The following gluing procedure was already introduced in~\cite{HRSZ23},
and it is independent of the spatial dimension. We thus also have~\cite[Proposition~3.3]{HRSZ23} in dimension $d = 4$.

\begin{proposition} [Gluing solutions] \label{prop:GluingSolutions}
Let $(u_1,v_1)\in C([0,\sigma], H^1 \times L^2)$
be an analytically weak solution of~\eqref{eq:RanZakW1W2} on $[0,\sigma]$,
and let $(u_\sigma,v_\sigma)\in C([0,\tau], H^1 \times L^2)$
be an analytically weak solution of the refined Zakharov system
\eqref{eq:RanZakSigma} on $[0,\tau]$
with the initial condition
\begin{align*}
   (u_\sigma(0), v_\sigma(0))
   := (e^{W_1(\sigma)} u_1(\sigma), v_1(\sigma) + \cT_\sigma(W_2)).
\end{align*}
For every $t\in [0,\sigma+\tau]$,
we set
\begin{align*}
\label{eq:GluingSolutions}
   u(t):= \begin{cases}
   				u_1(t), \quad &\text{if } t \in [0,\sigma), \\
   				 e^{-W_1(\sigma)}u_\sigma(t-\sigma), &\text{if } t \in [\sigma, \sigma + \tau], 						\end{cases} \qquad
   v(t):= \begin{cases}
   				 v_1(t), \quad &\text{if } t \in [0,\sigma), \\
   				v_\sigma(t-\sigma) - e^{\imu (t-\sigma)|\na|}\cT_\sigma(W_2)), &\text{if } t \in [\sigma, \sigma + \tau].
   				\end{cases}
\end{align*}
Then, $(u,v) \in C([0,\sigma+\tau], H^1\times L^2)$
is an analytically weak solution of  \eqref{eq:RanZakW1W2} on the larger
interval $[0,\sigma+\tau]$.
\end{proposition}

\section{Functional framework} \label{Sec:MultilinEst}

In this section we develop the main functional framework
for the solvability of the energy-critical stochastic Zakharov system.
We first introduce the functional spaces
essentially consisting of lateral Strichartz spaces
and adapted spaces.
Then, the key estimates 
in these functional spaces are derived 
for the Schr\"odinger flows, Schr\"odinger-wave interacting nonlinearity and
lower order terms arising from noise.

\subsection{Function spaces} \label{Subsec:FunctFrame}
Our functional framework combines lateral Strichartz spaces
and adapted spaces.

\subsubsection{Lateral Strichartz spaces}
Let us first introduce the lateral Strichartz spaces
which are used to capture the local smoothing effect of the Schr{\"o}dinger flow. 
 	
Let  $\textbf{e} \in \Sp^{3}$ and $\mathcal{P}_{\vece}=\{\xi \in \R^{4} \, |\, \xi \cdot \textbf{e}=0 \}$
with the induced Euclidean measure.
Set
\begin{equation} \label{Lepq-def}
\| f \|_{L_{\textbf{e}}^{p,q}(I \times \R^4)}
: = \Bigl( \int_{\R} \Bigl( \int_{I \times \mathcal{P}_{\textbf{e}}} |f(t, r \textbf{e} + y)|^q \dd t \dd y \Bigr)^{\frac{p}{q}} \dd r \Bigr)^{\frac{1}{p}},
\end{equation}
where $p, q \in [1,\infty]$,
with the usual adaptions if $p = \infty$ or $q = \infty$.

Let $\phi \in C_0^\infty(\R)$
be a nonnegative and symmetric function
such that $\phi(r)=0$ if $|r|\leq \frac{1}{8}$ or $|r|>4$ and $\phi(r)=1$ if $\frac{1}{4} \leq |r| \leq 2$, and set $\phi_N(r)=\phi(r/N)$.
Then,
\begin{equation} \label{eq:DecompositionIdentity}
\prod_{j=1}^{4}(1-\phi_N(\xi_j))=0
\end{equation}
for all $\xi \in \R^4$ with $N/2 < |\xi| < 2 N$.
Set  $P_{N,\vece} :=\cF_{x}^{-1} \phi_N(\xi \cdot \vece) \cF_{x}$.
By \eqref{eq:DecompositionIdentity},
one has the decomposition
 \begin{equation}   \label{eq:DecompositionAngular}
		P_N f = \sum_{j=1}^{4} P_{N,\vece_j} \Big[\prod_{l=1}^{j-1} (1-P_{N,\vece_l})\Big]P_N f,
	\end{equation}
	where $\vece_1, \ldots, \vece_4$ is the standard basis of $\R^4$.

\subsubsection{Adapted spaces}
	The other component of our functional framework consists of the adapted function spaces from~\cite{CHN23}, 
which have been developed recently to prove the 
optimal local well-posedness theory for the deterministic Zakharov system in dimensions $d \geq 4$. 

\paragraph{\bf (i) Schr\"odinger component:} 
Let us first recall the definition of these spaces.
	
	For $s, a, b \in \R$ and $\lambda \in 2^\N$ we define
	\begin{align}
	\label{eq:DefSsablambda}
		\|u\|_{S^{s,a,b}_\lambda} := \lambda^s \|u\|_{L^\infty_t L^2_x} + \lambda^{s - 2a} \|(\lambda + |\partial_t|)^a u\|_{L^2_t L^4_x} + \lambda^{s-1 + b} \Big\| \Big( \frac{\lambda + |\partial_t|}{\lambda^2 + |\partial_t|}\Big)^a (\imu \partial_t + \Delta) u \Big\|_{L^2_{t,x}}
	\end{align}
	and
	\begin{align}
	\label{eq:DefNsablambda}
		\|F\|_{N^{s,a,b}_\lambda} := \lambda^{s-2} \|P_{\leq(\frac{\lambda}{2^8})^2}^{(t)} F \|_{L^\infty_t L^2_x} + \lambda^s \|C_{\leq (\frac{\lambda}{2^8})^2} F \|_{L^2_t L^{\frac{4}{3}}_x} + \lambda^{s-1+b} \Big\| \Big( \frac{\lambda + |\partial_t|}{\lambda^2 + |\partial_t|}\Big)^a F \Big\|_{L^2_{t,x}}.
	\end{align}
	The corresponding $S^{s,a,b}$- and $N^{s,a,b}$-norm are defined by the $l^2$-sum of the dyadic pieces 
$\|P_\lambda u\|_{S^{s,a,b}_\lambda}$ and $\|P_\lambda F\|_{N^{s,a,b}_\lambda}$, respectively. 

In the case $0 \leq a \leq 1$ an application of Bernstein's inequality yields
\begin{align}
\label{eq:CharSsablambda}
	\|u_\lambda\|_{S^{s,a,b}_\lambda} \sim \lambda^s(\|u_\lambda\|_{L^\infty_t L^2_x} + \|C_{\leq (\frac{\lambda}{2^8})^2} u_\lambda\|_{L^2_t L^4_x}) + \lambda^{s-1+b}\Big\| \Big( \frac{\lambda + |\partial_t|}{\lambda^2 + |\partial_t|}\Big)^a (\imu \partial_t + \Delta) u_\lambda \Big\|_{L^2_{t,x}},
\end{align}	
see Remark~2.1 in~\cite{CHN23}.

We remark that the $S^{s,a,b}$-norm is used to control the Schr{\"o}dinger component of the Zakharov system and the $N^{s,a,b}$-norm to control the Schr{\"o}dinger nonlinearity. 
The parameters $a$ and $b$ are introduced in order to obtain the local well-posedness 
of deterministic Zakharov systems in the optimal regularity region, see Theorem~1.1 in~\cite{CHN23}.

In this work we do not need the full flexibility of these function spaces. 
We mainly work with the energy regularity $(s,l) = (1,0)$ and the endpoint regularity $(s,l) = (\frac{1}{2},0)$, 
to which the corresponding parameters are 
$(a,b) = (\frac{1}{4},0)$ and $(a,b) = (0,0)$, respectively, see~(2.4) in~\cite{CHN23}. In particular, the parameter $b$ is always~$0$ in the regime we are working in so that we drop it from our notation.

Using characterization~\eqref{eq:CharSsablambda}, we define for the case $(s,l) = (1,0)$
\begin{align}
	\label{eq:DefS1lambda}
	\|u\|_{\SOne_\lambda} := \lambda \|u\|_{L^\infty_t L^2_x} + \lambda \|C_{\leq (\frac{\lambda}{2^8})^2} u\|_{L^2_t L^4_x} + \Big\| \Big( \frac{\lambda + |\partial_t|}{\lambda^2 + |\partial_t|}\Big)^{\frac{1}{4}} (\imu \partial_t + \Delta) u \Big\|_{L^2_{t,x}},
\end{align}
while for the case $(s,l) = (\frac{1}{2},0)$ 
\begin{align}
	\label{eq:DefS12lambda}
	\|u\|_{\SHalf_\lambda} := \lambda^{\frac{1}{2}} \|u\|_{L^\infty_t L^2_x} + \lambda^{\frac{1}{2}} \|u\|_{L^2_t L^4_x} + \lambda^{-\frac{1}{2}} \Big\| (\imu \partial_t + \Delta) u \Big\|_{L^2_{t,x}}
\end{align}
based on~\eqref{eq:DefSsablambda}. 

Moreover, since both at the energy regularity and the endpoint regularity 
we have $0 \leq a < \frac{1}{2}$, we get
the characterization
\begin{align*}
	\|F_\lambda\|_{N^{s,a,b}_\lambda} \sim \lambda^s \| C_{\leq (\frac{\lambda}{2^8})^2} F_\lambda \|_{L^2_t L^{\frac{4}{3}}_x} + \lambda^{s-1+b} \Big\| \Big( \frac{\lambda + |\partial_t|}{\lambda^2 + |\partial_t|}\Big)^a F_\lambda \Big\|_{L^2_{t,x}},
\end{align*}
which follows from an application of Bernstein's inequality and Sobolev's embedding, 
see Remark~2.2 in~\cite{CHN23}. 
We define 
\begin{align*}
	\|F\|_{\NOne_\lambda} &:= \lambda \| C_{\leq (\frac{\lambda}{2^8})^2} F \|_{L^2_t L^{\frac{4}{3}}_x} +  \Big\| \Big( \frac{\lambda + |\partial_t|}{\lambda^2 + |\partial_t|}\Big)^{\frac{1}{4}} F \Big\|_{L^2_{t,x}}, \\
	\|F\|_{\NHalf_\lambda} &:= \lambda^{\frac{1}{2}} \| C_{\leq (\frac{\lambda}{2^8})^2} F \|_{L^2_t L^{\frac{4}{3}}_x} + \lambda^{-\frac{1}{2}} 
\| F \|_{L^2_{t,x}}.
\end{align*}
	The corresponding $\Ssa$- and $\Nsa$-norms are defined by
	\begin{align*}
		\|u\|_{\Ssa} := \Big(\sum_{\lambda \in 2^{\N_0}} \|u_\lambda\|_{\Ssa_\lambda}^2 \Big)^{\frac{1}{2}}, \qquad \qquad
		\|F\|_{\Nsa} := \Big(\sum_{\lambda \in 2^{\N_0}} \|F_\lambda\|_{\Nsa_\lambda}^2 \Big)^{\frac{1}{2}}
	\end{align*}
	for $(s,a) \in \{(\frac{1}{2},0),(1,\frac{1}{4})\}$. Finally, we set
    \begin{align*}
        \Ssa(\R) := \{u \in C(\R, H^s(\R^4)) \colon \|u\|_{\Ssa} < \infty\},
    \end{align*}
    while $\Nsa(\R)$ is the space of tempered distributions with finite $\| \cdot \|_{\Nsa}$-norm.

\begin{remark}  \label{rem:NormComp}
  $(i)$ To summarize, we note that
		\begin{align*}
			\|u_\lambda\|_{\SOne_\lambda} \sim \|u_\lambda\|_{S_\lambda^{1,\frac{1}{4},0}}, \qquad \|F_\lambda\|_{\NOne_\lambda} \sim \|F_\lambda\|_{N_\lambda^{1,\frac{1}{4},0}}, \qquad
			\|u_\lambda\|_{\SHalf_\lambda} \sim \|u_\lambda\|_{S_\lambda^{\frac{1}{2},0,0}}, \qquad \|F_\lambda\|_{\NHalf_\lambda} \sim \|F_\lambda\|_{N_\lambda^{\frac{1}{2},0,0}}
		\end{align*}
by Remark~2.1 and Remark~2.2 in~\cite{CHN23}.
		
  $(ii)$ 
We also observe that, because of $u_\lambda = C_{\leq (\frac{\lambda}{2^8})^2} u_\lambda + C_{> (\frac{\lambda}{2^8})^2} u_\lambda$ 
and an application of Bernstein's inequality, 
				\begin{align*}
					\|u_\lambda\|_{L^2_t L^4_x} &\lesssim \| C_{\leq (\frac{\lambda}{2^8})^2} u_\lambda\|_{L^2_t L^4_x} + \lambda \| C_{> (\frac{\lambda}{2^8})^2} u_\lambda \|_{L^2_{t,x}}  \\ 
					&\lesssim  \| C_{\leq (\frac{\lambda}{2^8})^2} u_\lambda\|_{L^2_t L^4_x} + \lambda^{-1} \| C_{> (\frac{\lambda}{2^8})^2} (\imu \partial_t + \Delta) u_\lambda \|_{L^2_{t,x}} \\
					&\lesssim  \| C_{\leq (\frac{\lambda}{2^8})^2} u_\lambda\|_{L^2_t L^4_x} + \lambda^{-1 + a} \Big\| \Big(\frac{\lambda + |\partial_t|}{\lambda^2 + |\partial_t|} \Big)^{a} (\imu \partial_t + \Delta) u_\lambda \Big\|_{L^2_{t,x}},
				\end{align*}					
which shows that
				\begin{align*}
					\|u\|_{L^2_t W^{\frac{1}{2},4}_x} \lesssim \|u\|_{\SHalf} \lesssim \|u\|_{\SOne}.
				\end{align*} 
	\end{remark}

\paragraph{\bf $(ii)$ Wave component} 
Concerning the wave component, we use the same norm as in the deterministic setting, where
\begin{align*}
	& \|v\|_{W^{l,\alpha,\beta}_\lambda} = \lambda^l \|v\|_{L^\infty_t L^2_x}
   + \lambda^{l - \alpha} \|(\lambda + |\partial_t|)^{\alpha} \TempN v\|_{L^\infty_t L^2_x}
    + \lambda^{\beta - 1} \|(\imu \partial_t + |\nabla|) v\|_{L^2_{t,x}}
\end{align*}
was introduced with the choice $\alpha = a$ and $\beta = s - \frac{1}{2}$, see~\cite[Section~2.2]{CHN23}. Recalling that we only work at the regularity $l = 0$ and that $a = \frac{1}{4}$ in the case $(s,l) = (1,0)$ and $a = 0$ in the case $(s,l) = (\frac{1}{2},0)$, we will use
\begin{align*}
	\|v\|_{\WEner_\lambda} &:= \|v\|_{L^\infty_t L^2_x} + \lambda^{-\frac{1}{4}} \|(\lambda + |\partial_t|)^{\frac{1}{4}} \TempN v\|_{L^\infty_t L^2_x} + \lambda^{-\frac{1}{2}} \|(\imu \partial_t + |\nabla|) v\|_{L^2_{t,x}}, \\
	\|v\|_{\WEndp_\lambda} &:= \|v\|_{L^\infty_t L^2_x} + \|\TempN v\|_{L^\infty_t L^2_x} + \lambda^{-1} \|(\imu \partial_t + |\nabla|) v\|_{L^2_{t,x}}.
\end{align*}
We define
\begin{align*}
    W^{0,\alpha,\beta}(\R) := \{v \in C(\R, L^2(\R^4)) \colon \| v \|_{W^{0,\alpha, \beta}} < \infty\}
\end{align*}
for $(\alpha,\beta) \in \{(0,0),(\frac{1}{4}, \frac{1}{2})\}$.

\subsubsection{Setup for the stochastic Zakharov system}  
The Schr{\"o}dinger component of the Zakharov system will be controlled 
in both the adapted and lateral Strichartz spaces. 
We define
\begin{align*}
	& \|u\|_{\Xs_\lambda} := \|u\|_{\Ssa_\lambda} + \sum_{j = 1}^4 \lambda^{s + \frac{1}{2}} \|P_{\lambda, \vece_j} \ModN u \|_{L^{\infty,2}_{\vece_j}} \quad \text{if } \lambda > 1, \qquad \|u\|_{\Xs_\lambda} := \|u\|_{\Ssa_\lambda} \quad \text{if } \lambda = 1,
\end{align*}
and
\begin{align*}
     \|u\|_{\Xs} := \Big( \sum_{\lambda \in 2^{\N_0}} \|u_\lambda\|_{\Xs_\lambda}^2\Big)^{\frac{1}{2}}
\end{align*}
for $s \in \{\frac{1}{2},1\}$, where $a = \frac{1}{4}$ in the case $s = 1$ and $a = 0$ in the case $s = \frac{1}{2}$.
We distinguish between high and low frequencies in the definition of $\| \cdot \|_{\Xs_\lambda}$ since inhomogeneous function spaces are used here, and hence the local smoothing estimate is only available for high frequencies, 
which is sufficient to control the problematic derivative terms caused by thenoise.

The nonlinearity in the Schr{\"o}dinger equation will be controlled via
\begin{align*}
	\|F\|_{\Gs} := \Big(\| P_1 F\|_{\Nsa_1}^2 + \inf_{F = F_1 + F_2} \Big(\sum_{\lambda \in 2^\N} \|P_\lambda F_1\|_{\Nsa_\lambda}^2 + \sum_{j = 1}^4 \sum_{\lambda \in 2^{\N}} \lambda^{2s - 1} \|P_\lambda F_2\|_{L^{1,2}_{\vece_j}}^2 \Big)\Big)^{\frac{1}{2}},
\end{align*}
where we again choose $a = \frac{1}{4}$ in the case $s = 1$ and $a = 0$ in the case $s = \frac{1}{2}$. 

The wave component is controlled by 
\begin{align*}
	\|v\|_{\Y} = \|v\|_{\WEner} :=  \Big( \sum_{\lambda \in 2^\N} \|v_\lambda\|_{\WEner_\lambda}^2 \Big)^{\frac{1}{2}} \qquad \text{and} \qquad
     \|v\|_{\WEndp} := \Big( \sum_{\lambda \in 2^\N} \|v_\lambda\|_{\WEndp_\lambda}^2 \Big)^{\frac{1}{2}}
\end{align*} 
at the energy regularity $(s,l) = (1,0)$ and at the endpoint regularity $(s,l) = (\frac{1}{2},0)$, respectively.

Again, we define the function spaces
\begin{align*}
    \Xs(\R) := \{ u \in C(\R, H^s(\R^4)) \colon \| u \|_{\Xs} < \infty\}, \qquad
    \Y(\R) := \{ v \in C(\R, L^2(\R^4)) \colon \| v \|_{\Y} < \infty\},
\end{align*}
and $\Gs(\R)$ as the set of tempered distributions with finite $\| \cdot \|_{\Gs}$-norm.

Finally, we localize the norms and spaces above to intervals $I \subseteq \R$ via restriction. For example, we set
\begin{equation}
	\label{eq:DefRestrictionNorm}
	\|u\|_{\Xs(I)} = \inf_{u' \in \Xs(\R), u'_{|I} = u} \|u'\|_{\Xs(\R)}.
\end{equation}

\subsection{Control of linear Schr\"odinger and wave flows}
Lemma \ref{lem:StrichartzLocalSmooth} collects
Strichartz and local smoothing estimates for the linear Schr\"odinger flow.

\begin{lemma} [Strichartz and local smoothing estimates]  \label{lem:StrichartzLocalSmooth}
	Let $\lambda, \mu \in 2^{\N_0}$ with $|\log_2 (\mu/\lambda)| \leq 4$, $\vece \in \Sp^{d-1}$, and $(q,p)$, $(\tilde{q}, \tilde{p})$ be Schr\"odinger admissible. We then have the following estimates.
	\begin{enumerate}
		\item \label{it:HomStrichartz} Homogeneous Strichartz estimate:
				\begin{align*}
					\|e^{\imu t \Delta}  f_\lambda\|_{L^q_t L^p_x} \lesssim \|f_\lambda\|_{L^2_x}.
				\end{align*}
		\item \label{it:HomLocalSmooth} Homogeneous local smoothing estimate:
				\begin{align*}
					\|e^{\imu t \Delta} P_{\mu,\vece} f\|_{L^{\infty,2}_{\vece}} \lesssim \mu^{-\frac{1}{2}} \|f\|_{L^2_x}, \qquad \mu > 1.
				\end{align*}
		\item \label{it:InhomStrichartz} Inhomogeneous Strichartz estimate:
				\begin{align*}
					\Big\| \int_{s < t} e^{\imu (t-s) \Delta} g_\lambda(s) \dd s \Big\|_{L^q_t L^p_x} \lesssim \|g_\lambda\|_{L^{\tilde{q}'}_t L^{\tilde{p}'}_x},
				\end{align*}
      where $\tilde{p}'$ and $\tilde{q}'$ are the conjugate numbers of $\tilde{p}$ and $\tilde{q}$, respectively.
      That is, $1/\tilde{p}' + 1/\tilde{p}=1$ and $1/\tilde{q}' + 1/\tilde{q}=1$.

		\item \label{it:InhomLocalSmooth} Inhomogeneous local smoothing estimate:
				\begin{align*}
					\Big\| \int_{s < t} e^{\imu (t-s) \Delta} P_\lambda P_{\mu,\vece} g(s) \dd s \Big\|_{L^{\infty,2}_\vece} \lesssim \lambda^{-1} \|g\|_{L^{1,2}_\vece}, \qquad \mu,\lambda > 1.
				\end{align*}
		\item \label{it:InhomStrichartzLocalSmooth} Inhomogeneous Strichartz to local smoothing estimate:
				\begin{align*}
					\Big\| \int_{s < t} e^{\imu (t-s) \Delta} P_\lambda P_{\mu,\vece} g(s) \dd s \Big\|_{L^q_t L^p_x} \lesssim \lambda^{-\frac{1}{2}} \|g\|_{L^{1,2}_\vece}, \qquad \mu,\lambda > 1.
				\end{align*}
		\item \label{it:InhomLocalSmoothStrichartz} Inhomogeneous local smoothing to Strichartz estimate:
				\begin{align*}
					\Big\| \int_{s < t} e^{\imu (t-s) \Delta} P_\lambda P_{\mu,\vece} g(s) \dd s \Big\| _{L^{\infty,2}_\vece} \lesssim \lambda^{-\frac{1}{2}} \|g\|_{L^{\tilde{q}'}_t L^{\tilde{p}'}_x}, \qquad \mu,\lambda > 1.
				\end{align*}
	\end{enumerate}
\end{lemma}	

	\begin{proof}
		Estimates~\ref{it:HomStrichartz} and~\ref{it:InhomStrichartz} are the well-known Strichartz estimates, see~\cite{KT98}. The local smoothing estimates~\ref{it:HomLocalSmooth} and~\ref{it:InhomLocalSmooth} are contained in Proposition~3.8 in~\cite{BIKT11} as one sees by checking the definition of the involved norms there. Although estimate~\ref{it:InhomStrichartzLocalSmooth} is contained in Proposition~3.8 in~\cite{BIKT11} only for one particular Schr\"odinger admissible pair, an inspection of the proof of that proposition reveals that one can take any Schr\"odinger admissible pair on the left hand side. In fact, a verbatim copy of the proof of Lemma~7.4 in~\cite{BIKT11} yields~\ref{it:InhomStrichartzLocalSmooth}.
		
		The remaining estimate~\ref{it:InhomLocalSmoothStrichartz} follows from~\ref{it:InhomStrichartzLocalSmooth} by duality. We first note that~\ref{it:InhomStrichartzLocalSmooth} also holds if we integrate over $s > t$. Exploiting this estimate, we then obtain
		\begin{align*}
			\Big\| \int_{s < t} e^{\imu (t-s) \Delta} P_\lambda P_{\mu,\vece} g(s) \dd s \Big\| _{L^{\infty,2}_\vece}
			&= \sup_{\|h\|_{L^{1,2}_\vece} \leq 1} \Big| \int_\R \int_{\R^d} \int_{s < t} e^{\imu (t-s) \Delta} P_\lambda P_{\mu,\vece} g(s) \dd s\, \overline{h(t)} \dd x \dd t \Big| \\
			&= \sup_{\|h\|_{L^{1,2}_\vece} \leq 1} \Big| \int_\R  \int_{s < t} \int_{\R^d}  g(s)  \overline{e^{\imu (s-t) \Delta} P_\lambda P_{\mu,\vece} h(t)} \dd x \dd s \dd t \Big| \\
			&= \sup_{\|h\|_{L^{1,2}_\vece} \leq 1} \Big| \int_\R   \int_{\R^d}  g(s)  \overline{\int_{t > s}e^{\imu (s-t) \Delta} P_\lambda P_{\mu,\vece} h(t) \dd t} \dd x \dd s \Big| \\
			&\leq \|g\|_{L^{\tilde{q}'}_t L^{\tilde{p}'}_x} \sup_{\|h\|_{L^{1,2}_\vece} \leq 1} \Big\|\int_{t > s}e^{\imu (s-t) \Delta} P_\lambda P_{\mu,\vece} h(t) \dd t\Big\|_{L^{\tilde{q}}_s L^{\tilde{p}}_x} \\
			&\lesssim \|g\|_{L^{\tilde{q}'}_t L^{\tilde{p}'}_x} \sup_{\|h\|_{L^{1,2}_\vece} \leq 1} \lambda^{-\frac{1}{2}} \|h\|_{L^{1,2}_\vece} \lesssim \lambda^{-\frac{1}{2}}  \|g\|_{L^{\tilde{q}'}_t L^{\tilde{p}'}_x}. \qedhere
		\end{align*}
	\end{proof}

The control of the linear Schr\"odinger flow in adapted spaces has been worked out in~\cite{CHN23}. We collect them in the next lemma.
\begin{lemma} [\cite{CHN23}] 	\label{lem:LinFlowAdaptedSpaces}
	Let $(s,a) \in \{(\frac{1}{2},0),(1,\frac{1}{4})\}$. For any $\lambda \in 2^{\N}$ we have
	\begin{equation*}
		\|e^{\imu t \Delta} f_\lambda\|_{\Ssa_\lambda} \lesssim \lambda^s \|f_\lambda\|_{L^2_x}, \qquad
		\Big\|\int_{t_0}^t e^{\imu (t-t') \Delta} g_\lambda(t') \dd t' \Big\|_{\Ssa_\lambda} \lesssim \|g_\lambda\|_{\Nsa_\lambda}.
	\end{equation*}
\end{lemma}

\begin{proof}
 These estimates follow from Lemma~2.4 in~\cite{CHN23} and Remark~\ref{rem:NormComp}.
\end{proof}

Next, we show the compatibility between
lateral Strichartz spaces and adapted spaces. 
That is, we show that the linear Schr{\"o}dinger flow is controlled in the new $\Xs$-space by the inital datum in $H^s$ and the inhomogeneity in $\Gs$.

\begin{lemma}   [Control of linear Schr\"odinger flows in $\Xs$-spaces]
	\label{lem:LinEstimates}
	Let $(s,a) \in \{(\frac{1}{2},0),(1,\frac{1}{4})\}$, $f \in H^s(\R^4)$, $g \in \Gs$, and $u$ solve the linear Schr{\"o}dinger equation
	\begin{align*}
		\imu \partial_t u + \Delta u = g, \qquad u(t_0) = f.
	\end{align*}
	Then
	\begin{align*}
		\|u\|_{\Xs} \lesssim \|f\|_{H^s} +  \|g \|_{\Gs}.
	\end{align*}
\end{lemma}

\begin{proof}
	From Lemmas~\ref{lem:StrichartzLocalSmooth} and \ref{lem:LinFlowAdaptedSpaces} we have 
	\begin{align*}
		\|e^{\imu t \Delta} f_\lambda\|_{\Ssa_\lambda} \lesssim \lambda^s \|f_\lambda\|_{L^2_x}, \qquad
		\lambda^{s + \frac{1}{2}} \| P_{\lambda, \vece} \ModN e^{\imu t \Delta}f_\lambda \|_{L^{\infty,2}_{\vece}} \lesssim \lambda^s \|f_\lambda\|_{L^2_x}
	\end{align*}
	for any $\vece \in \Sp^{3}$, which immediately implies
	\begin{align*}
		\|e^{\imu t \Delta} f\|_{\Xs} \lesssim \|f\|_{H^s}.
	\end{align*}

	For the remaining estimate
    \begin{equation*}
        \Big\|\int_{t_0}^t e^{\imu (t-t') \Delta} g(t') \dd t' \Big\|_{\X^s} \lesssim \|g\|_{\G^s},
    \end{equation*}
it is sufficient to show
	\begin{align}
		\Big\|\int_{t_0}^t e^{\imu (t-t') \Delta} g_\lambda(t') \dd t' \Big\|_{\Xs_\lambda} &\lesssim \sum_{j = 1}^4 \lambda^{s-\frac{1}{2}} \| g_\lambda\|_{L^{1,2}_{\vece_j}}, \label{eq:DuhamelLS} \\ 
  \Big\|\int_{t_0}^t e^{\imu (t-t') \Delta} g_\lambda(t') \dd t' \Big\|_{\Xs_\lambda} &\lesssim \|g_\lambda\|_{\Nsa_\lambda}  \label{eq:DuhamelNlambda}  
	\end{align}
 for all $\lambda > 1$ and~\eqref{eq:DuhamelNlambda} for $\lambda = 1$. The latter directly follows from Lemma~\ref{lem:LinFlowAdaptedSpaces} so that we only consider the case $\lambda > 1$ in the following.

    Let us start with the proof of~\eqref{eq:DuhamelLS}. We first consider the $\Ssa_\lambda$-component of the $\Xs_\lambda$-norm. For the Strichartz components we use Lemma~\ref{lem:StrichartzLocalSmooth}~\ref{it:InhomStrichartzLocalSmooth} and the decomposition~\eqref{eq:DecompositionAngular} to infer
	\begin{align}
		&\lambda^s \Big\|\int_{t_0}^t e^{\imu (t-t') \Delta} g_\lambda(t') \dd t' \Big\|_{L^\infty_t L^2_x} + \lambda^s \Big\|\ModN \int_{t_0}^t e^{\imu (t-t') \Delta} g_\lambda(t') \dd t' \Big\|_{L^2_t L^4_x} \nonumber\\
		&\lesssim \sum_{j = 1}^4 \Big(\lambda^s \Big\|\int_{t_0}^t e^{\imu (t-t') \Delta} P_{\lambda, \vece_j} g_\lambda(t') \dd t' \Big\|_{L^\infty_t L^2_x} + \lambda^s \Big\| \int_{t_0}^t e^{\imu (t-t') \Delta} P_{\lambda, \vece_j} g_\lambda(t') \dd t' \Big\|_{L^2_t L^4_x}\Big)  \nonumber  \\
        & \lesssim \sum_{j = 1}^4 \lambda^{s-\frac{1}{2}} \| g_\lambda\|_{L^{1,2}_{\vece_j}}. \label{eq:DuhamelSlLSl1}
	\end{align}
	For the remaining component of the $\Ssa_\lambda$-norm, we get by Bernstein's inequality
	\begin{align}
		\Big\|\Big(\frac{\lambda + |\partial_t|}{\lambda^2 + |\partial_t|} \Big)^{\frac{1}{4}} (\imu \partial_t + \Delta)\int_{t_0}^t e^{\imu (t-t') \Delta} g_\lambda(t') \dd t' \Big\|_{L^2_{t,x}}
		\lesssim \| g_\lambda\|_{L^2_{t,x}} \lesssim  \lambda^{\frac{1}{2}} \| g_\lambda\|_{L^{1,2}_{\vece_1}}  \label{eq:DuhamelSlLSl2s1}
	\end{align}
	in the case $(s,a) = (1,\frac{1}{4})$ and
	\begin{align}
		\lambda^{-\frac{1}{2}} \Big\|(\imu \partial_t + \Delta)\int_{t_0}^t e^{\imu (t-t') \Delta} g_\lambda(t') \dd t' \Big\|_{L^2_{t,x}}
		\lesssim \lambda^{-\frac{1}{2}} \| g_\lambda\|_{L^2_{t,x}} \lesssim  \| g_\lambda\|_{L^{1,2}_{\vece_1}} \label{eq:DuhamelSlLSl2s12}
	\end{align}
	in the case $(s,a) = (\frac{1}{2},0)$.
    To estimate the lateral Strichartz component, we apply Lemma~\ref{lem:StrichartzLocalSmooth}~\ref{it:InhomLocalSmooth} to derive
	\begin{align}
		\sum_{j = 1}^4 \lambda^{s + \frac{1}{2}} \Big\|P_{\lambda, \vece_j} \ModN \int_{t_0}^t e^{\imu (t-t') \Delta} g_\lambda(t') \dd t' \Big\|_{L^{\infty,2}_{\vece_j}}
		&\lesssim \sum_{j = 1}^4 \lambda^{s + \frac{1}{2}} \Big\|
          \int_{t_0}^t e^{\imu (t-t') \Delta} P_{\lambda, \vece_j}g_\lambda(t') \dd t' \Big\|_{L^{\infty,2}_{\vece_j}} \notag \\
        & \lesssim \sum_{j = 1}^4 \lambda^{s - \frac{1}{2}} \|g_\lambda\|_{L^{1,2}_{\vece_j}}. \label{eq:DuhamelLSlLSl}
	\end{align}
	The combination of~\eqref{eq:DuhamelSlLSl1} to~\eqref{eq:DuhamelLSlLSl} yields~\eqref{eq:DuhamelLS}.
 
	To prove~\eqref{eq:DuhamelNlambda}, we first note that
	\begin{align}
		\Big\|\int_{t_0}^t e^{\imu (t-t') \Delta} P_\lambda g(t') \dd t' \Big\|_{\Ssa_\lambda} \lesssim \|g_\lambda\|_{\Nsa_\lambda} \label{eq:DuhamelSlNl}
	\end{align}
	by Lemma~\ref{lem:LinFlowAdaptedSpaces}.
	For the remaining lateral Strichartz component of the $\Xs_\lambda$-norm, we have to show
	\begin{align}
    \label{eq:DuhamelLSNs}
		\lambda^{s + \frac{1}{2}} \Big\| P_{\lambda, \vece_j} C_{\leq (\frac{\lambda}{2^8})^2} \int_{t_0}^t e^{\imu (t-t') \Delta}   g_\lambda(t') \dd t' \Big\|_{L^{\infty,2}_{\vece_j}}
		\lesssim   \|g_\lambda\|_{\Nsa_\lambda}.
	\end{align}
    We recall that $\cI_0$ denotes the Duhamel integral for the Schr{\"o}dinger group. Splitting $g_\lambda$ in its low and high modulation part, we first get
    \begin{align}
        \lambda^{s + \frac{1}{2}} \| P_{\lambda, \vece_j} C_{\leq (\frac{\lambda}{2^8})^2} \cI_0[g_\lambda]\|_{L^{\infty,2}_{\vece_j}}
        &\leq \lambda^{s + \frac{1}{2}} \| P_{\lambda, \vece_j} C_{\leq (\frac{\lambda}{2^8})^2} \cI_0[C_{\leq (\frac{\lambda}{2^8})^2} g_\lambda] \|_{L^{\infty,2}_{\vece_j}} \nonumber\\
        &\qquad + \lambda^{s + \frac{1}{2}} \| P_{\lambda, \vece_j} C_{\leq (\frac{\lambda}{2^8})^2} \cI_0[C_{> (\frac{\lambda}{2^8})^2}  g_\lambda] \|_{L^{\infty,2}_{\vece_j}}. \label{eq:DuhamelLSNsSplit}
    \end{align}
    For the first term on the above right-hand side, we apply Lemma~\ref{lem:StrichartzLocalSmooth}~\ref{it:InhomLocalSmoothStrichartz} to derive
    \begin{align}
        \label{eq:DuhamelLSNsLowMod}
        \lambda^{s + \frac{1}{2}} \| P_{\lambda, \vece_j} C_{\leq (\frac{\lambda}{2^8})^2} \cI_0[ C_{\leq (\frac{\lambda}{2^8})^2} g_\lambda ]\|_{L^{\infty,2}_{\vece_j}}
        &\lesssim \lambda^{s + \frac{1}{2}} \| \cI_0[P_{\lambda, \vece_j} C_{\leq (\frac{\lambda}{2^8})^2} g_\lambda]\|_{L^{\infty,2}_{\vece_j}} \nonumber \\ 
        &\lesssim \lambda^s \| C_{\leq (\frac{\lambda}{2^8})^2} g_\lambda \|_{L^2_t L^{\frac{4}{3}}_x} \lesssim \| g_\lambda \|_{\Nsa_\lambda}.
    \end{align}
    To estimate the second term in~\eqref{eq:DuhamelLSNsSplit}, we first claim that
    \begin{align}
        \lambda^{s + \frac{1}{2}}\|P_{\lambda, \vece_j} C_{\leq (\frac{\lambda}{2^8})^2} \cI_0[ C_{> (\frac{\lambda}{2^8})^2} g_\lambda ]\|_{L^{\infty,2}_{\vece_j}} 
        &\lesssim \lambda^{s-2} \| C_{> (\frac{\lambda}{2^8})^2} g_\lambda \|_{L^\infty_t L^2_x}, \label{eq:EstDuhamelLSNsCl1} \\
        \lambda^{s + \frac{1}{2}}\|P_{\lambda, \vece_j} C_{\leq (\frac{\lambda}{2^8})^2} \cI_0[ C_{\sim \nu} C_{> (\frac{\lambda}{2^8})^2} g_\lambda ]\|_{L^{\infty,2}_{\vece_j}} 
        &\lesssim \lambda^{s} \nu^{-1} \| C_{> (\frac{\lambda}{2^8})^2} g_\lambda \|_{L^\infty_t L^2_x} \label{eq:EstDuhamelLSNsCl2}
    \end{align}
    for any $\nu > (\frac{\lambda}{2^8})^2$. Assuming these two estimates for the moment, we further split
    \begin{align*}
        C_{> (\frac{\lambda}{2^8})^2} g_\lambda = P^{(t)}_{\leq (\frac{\lambda}{2^8})^2} C_{> (\frac{\lambda}{2^8})^2} g_\lambda + P^{(t)}_{> (\frac{\lambda}{2^8})^2} C_{> (\frac{\lambda}{2^8})^2} g_\lambda.
    \end{align*}
    Employing~\eqref{eq:EstDuhamelLSNsCl1}, we estimate for the first summand
    \begin{align}
        \label{eq:EstDuhamelLSNsHighModLowTemp}
        \lambda^{s + \frac{1}{2}} \| P_{\lambda, \vece_j} C_{\leq (\frac{\lambda}{2^8})^2} \cI_0[P^{(t)}_{\leq (\frac{\lambda}{2^8})^2} C_{> (\frac{\lambda}{2^8})^2} g_\lambda ]\|_{L^{\infty,2}_{\vece_j}}
        &\lesssim \lambda^{s-2}\| C_{> (\frac{\lambda}{2^8})^2} P^{(t)}_{\leq (\frac{\lambda}{2^8})^2} g_\lambda\|_{L^\infty_t L^2_x} \nonumber\\
        &\lesssim \lambda^{s-2}\| P^{(t)}_{\leq (\frac{\lambda}{2^8})^2}  g_\lambda\|_{L^\infty_t L^2_x} \lesssim \| g_\lambda \|_{\Nsa_\lambda},
    \end{align}
    where we combined~\eqref{eq:DefNsablambda} and Remark~\ref{rem:NormComp} in the last step. 
    For the high temporal frequencies, we exploit~\eqref{eq:EstDuhamelLSNsCl2} and Bernstein's inequality to infer
    \begin{align}
        &\lambda^{s + \frac{1}{2}} \| P_{\lambda, \vece_j} C_{\leq (\frac{\lambda}{2^8})^2} \cI_0[P^{(t)}_{> (\frac{\lambda}{2^8})^2} C_{> (\frac{\lambda}{2^8})^2} g_\lambda ]\|_{L^{\infty,2}_{\vece_j}}
        \lesssim \lambda^{s + \frac{1}{2}} \sum_{\nu > (\frac{\lambda}{2^8})^2} \| P_{\lambda, \vece_j} C_{\leq (\frac{\lambda}{2^8})^2} \cI_0[P^{(t)}_{\nu} C_{\sim \nu} C_{> (\frac{\lambda}{2^8})^2} g_\lambda ]\|_{L^{\infty,2}_{\vece_j}} \nonumber\\
        &\lesssim \lambda^s \sum_{\nu > (\frac{\lambda}{2^8})^2} \nu^{-1} \| C_{> (\frac{\lambda}{2^8})^2} P^{(t)}_{\nu} g_\lambda \|_{L^\infty_t L^2_x} \lesssim \lambda^s \sum_{\nu > (\frac{\lambda}{2^8})^2} \nu^{-\frac{1}{2}} \| C_{> (\frac{\lambda}{2^8})^2} P^{(t)}_{\nu} g_\lambda \|_{L^2_t L^2_x} \nonumber \\
        &\lesssim \lambda^s \sum_{\nu > (\frac{\lambda}{2^8})^2} \nu^{-\frac{1}{2}} \Big\| \Big( \frac{\lambda + |\partial_t |}{\lambda^2 + |\partial_t|}\Big)^a g_\lambda \Big\|_{L^2_t L^2_x} \lesssim  \lambda^{s-1}  \Big\| \Big( \frac{\lambda + |\partial_t |}{\lambda^2 + |\partial_t|}\Big)^a g_\lambda \Big\|_{L^2_t L^2_x}
        \lesssim \| g_\lambda \|_{\Nsa}. \label{eq:EstDuhamelLSNsHighModHighTemp}
    \end{align}
    The combination of~\eqref{eq:EstDuhamelLSNsHighModLowTemp} and~\eqref{eq:EstDuhamelLSNsHighModHighTemp} with~\eqref{eq:DuhamelLSNsLowMod} yields~\eqref{eq:DuhamelLSNs}.

    It only remains to prove~\eqref{eq:EstDuhamelLSNsCl1} and~\eqref{eq:EstDuhamelLSNsCl2}. We first show~\eqref{eq:EstDuhamelLSNsCl1}. To that purpose, we recall the commutation relation
    \begin{align*}
        e^{\imu t \Delta} C_{> (\frac{\lambda}{2^8})^2} = P^{(t)}_{> (\frac{\lambda}{2^8})^2} e^{\imu t \Delta}
    \end{align*}
    and correspondingly for $C_{\leq  (\frac{\lambda}{2^8})^2}$. Hence, we can write
    \begin{align}
        P_{\lambda, \vece_j} C_{\leq (\frac{\lambda}{2^8})^2} \int_{t_0}^t e^{\imu (t-t') \Delta} C_{> (\frac{\lambda}{2^8})^2} g_\lambda(t') \dd t' 
        &= P_{\lambda, \vece_j} e^{\imu t \Delta} P^{(t)}_{\leq (\frac{\lambda}{2^8})^2} \int_{t_0}^t \partial_t \partial_t^{-1} P^{(t')}_{> (\frac{\lambda}{2^8})^2}(e^{-\imu t' \Delta}  g_\lambda(t')) \dd t' \nonumber \\
        &= P_{\lambda, \vece_j} e^{\imu t \Delta} P^{(t)}_{\leq (\frac{\lambda}{2^8})^2} (H(t) - H(t_0)), \label{eq:DuhamelLowHighMod}
    \end{align}
    where we set
    \begin{align*}
        H(t) = \partial_t^{-1} P^{(t)}_{> (\frac{\lambda}{2^8})^2}(e^{-\imu (\cdot) \Delta} g_\lambda)(t).
    \end{align*}
    Since 
    \begin{align*}
        P^{(t)}_{\leq (\frac{\lambda}{2^8})^2} H(t) = \partial_t^{-1} P^{(t)}_{\leq (\frac{\lambda}{2^8})^2} P^{(t)}_{> (\frac{\lambda}{2^8})^2}(e^{-\imu (\cdot) \Delta} g_\lambda)(t) = 0,
    \end{align*}
    we infer that~\eqref{eq:DuhamelLowHighMod} reduces to the free Schr{\"o}dinger solution $-P_{\lambda, \vece_j} e^{\imu t \Delta} H(t_0)$. Applying the local smoothing estimate from Lemma~\ref{lem:StrichartzLocalSmooth}~\ref{it:HomLocalSmooth}, we thus obtain
    \begin{align}
    \label{eq:EstLSLowHighMod}
       \lambda^{s + \frac{1}{2}} \| P_{\lambda, \vece_j} C_{\leq (\frac{\lambda}{2^8})^2} \cI_0[ C_{> (\frac{\lambda}{2^8})^2} g_\lambda] \|_{L^{\infty,2}_{\vece_j}}
       \lesssim \lambda^s \| H(t_0) \|_{L^2_x} \lesssim \lambda^s\| H \|_{L^\infty_t L^2_x}.
    \end{align}
    A computation yields the bound
    \begin{align*}
        \| H \|_{L^\infty_t L^2_x} \lesssim \lambda^{-2} \| e^{-\imu t \Delta} C_{> (\frac{\lambda}{2^8})^2}g_\lambda\|_{L^\infty_t L^2_x} \lesssim \lambda^{-2} \| C_{> (\frac{\lambda}{2^8})^2} g_\lambda \|_{L^\infty_t L^2_x}.
    \end{align*}
    In combination with~\eqref{eq:EstLSLowHighMod}, this estimate finally yields~\eqref{eq:EstDuhamelLSNsCl1}. The estimate~\eqref{eq:EstDuhamelLSNsCl2} follows along the same lines.
\end{proof}

The definition of the $\Y$-norm and~\cite[Lemma~2.6]{CHN23} also yield the following bound for the linear half-wave flow in the $\Y$-space.

\begin{lemma}   [Control of linear wave flows in $\Y$-space]
	\label{lem:LinearEstimateHalfWave}
	Let $g \in L^2(\R^4)$. Then one has
	\begin{align*}
		\|e^{\imu t |\nabla|} g\|_{\Y} \lesssim \|g\|_{L^2}.
	\end{align*}
\end{lemma}

\subsection{Control of nonlinearity and noise terms}  \label{Subsec-Esti-Nonl}

We next provide  bilinear estimates for the nonlinearities.
\begin{lemma} [Bilinear estimates] 	\label{lem:BilinearEstimates}
	\begin{enumerate}
		\item[]
		\item \label{it:BilinearEstNonendpoint} (Energy regularity) There exist a parameter $\theta \in (0,1)$ and constant $C > 0$ such that for any interval $I \subseteq \R$ we have
	\begin{align}
		\|\Re(v) u\|_{\NOne(I)}
       &\leq C \|v\|_{\Y(I) + L^2_t W^{1,4}_x(I \times \R^4)} \|u\|_{\SOne(I)}, \label{eq:Bilinvu}\\
		\|\cJ_0[|\nabla| ({u} w)]\|_{\Y(I)}
        &\leq C (\|u\|_{\SOne(I)} \|w\|_{\SOne(I)})^{1-\theta} (\|u\|_{L^2_t L^{4}_x(I \times \R^4)} \|w\|_{L^2_t L^{4}_x(I \times \R^4)})^{\theta}. \label{eq:Bilinnablauu}
	\end{align}
	\item \label{it:BilinearEstEndpoint}(Endpoint regularity) There exists a constant $C > 0$ such that for any interval $I \subseteq \R$ we have
	\begin{align}
		\|\Re(v) u\|_{\NHalf(I)} &\leq C \| v \|_{W^{0,0,0}(I)} \|u\|_{D(I)}^{\frac{1}{2}} \|u\|_{\SHalf(I)}^{\frac{1}{2}}, \label{eq:BilinvuEndpoint} \\
		\|\cJ_0[|\nabla| (u w)]\|_{W^{0,0,0}} &\leq C \Big( \|u\|_{D(I)} \|w\|_{D(I)} \Big)^{\frac{1}{2}} \Big( \|u\|_{\SHalf(I)} \|w\|_{\SHalf(I)} \Big)^{\frac{1}{2}}. \label{eq:BilinnablauwEndpoint}
	\end{align}
	 \end{enumerate}
\end{lemma}

\begin{proof}
	To prove~\eqref{eq:Bilinvu}, we note that~\cite[Theorem~3.1]{CHN23} yields
 \begin{align*}
     \|\Re(v) u\|_{\NOne(I)}
       &\lesssim \|v\|_{\Y(I)} \|u\|_{\SOne(I)}
 \end{align*}
 so that we only have to show
 \begin{align*}
     \|\Re(v) u\|_{\NOne(I)}
       &\lesssim \|v\|_{ L^2_t W^{1,4}_x(I \times \R^4)} \|u\|_{\SOne(I)}.
 \end{align*}
 To that purpose, we observe that, extending $g$ by $0$ from $I$ to $\R$ and applying Bernstein's inequality, we get
 \begin{align*}
     \| g \|_{\NOne(I)} \lesssim \| g\|_{L^2_t W^{1,\frac{4}{3}}_x(I \times \R^4)}.
 \end{align*}
 An elementary product estimate thus gives
 \begin{align*}
      \| \Re(v) u \|_{\NOne(I)} &\lesssim \| \Re(v) u\|_{L^2_t W^{1,\frac{4}{3}}_x(I \times \R^4)} \lesssim \| v \|_{L^2_t W^{1,4}_x(I \times \R^4)} \| u \|_{L^\infty_t H^1(I \times \R^4)} \lesssim \| v \|_{L^2_t W^{1,4}_x(I \times \R^4)} \| u \|_{\SOne(I)},
 \end{align*}
 which finishes the proof of~\eqref{eq:Bilinvu}.
 
 Estimate~\eqref{eq:Bilinnablauu} is a direct consequence of~\cite[Corollary~4.2]{CHN23}.
	Estimates~\eqref{eq:BilinvuEndpoint} and~\eqref{eq:BilinnablauwEndpoint} are the estimates from Propositions~6.1 and 6.2 in~\cite{CHN23}, respectively, in the case $d = 4$.
\end{proof}

The presence of noise gives rise to several lower order terms,
particularly,
including derivative terms  
that are usually hard for Schr\"odinger flows. 
The following estimates are important to control these terms 
in the new functional spaces.  

\begin{lemma} [Control of noise terms]  \label{lem:BilinearLowerOrder} 
Let $I \subseteq \R$ be a finite interval.
	\begin{enumerate}
	\item \label{it:ControlNoiseG1}  
We have the estimates  
	\begin{align}
		\| b \cdot \nabla u \|_{\GOne(I)} &\lesssim \Big(|I|^{\frac{1}{2}} \|b\|_{L^\infty_t H^2_x} + \sum_{j = 1}^4 \|b\|_{L^{1,\infty}_{\vece_j}}^{\frac{1}{2}}(\|b\|_{L^{1,\infty}_{\vece_j}}^{\frac{1}{2}} + \|b\|_{L^{\infty}_{t,x}}^{\frac{1}{2}})\Big) \|u\|_{\XOne(I)}, \label{eq:Bilinbnablau} \\
		\| c u\|_{\GOne(I)} &\lesssim |I|^{\frac{1}{2}} \|c\|_{L^\infty_t H^2_x} \|u\|_{\XOne(I)}, \label{eq:Bilincu} \\
		\| \cT_{\cdot}(W_2) u \|_{\GOne(I)} &\lesssim |I|^{\frac{1}{2}} \|\cT_{\cdot}(W_2)\|_{L^\infty_t H^2_x} \|u\|_{\XOne(I)}. \label{eq:BilincTu}
	\end{align}
	\item \label{it:ControlNoiseN12}  
Moreover, we have 
	\begin{align} 
		&\|(b \cdot \nabla u)_{HL + HH} + c u - \Re(\cT_t(W_2)) u\|_{\NHalf(I)} \nonumber \\
		&\qquad \lesssim |I|^{\frac{1}{2}}(\|b\|_{L^\infty_t H^2_x} + \|c\|_{L^\infty_t H^2_x} + \|\Re(\cT_t(W_2))\|_{L^\infty_t H^2_x}) \|u\|_{L^\infty_t H^1_x}. \label{eq:BilinNoiseN12} 
	\end{align}
In the above estimates all the space-time norms are taken over $I \times \R^4$.
	\end{enumerate}
\end{lemma}

\begin{proof}
	\ref{it:ControlNoiseG1}
 Let us start with the proof of estimate~\eqref{eq:Bilinbnablau}.
 Writing $b \cdot \nabla u = (b \cdot \nabla u)_{HL} + (b \cdot \nabla u)_{HH} + (b \cdot \nabla u)_{LH}$ and extending $b$ and $u$ by $0$ from $I$ to $\R$, 
 the definition of the $\GOne(I)$-norm implies
 \begin{align}
 \label{eq:bnablauG}
 	\|b \cdot \nabla u\|_{\GOne(I)} \lesssim \| P_1(b \cdot \nabla u)\|_{\NOne_1} + \Big(\sum_{\lambda \in 2^\N} \|P_\lambda (b \cdot \nabla u)_{HL + HH}\|_{\NOne_\lambda}^2\Big)^{\frac{1}{2}} + \sum_{j = 1}^4 \Big(\sum_{\lambda \in 2^\N} \lambda \|P_\lambda (b \cdot \nabla u)_{LH}\|_{L^{1,2}_{\vece_j}}^2\Big)^{\frac{1}{2}}.
 \end{align}
 Since by Bernstein's inequality,
 \begin{align*}
    \|P_\lambda g\|_{\NOne_\lambda} \lesssim \lambda \|P_\lambda g\|_{L^2_t L^{\frac{4}{3}}_x},
 \end{align*}
 we obtain
 \begin{align*}
 	\|P_\lambda (b \cdot \nabla u)_{HL}\|_{\NOne_\lambda}
    &\lesssim \sum_{\mu \sim \lambda} \mu \|P_\mu b P_{\leq \frac{\mu}{2^8}} \nabla u\|_{L^2_t L^{\frac{4}{3}}_x} \\
 	&\lesssim \sum_{\mu \sim \lambda} \mu \| P_\mu b\|_{L^2_t L^4_x} \|P_{\leq \frac{\mu}{2^8}} \nabla u\|_{L^\infty_t L^2_x} \\
    &\lesssim \|u\|_{\XOne} \sum_{\mu \sim \lambda} \mu \|P_\mu b\|_{L^2_t L^4_x}.
 \end{align*}
 Summing up, we conclude
 \begin{align}
 	\Big(\sum_{\lambda \in 2^\N} \|P_\lambda (b \cdot \nabla u)_{HL}\|_{\NOne_\lambda}^2\Big)^{\frac{1}{2}}
 	\lesssim \|b\|_{L^2_t B^{1}_{4,2}} \|u\|_{\XOne} \lesssim |I|^{\frac{1}{2}} \|b\|_{L^\infty_t H^2_x} \|u\|_{\XOne}. \label{eq:bnablauNl}
 \end{align}

 The usual adaptions yield the same estimate for the $HH$-component of $b \cdot \nabla u$. Since $P_1 (b \cdot \nabla u) = P_1 (b \cdot \nabla u)_{HH}$, we also obtain this estimate for the first term on the right-hand side of~\eqref{eq:bnablauG} in this way.

 For the third summand in~\eqref{eq:bnablauG},
 we use the H\"older inequality and decompose the modulation to derive
 \begin{align}
 	\lambda^{\frac{1}{2}} \|P_\lambda (b \cdot \nabla u)_{LH}\|_{L^{1,2}_{\vece_j}}
 	& \lesssim \sum_{\mu \sim \lambda} \mu^{\frac{1}{2}} \|P_{\leq \frac{\mu}{2^8}} b P_\mu \nabla u\|_{L^{1,2}_{\vece_j}} \notag \\
 	& \lesssim \mu^{\frac{1}{2}} \sum_{\mu \sim \lambda}\||P_{\leq \frac{\mu}{2^8}} b|^{\frac{1}{2}}\|_{L^{2,\infty}_{\vece_j}}  \mu^{\frac{1}{2}} \||P_{\leq \frac{\mu}{2^8}} b|^{\frac{1}{2}} |P_\mu \nabla u| \|_{L^2_{t,x}}     	\label{eq:bnablauLS} \\
 	&\lesssim \|b\|_{L^{1,\infty}_{\vece_j}}^{\frac{1}{2}} \sum_{\mu \sim \lambda} \mu^{\frac{1}{2}}
    \left(\||P_{\leq \frac{\mu}{2^8}} b|^{\frac{1}{2}} |C_{\leq (\frac{\mu}{2^8})^2}P_\mu \nabla u| \|_{L^2_{t,x}} + \||P_{\leq \frac{\mu}{2^8}} b|^{\frac{1}{2}} |C_{> (\frac{\mu}{2^8})^2}P_\mu \nabla u| \|_{L^2_{t,x}} \right). \notag
 \end{align}
 For the low-modulation contribution on the right-hand side,
 we use \eqref{eq:DecompositionAngular} to infer
 \begin{align}
 	&\|b\|_{L^{1,\infty}_{\vece_j}}^{\frac{1}{2}} \sum_{\mu \sim \lambda} \mu^{\frac{1}{2}} \||P_{\leq \frac{\mu}{2^8}} b|^{\frac{1}{2}} |C_{\leq (\frac{\mu}{2^8})^2}P_\mu \nabla u| \|_{L^2_{t,x}} \nonumber\\
 	&\lesssim \|b\|_{L^{1,\infty}_{\vece_j}}^{\frac{1}{2}} \sum_{\mu \sim \lambda} \sum_{l = 1}^4 \mu^{\frac{1}{2}} \Big\||P_{\leq \frac{\mu}{2^8}} b|^{\frac{1}{2}} \Big|P_{\mu, \vece_l} \Big[\prod_{k = 1}^{l-1} (I - P_{\mu, \vece_k})\Big] C_{\leq (\frac{\mu}{2^8})^2}P_\mu \nabla u \Big| \Big\|_{L^2_{t,x}} \nonumber\\
 	&\lesssim \|b\|_{L^{1,\infty}_{\vece_j}}^{\frac{1}{2}} \sum_{\mu \sim \lambda} \sum_{l = 1}^4 \mu^{\frac{1}{2}} \||P_{\leq \frac{\mu}{2^8}} b|^{\frac{1}{2}}\|_{L^{2,\infty}_{\vece_l}} \|P_{\mu, \vece_l} C_{\leq (\frac{\mu}{2^8})^2}P_\mu \nabla u \|_{L^{\infty,2}_{\vece_l}} \nonumber\\
 	&\lesssim \max_{j = 1, \ldots, 4} \|b\|_{L^{1,\infty}_{\vece_j}}  \sum_{\mu \sim \lambda} \sum_{l = 1}^4 \mu^{\frac{3}{2}} \|P_{\mu, \vece_l} C_{\leq (\frac{\mu}{2^8})^2}P_\mu u \|_{L^{\infty,2}_{\vece_l}}
 	\lesssim \sum_{j = 1}^4 \|b\|_{L^{1,\infty}_{\vece_j}} \sum_{\mu \sim \lambda} \|u_\mu\|_{\XOne_\mu}. \label{eq:bnablauLS1}
 \end{align}
 For the high-modulation contribution in~\eqref{eq:bnablauLS} we derive
 \begin{align}
 	&\|b\|_{L^{1,\infty}_{\vece_j}}^{\frac{1}{2}} \sum_{\mu \sim \lambda} \mu^{\frac{1}{2}}  \||P_{\leq \frac{\mu}{2^8}} b|^{\frac{1}{2}} |C_{> (\frac{\mu}{2^8})^2}P_\mu \nabla u| \|_{L^2_{t,x}}   \notag \\
 	&\lesssim \|b\|_{L^{1,\infty}_{\vece_j}}^{\frac{1}{2}} \|b\|_{L^{\infty}_{t,x}}^{\frac{1}{2}} \sum_{\mu \sim \lambda} \mu^{\frac{3}{2}}  \|C_{> (\frac{\mu}{2^8})^2}P_\mu  u \|_{L^2_{t,x}} \nonumber\\
 	&\lesssim \|b\|_{L^{1,\infty}_{\vece_j}}^{\frac{1}{2}} \|b\|_{L^{\infty}_{t,x}}^{\frac{1}{2}} \sum_{\mu \sim \lambda} \mu^{-\frac{1}{2}}  \|C_{> (\frac{\mu}{2^8})^2}(\imu \partial_t + \Delta) P_\mu  u \|_{L^2_{t,x}} \nonumber \\
 	&\lesssim \|b\|_{L^{1,\infty}_{\vece_j}}^{\frac{1}{2}} \|b\|_{L^{\infty}_{t,x}}^{\frac{1}{2}} \sum_{\mu \sim \lambda}   \Big\|\Big(\frac{\mu + |\partial_t|}{\mu^2 + |\partial_t|}\Big)^{\frac{1}{4}}(\imu \partial_t + \Delta) P_\mu  u \Big\|_{L^2_{t,x}}   \notag \\
 	&\lesssim \|b\|_{L^{1,\infty}_{\vece_j}}^{\frac{1}{2}} \|b\|_{L^{\infty}_{t,x}}^{\frac{1}{2}} \sum_{\mu \sim \lambda}   \|u_\mu\|_{\XOne_\mu}. \label{eq:bnablauLS2}
 \end{align}
 Combining~\eqref{eq:bnablauLS} to~\eqref{eq:bnablauLS2}, we obtain
 \begin{align}
 	\sum_{j = 1}^4 \Big(\sum_{\lambda \in 2^\N} \lambda \| P_\lambda (b \cdot \nabla u)_{LH} \|_{L^{1,2}_{\vece_j}}^2 \Big)^{\frac{1}{2}}
 	\lesssim \sum_{j = 1}^4 \|b\|_{L^{1,\infty}_{\vece_j}}^{\frac{1}{2}} (\|b\|_{L^{1,\infty}_{\vece_j}}^{\frac{1}{2}} + \|b\|_{L^{\infty}_{t,x}}^{\frac{1}{2}}) \|u\|_{\XOne}.  \label{eq:bnablauLS3}
 \end{align}
 Inserting~\eqref{eq:bnablauNl} and~\eqref{eq:bnablauLS3} into~\eqref{eq:bnablauG}, we arrive at~\eqref{eq:Bilinbnablau}.

 Estimates~\eqref{eq:Bilincu} and~\eqref{eq:BilincTu} follow in the same way as~\eqref{eq:bnablauNl}.

 \ref{it:ControlNoiseN12} We first note that an application of Bernstein's inequality shows that for every $\mu \sim \lambda$ we have
 \begin{align*}
     \| P_\mu g \|_{\NHalf_\lambda} \lesssim \lambda^{\frac{1}{2}} \| P_\mu g \|_{L^2_t L^{\frac{4}{3}}_x}.
 \end{align*}
 Consequently, we derive
 \begin{align*}
     \| P_\lambda (b \cdot \nabla u)_{HL} \|_{\NHalf_\lambda} \lesssim \lambda^{\frac{1}{2}}\sum_{\mu \sim \lambda} \| P_\mu b \cdot P_{\leq \frac{\mu}{2^8}} \nabla u \|_{L^2_t L^{\frac{4}{3}}_x}
     \lesssim \lambda^{\frac{3}{2}} \sum_{\mu \sim \lambda} \| P_\mu b\|_{L^2_t L^{2}_x} \| \nabla u\|_{L^\infty_t L^2_x}.
 \end{align*}
 Taking the $l^2$-sum in $\lambda$ and extending $(b \cdot \nabla u)_{HL}$ by $0$ from $I$ to $\R$, we arrive at
 \begin{align*}
     \| (b \cdot \nabla u)_{HL} \|_{\NHalf(I)} \lesssim \| b \|_{L^2_t H^{\frac{3}{2}}_x(I \times \R^4)} \| u \|_{L^\infty_t H^1_x(I \times \R^4)} \lesssim |I|^{\frac{1}{2}} \| b \|_{L^\infty_t H^2_x(I \times \R^4)} \| u \|_{L^\infty_t H^1_x(I \times \R^4)}.
 \end{align*}
 The standard adaptions yield the same estimate for $(b \cdot \nabla u)_{HH}$. For the remaining terms we simply estimate
 \begin{align*}
     \| c u \|_{\NHalf} \lesssim \| c u \|_{L^2_t B^{\frac{1}{2}}_{\frac{4}{3},2}} \lesssim \| c \|_{L^2_t B^{\frac{1}{2}}_{4,2}} \| u \|_{L^\infty_t H^{\frac{1}{2}}_x} \lesssim \| c \|_{L^2_t H^2_x} \| u \|_{L^\infty_t H^1_x}.
 \end{align*}
 Extending $c u$ by $0$ from $I$ to $\R$ again, we get
 \begin{align*}
     \| c u \|_{\NHalf(I)} \lesssim \| c \|_{L^2_t H^2_x(I \times \R^4)} \| u \|_{L^\infty_t H^1_x(I \times \R^4)} \lesssim |I|^{\frac{1}{2}}  \| c \|_{L^\infty_t H^2_x(I \times \R^4)} \| u \|_{L^\infty_t H^1_x(I \times \R^4)}.
 \end{align*}
 The term $\Re(\cT_t(W_2))$ is treated analogously. The combination of these estimates implies~\eqref{eq:BilinNoiseN12}.
\end{proof}

\section{Product estimate for rescaling transforms} 
\label{Sec-Product}

In this section we show that the rescaling transforms are bounded maps on the $\XOne$-space.
The first step in this direction is the following lemma.

\begin{lemma}   [Product estimate for rescaling transforms]
	\label{lem:ProductNoiseInX}
	Let $\sigma \in [0,\infty)$, $I \subseteq \R$ be a bounded interval, and $u \in \XOne(I)$.
    Then, $e^{\pm W_1(\sigma)} u$ belongs to $\XOne(I)$ and we have
			\begin{align*}
				\|e^{\pm W_1(\sigma)} u\|_{\XOne(I)}
               \lesssim (1 +  \|e^{\pm W_1(\sigma)} - 1\|_{H^4}(1 + |I|^{\frac{1}{2}})) \|u\|_{\XOne(I)}.
			\end{align*}
\end{lemma}

\begin{proof}
By the definition of $\XOne(I)$, there exist extensions of $u$ which belong to $\XOne(\R)$.
We fix such an extension and also denote it by $u$ to ease the notation.
Let $\rho \in C_c^\infty(\R)$ be such that $\rho(t) = 1$ for $t \in I$ and $\rho(t) = 0$ for $t \notin I + [-1,1]$. We further set
\begin{align*}
	w_0 = e^{\pm W_1(\sigma)} - 1 \qquad \text{and} \qquad w(t) = \rho(t) w_0
\end{align*}
and note that $w_0 \in H^4(\R^4)$ as $W_1(\sigma) \in H^4(\R^4)$.

We claim that $w u$ belongs to $\XOne(\R)$ and
\begin{equation}
\label{eq:EstwuRealline}
	\|w u \|_{\XOne(\R)} \lesssim \|w_0\|_{H^4} (1 + |I|^{\frac{1}{2}}) \|u\|_{\XOne(\R)}.
\end{equation}
Since $(w u)_{|I} = w_0 u$, this claim implies that $w_0 u \in \XOne(I)$ and
\begin{align*}
	\|w_0 u\|_{\XOne(I)} \lesssim \|w_0\|_{H^4}(1 + |I|^{\frac{1}{2}}) \|u\|_{\XOne(I)},
\end{align*}
since $u$ is an arbitrary extension of $u$ in $\XOne(\R)$.
This yields the assertion of the lemma as $u \in \XOne(I)$.

Below we focus on the proof of \eqref{eq:EstwuRealline}
and consider the lateral Strichartz and adapted spaces separately.

\medskip
\paragraph{\bf  $\bullet$ Lateral Strichartz space component.}
We start with the lateral Strichartz space component of the $\XOne(\R)$-norm. For that component it is sufficient to show
	\begin{align}
		\label{eq:ProductNoiseLocalSmoothing}
		\Big( \sum_{\lambda \in 2^\N} \lambda^3 \|P_{\lambda, \vece_j} C_{\leq (\frac{\lambda}{2^8})^2} P_\lambda (w u)\|_{L^{\infty,2}_{\vece_j}}^2 \Big)^{\frac{1}{2}}
	   &\lesssim   \|w_0\|_{H^4_x} \Big( \sum_{\lambda \in 2^\N} \lambda^3 \|P_{\lambda, \vece_j} C_{\leq (\frac{\lambda}{2^8})^2} P_\lambda u\|_{L^{\infty,2}_{\vece_j}}^2 \Big)^{\frac{1}{2}} \notag \\
      &\quad + (1   + |I|^{\frac{1}{2}}) \|w_0\|_{H^4_x} \|u\|_{L^\infty_t H^1_x}.
	\end{align}

	To that purpose, we decompose
	\begin{align}
	\label{eq:ProductNoiseParaProductDecomp}
		w u &= \sum_{\lambda \in 2^{\N}} P_\lambda w P_{\leq \frac{\lambda}{2^8}} u + \sum_{\lambda \in 2^{\N}} \sum_{\mu \sim \lambda} P_\lambda w P_\mu u + \sum_{\lambda \in 2^\N} P_{\leq \frac{\lambda}{2^8}} w P_\lambda u   \notag \\
        &=: (wu)_{HL} + (wu)_{HH} + (wu)_{LH},
	\end{align}
	where $\mu \sim \lambda$ means $|\log_2(\frac{\mu}{\lambda})| \leq 7$.

For the high-low contribution we estimate via Bernstein's inequality
	\begin{align}
	\label{eq:ProductNoiseLocalSmoothingHL}
		 \Big( \sum_{\lambda \in 2^\N} \lambda^3 \|P_{\lambda, \vece_j} C_{\leq (\frac{\lambda}{2^8})^2}
        P_\lambda (w u)_{HL}\|_{L^{\infty,2}_{\vece_j}}^2 \Big)^{\frac{1}{2}}
         & \lesssim \Big( \sum_{\lambda \in 2^\N} \lambda^{4} \| P_\lambda w P_{\leq \frac{\lambda}{2^8}} u\|_{L^2_{t,x}}^2 \Big)^{\frac{1}{2}} \nonumber \\
		&\lesssim \Big( \sum_{\lambda \in 2^\N} \lambda^{4} \| P_\lambda w \|_{L^2_t L^\infty_x}^2
           \| P_{\leq \frac{\lambda}{2^8}} u\|_{L^\infty_t L^2_x}^2 \Big)^{\frac{1}{2}} \nonumber\\
		& \lesssim |I|^{\frac{1}{2}} \| u \|_{L^\infty_t L^2_x} \Big( \sum_{\lambda \in 2^\N} \lambda^8 \| P_\lambda w_0\|_{L^2_x}^2 \Big)^{\frac{1}{2}} \notag \\
		& \lesssim |I|^{\frac{1}{2}} \|w_0\|_{H^4} \| u \|_{L^\infty_t L^2_x}.
	\end{align}
	The usual adaptions yield the same estimate for the high-high contribution $(wu)_{HH}$.

   It remains to estimate the low-high contribution $(wu)_{LH}$.
   We first note that
	\begin{align}
	\label{eq:ProductNoiseLocalSmoothingLH}
		P_{\lambda, \vece_j} C_{\leq (\frac{\lambda}{2^8})^2} P_\lambda (w u)_{LH} = \sum_{\frac{\lambda}{2} \leq \mu \leq 2 \lambda} P_{\lambda, \vece_j} C_{\leq (\frac{\lambda}{2^8})^2} P_\lambda (P_{\leq \frac{\mu}{2^8}} w P_\mu u).
	\end{align}		
	We fix a dyadic number $\mu \in \{\frac{\lambda}{2}, \lambda, 2\lambda\}$ and write
	\begin{align}
		\label{eq:ProductNoiseLocalSmoothingIntroCommutator}
			P_{\lambda, \vece_j} C_{\leq (\frac{\lambda}{2^8})^2} P_\lambda (P_{\leq \frac{\mu}{2^8}} w P_\mu u)
    = & P_{\leq \frac{\mu}{2^8}} w P_{\lambda, \vece_j} C_{\leq (\frac{\lambda}{2^8})^2} P_\lambda P_\mu u \notag \\
    &+ \(P_{\lambda, \vece_j} C_{\leq (\frac{\lambda}{2^8})^2} P_\lambda (P_{\leq \frac{\mu}{2^8}} w P_\mu u) - P_{\leq \frac{\mu}{2^8}} w  P_{\lambda, \vece_j} C_{\leq (\frac{\lambda}{2^8})^2} P_\lambda P_\mu u\).
	\end{align}
	For the first term on the right-hand side of \eqref{eq:ProductNoiseLocalSmoothingIntroCommutator} we get
	\begin{align}
	\label{eq:ProductNoiseLocalSmoothingLHEstMain}
		\|P_{\leq \frac{\mu}{2^8}} w P_{\lambda, \vece_j} C_{\leq (\frac{\lambda}{2^8})^2} P_\lambda P_\mu u\|_{L^{\infty,2}_{\vece_j}} &\lesssim \|P_{\leq \frac{\mu}{2^8}} w \|_{L^\infty_{t,x}} \|P_{\lambda, \vece_j} C_{\leq (\frac{\lambda}{2^8})^2} P_\lambda P_\mu u\|_{L^{\infty,2}_{\vece_j}} \nonumber \\
		&\lesssim \|w_0\|_{H^4_x} \|P_{\lambda, \vece_j} C_{\leq (\frac{\lambda}{2^8})^2} P_\lambda u\|_{L^{\infty,2}_{\vece_j}}.
	\end{align}
	To estimate the remaining commutator term in~\eqref{eq:ProductNoiseLocalSmoothingIntroCommutator},
    we recall that $P_{\lambda, \vece_j} C_{\leq (\frac{\lambda}{2^8})^2} P_\lambda$
    is a convolution operator with kernel $\phi_\lambda$,
    where $\phi_\lambda(t,x) = \lambda^6 \phi(\lambda^2 t, \lambda x)$
    for a Schwartz function $\phi \in \Schw(\R \times \R^4)$.
    Hence, we have
	\begin{align}
	\label{eq:ProductNoiseCommutatorConvolution}
		&P_{\lambda, \vece_j} C_{\leq (\frac{\lambda}{2^8})^2} P_\lambda (P_{\leq \frac{\mu}{2^8}} w P_\mu u)(t,x) - P_{\leq \frac{\mu}{2^8}} w  P_{\lambda, \vece_j} C_{\leq (\frac{\lambda}{2^8})^2} P_\lambda P_\mu u(t,x) \nonumber\\
		&= \int_{\R \times \R^4} (P_{\leq \frac{\mu}{2^8}} w (t-s,x-y) - P_{\leq \frac{\mu}{2^8}} w(t,x)) \phi_\lambda(s, y) P_\mu u(t-s,x-y) \dd (s,y) \nonumber\\
		&= \int_{\R \times \R^4} \int_0^1 \nabla_{t,x} P_{\leq \frac{\mu}{2^8}} w(t - \eta s, x - \eta y) \dd \eta \cdot (-s, -y) \phi_\lambda (s,y)P_\mu u(t-s, x-y) \dd (s,y).
	\end{align}
	Using Bernstein's and Minkowski's inequality, we thus infer
	\begin{align}
		\label{eq:ProductNoiseLocalSmoothingEstCommutator}
		&\|P_{\lambda, \vece_j} C_{\leq (\frac{\lambda}{2^8})^2} P_\lambda (P_{\leq \frac{\mu}{2^8}} w P_\mu u) - P_{\leq \frac{\mu}{2^8}} w  P_{\lambda, \vece_j} C_{\leq (\frac{\lambda}{2^8})^2} P_\lambda P_\mu u\|_{L^{\infty,2}_{\vece_j}} \nonumber\\
		&\lesssim \lambda^{\frac{1}{2}} \|P_{\lambda, \vece_j} C_{\leq (\frac{\lambda}{2^8})^2} P_\lambda (P_{\leq \frac{\mu}{2^8}} w P_\mu u) - P_{\leq \frac{\mu}{2^8}} w  P_{\lambda, \vece_j} C_{\leq (\frac{\lambda}{2^8})^2} P_\lambda P_\mu u\|_{L^2_{t,x}} \nonumber \\
		&\lesssim \lambda^{\frac{1}{2}} \int_{\R \times \R^4} \Big\|  \int_0^1 \nabla_{t,x} P_{\leq \frac{\mu}{2^8}} w(t - \eta s, x - \eta y) \dd \eta \Big\|_{L^2_t L^\infty_x} \|P_\mu u(t-s, x-y)\|_{L^\infty_t L^2_x} (|s| + |y|) |\phi_\lambda(s,y)| \dd(s,y) \nonumber\\		
		&\lesssim \lambda^{\frac{1}{2}} \|\nabla_{t,x} w\|_{L^2_t L^\infty_x} \|P_\mu u\|_{L^\infty_t L^2_x} \lambda^{-1} \int_{\R \times \R^4}  (|\lambda^2 s| + |\lambda y|) \lambda^6 |\phi(\lambda^2 s,\lambda y)| \dd(s,y) \nonumber\\
		&\lesssim \lambda^{-\frac{1}{2}} \|\rho'\|_{L^2_t} \|\nabla w_0\|_{L^\infty_x} \|P_\mu u\|_{L^\infty_t L^2_x} \|(|s| + |y|)\phi(s,y)\|_{L^1_{s,y}} \nonumber \\
        &\lesssim \lambda^{-\frac{1}{2}} \| w_0\|_{H^4_x} \|P_\mu u\|_{L^\infty_t L^2_x}.
	\end{align}
	Thus, combining~\eqref{eq:ProductNoiseLocalSmoothingLH} to~\eqref{eq:ProductNoiseLocalSmoothingEstCommutator}  we derive
	\begin{align}
		\label{eq:ProductNoiseLocalSmoothingEstLH}
		&\Big( \sum_{\lambda \in 2^\N} \lambda^3 \|P_{\lambda, \vece_j} C_{\leq (\frac{\lambda}{2^8})^2} P_\lambda (w u)_{LH}\|_{L^{\infty,2}_{\vece_j}}^2 \Big)^{\frac{1}{2}} \nonumber\\
		&\lesssim  \Big( \sum_{\lambda \in 2^\N} \lambda^3 \|w_0\|_{H^4_x}^2 \|P_{\lambda, \vece_j} C_{\leq (\frac{\lambda}{2^8})^2} P_\lambda u\|_{L^{\infty,2}_{\vece_j}}^2 \Big)^{\frac{1}{2}} + \Big( \sum_{\lambda \in 2^\N} \lambda^2 \|w_0\|_{H^4_x}^2 \| P_\lambda u\|_{L^\infty_t L^2_x}^2 \Big)^{\frac{1}{2}}  \nonumber \\
		&\lesssim  \|w_0\|_{H^4_x} \|u\|_{L^\infty_t H^1_x} + \|w_0\|_{H^4_x} \Big( \sum_{\lambda \in 2^\N} \lambda^3 \|P_{\lambda, \vece_j} P_\lambda u\|_{L^{\infty,2}_{\vece_j}}^2 \Big)^{\frac{1}{2}}.
	\end{align}

Therefore, taking into account \eqref{eq:ProductNoiseLocalSmoothingHL} and the corresponding estimate for the high-high contribution
we obtain \eqref{eq:ProductNoiseLocalSmoothing}.

\medskip
\paragraph{\bf $\bullet$ Adapted space component.}
		We continue with the adapted function space component of the $\XOne(\R)$-norm.

\medskip
\paragraph{\bf $(i)$ High-low interaction.}
Let us first treat the high-low interaction. Using Bernstein's inequality we estimate
		\begin{align}
			\label{eq:ProductNoiseXAdaptStrichartzHL}
			&\Big( \sum_{\lambda \in 2^\N} (\lambda \|P_\lambda (w u)_{HL}\|_{L^\infty_t L^2_x} + \lambda \|C_{\leq (\frac{\lambda}{2^8})^2} P_\lambda (wu)_{HL}\|_{L^2_t L^4_x})^2 \Big)^{\frac{1}{2}} \nonumber\\
			&\lesssim \Big( \sum_{\lambda \in 2^\N} \lambda^2 \|P_\lambda w P_{\leq \frac{\lambda}{2^8}} u\|_{L^\infty_t L^2_x}^2\Big)^{\frac{1}{2}} +  \Big( \sum_{\lambda \in 2^\N} \lambda^4 \| P_\lambda w P_{\leq \frac{\lambda}{2^8}} u\|_{L^2_{t,x}}^2 \Big)^{\frac{1}{2}} \nonumber\\
			&\lesssim  \Big( \sum_{\lambda \in 2^\N} \lambda^2 \|P_\lambda w \|_{L^\infty_{t,x}}^2 \|P_{\leq \frac{\lambda}{2^8}} u\|_{L^\infty_t L^2_x}^2\Big)^{\frac{1}{2}}
         +  \Big( \sum_{\lambda \in 2^\N} \lambda^4
         \| P_\lambda w \|_{L^2_t L^\infty_x}^2 \| P_{\leq \frac{\lambda}{2^8}} u\|_{L^\infty_t L^2_x}^2 \Big)^{\frac{1}{2}} \nonumber \\
			&\lesssim \|u\|_{L^\infty_t L^2_x} \Big( \sum_{\lambda \in 2^\N} \lambda^6 \|P_\lambda w_0 \|_{L^2_x}^2\Big)^{\frac{1}{2}} + |I|^{\frac{1}{2}} \|u\|_{L^\infty_t L^2_x} \Big( \sum_{\lambda \in 2^\N} \lambda^8 \| P_\lambda w_0 \|_{L^2_x}^2  \Big)^{\frac{1}{2}}  \notag \\
			& \lesssim  (1 + |I|^{\frac{1}{2}}) \|w_0\|_{H^4_x} \|u\|_{\XOne(\R)}.
		\end{align}

		Next, we compute
		\begin{align}
		\label{eq:ProductNoiseXSchrOpHL}
			(\imu \partial_t + \Delta)(P_{\lambda} w P_{\leq \frac{\lambda}{2^8}} u)
			= P_{\lambda} w (\imu \partial_t + \Delta) P_{\leq \frac{\lambda}{2^8}} u + \imu \partial_t P_{\lambda} w P_{\leq \frac{\lambda}{2^8}} u
     + 2\nabla P_{\lambda} w \cdot \nabla P_{\leq \frac{\lambda}{2^8}} u
     + \Delta P_{\lambda} w P_{\leq \frac{\lambda}{2^8}} u.
		\end{align}

		For the first summand on the right-hand side we infer
		\begin{align}
		\label{eq:ProductNoiseXSchrOpEst}
			 \|P_{\lambda} w (\imu \partial_t + \Delta) P_{\leq \frac{\lambda}{2^8}} u \|_{L^2_{t,x}}
         & \lesssim \|P_{\lambda} w\|_{L^\infty_t L^4_x} \|(\imu \partial_t + \Delta) P_{\leq \frac{\lambda}{2^8}} u \|_{L^2_t L^4_x}  \notag \\
         & \lesssim \lambda\|P_{\lambda} w_0\|_{L^2_x} \sum_{1 \leq \mu \leq \frac{\lambda}{2^8}} \mu \|(\imu \partial_t + \Delta) P_{\mu} u \|_{L^2_{t,x}} \nonumber\\
	     &\lesssim \lambda\|P_{\lambda} w_0\|_{L^2_x} \sum_{1 \leq \mu \leq \frac{\lambda}{2^8}} \mu^{\frac{5}{4}} \Big\|\Big(\frac{\mu + |\partial_t|}{\mu^2 + |\partial_t|} \Big)^{\frac{1}{4}}(\imu \partial_t + \Delta) P_{\mu} u \Big\|_{L^2_{t,x}} \notag \\
         & \lesssim \lambda^3 \|P_\lambda w_0\|_{L^2_x} \|u\|_{\XOne(\R)}.
		\end{align}

For the second summand in~\eqref{eq:ProductNoiseXSchrOpHL},
since $\partial_t P_\lambda w = \partial_t \rho P_\lambda w_0$,
we have
		\begin{align}
			\label{eq:ProductNoiseXSchrOpEstTimeDer}
			\| \imu \partial_t P_{\lambda} w P_{\leq \frac{\lambda}{2^8}} u\|_{L^2_{t,x}}
			\lesssim \| \partial_t \rho P_\lambda w_0\|_{L^2_t L^\infty_x} \|P_{\leq \frac{\lambda}{2^8}} u\|_{L^\infty_t L^2_x}
			\lesssim \lambda^2 \|P_\lambda w_0\|_{L^2_x} \|u\|_{L^\infty_t L^2_x}.
		\end{align}

		For the remaining two summands in~\eqref{eq:ProductNoiseXSchrOpHL} we simply estimate by Bernstein's inequality
		\begin{align}
			\label{eq:ProductNoiseXSchrOpEstLowOrder}
			\|\nabla P_{\lambda} w \cdot \nabla P_{\leq \frac{\lambda}{2^8}} u\|_{L^2_{t,x}} + \|\Delta P_{\lambda} w P_{\leq \frac{\lambda}{2^8}} u\|_{L^2_{t,x}}
			&\lesssim \lambda^2 \|P_\lambda w\|_{L^2_t L^\infty_x} \|P_{\leq \frac{\lambda}{2^8}} u\|_{L^\infty_t L^2_x} \notag \\
			&\lesssim |I|^{\frac{1}{2}} \lambda^4 \|P_\lambda w_0\|_{L^2_x} \|u\|_{L^\infty_t L^2_x}.
		\end{align}

		Combining~\eqref{eq:ProductNoiseXSchrOpEst}, \eqref{eq:ProductNoiseXSchrOpEstTimeDer}  and~\eqref{eq:ProductNoiseXSchrOpEstLowOrder}, we infer
		\begin{align}
			\label{eq:ProductNoiseXL2Comp}
			\Big(\sum_{\lambda \in 2^\N} \Big\|\Big(\frac{\lambda + |\partial_t|}{\lambda^2 + |\partial_t|}\Big)^{\frac{1}{4}} (\imu \partial_t + \Delta) P_\lambda(w u)_{HL}\Big\|_{L^2_{t,x}}^2 \Big)^{\frac{1}{2}}
			&\lesssim \Big(\sum_{\lambda \in 2^\N}\| (\imu \partial_t + \Delta) P_\lambda(w u)_{HL}\|_{L^2_{t,x}}^2 \Big)^{\frac{1}{2}} \nonumber\\
			&\lesssim \|w_0\|_{H^4_x} (1 + |I|^{\frac{1}{2}}) \|u\|_{\XOne(I)}.
		\end{align}
		In view of~\eqref{eq:ProductNoiseXAdaptStrichartzHL}, we thus have
		\begin{align}
			\label{eq:ProductNoiseXSnormHL}
			\Big(\sum_{\lambda \in 2^\N} \|P_\lambda(w u)_{HL}\|_{\SOne_\lambda}^2 \Big)^{\frac{1}{2}}
			\lesssim \|w_0\|_{H^4_x} (1 + |I|^{\frac{1}{2}}) \|u\|_{\XOne(I)}.
		\end{align}

		Straightforward adaptions of the above arguments yield the same estimate for the high-high interaction, which includes the $P_1(w u) = P_1(wu)_{HH}$ part.
		
\medskip
\paragraph{\bf $(ii)$ Low-high interaction.}
		For the low-high contribution $(wu)_{LH}$,
        similar arguments as in~\eqref{eq:ProductNoiseXAdaptStrichartzHL} yield
		\begin{align}
		\label{eq:ProductNoiseXAdpatStrichartzLH1}
			\Big(\sum_{\lambda \in 2^\N} \lambda^2 \|P_\lambda (w u)_{LH}\|_{L^\infty_t L^2_x}^2\Big)^{\frac{1}{2}}
         &\lesssim \|w_0\|_{H^4_x}(1 + |I|^{\frac{1}{2}}) \Big(\sum_{\lambda \in 2^\N} \lambda^2 \|P_\lambda u\|_{L^\infty_t L^2_x}^2 \Big)^{\frac{1}{2}} \notag \\
			& \lesssim \|w_0\|_{H^4_x}(1 + |I|^{\frac{1}{2}})\|u\|_{\XOne(\R)}.
		\end{align}

		Regarding the estimate of
$\lambda \|C_{\leq (\frac{\lambda}{2^8})^2} P_\lambda (w u)_{LH}\|_{L^2_t L^4_x}$, we use the representations~\eqref{eq:ProductNoiseLocalSmoothingLH} and~\eqref{eq:ProductNoiseLocalSmoothingIntroCommutator} once more.	
For the first term in~\eqref{eq:ProductNoiseLocalSmoothingIntroCommutator}, we simply estimate
		\begin{align}
			\label{eq:ProductNoiseXStrichartzEndpoint}
			\|P_{\leq \frac{\mu}{2^8}} w  C_{\leq (\frac{\lambda}{2^8})^2} P_\lambda P_\mu u\|_{L^2_t L^4_x}
           &\lesssim \|P_{\leq \frac{\mu}{2^8}} w \|_{L^\infty_{t,x}} \|C_{\leq (\frac{\lambda}{2^8})^2} P_\lambda P_\mu u\|_{L^2_t L^4_x} \notag \\
           &\lesssim \|w_0\|_{H^4_x} \|C_{\leq (\frac{\lambda}{2^8})^2} P_\lambda u\|_{L^2_t L^4_x}.
		\end{align}
		For the commutator term in~\eqref{eq:ProductNoiseLocalSmoothingIntroCommutator}, we observe that
		\begin{align}
			\label{eq:ProductNoiseXStrichartzEndpointCommutator}
			&\|P_{\lambda, \vece_j} C_{\leq (\frac{\lambda}{2^8})^2} P_\lambda (P_{\leq \frac{\mu}{2^8}} w P_\mu u) - P_{\leq \frac{\mu}{2^8}} w  P_{\lambda, \vece_j} C_{\leq (\frac{\lambda}{2^8})^2} P_\lambda P_\mu u\|_{L^2_t L^4_x} \nonumber\\
		&\lesssim \lambda \|P_{\lambda, \vece_j} C_{\leq (\frac{\lambda}{2^8})^2} P_\lambda (P_{\leq \frac{\mu}{2^8}} w P_\mu u) - P_{\leq \frac{\mu}{2^8}} w  P_{\lambda, \vece_j} C_{\leq (\frac{\lambda}{2^8})^2} P_\lambda P_\mu u\|_{L^2_{t,x}}  \notag \\
		& \lesssim \| w_0\|_{H^4_x} \|P_\mu u\|_{L^\infty_t L^2_x},
		\end{align}	
		where the last estimate was shown in~\eqref{eq:ProductNoiseLocalSmoothingEstCommutator} by means of the representation~\eqref{eq:ProductNoiseCommutatorConvolution}.	Combining~\eqref{eq:ProductNoiseXStrichartzEndpoint} and~\eqref{eq:ProductNoiseXStrichartzEndpointCommutator} with~\eqref{eq:ProductNoiseLocalSmoothingLH} and~\eqref{eq:ProductNoiseLocalSmoothingIntroCommutator}, we arrive at
		\begin{align}
			\label{eq:ProductNoiseXAdpatStrichartzLH2}
			&\Big(\sum_{\lambda \in 2^\N} \lambda^2 \| C_{\leq (\frac{\lambda}{2^8})^2} P_\lambda (w u)_{LH}\|_{L^2_t L^4_x}^2\Big)^{\frac{1}{2}}
			\lesssim \|w_0\|_{H^4_x} \|u\|_{\XOne(\R)}.
		\end{align}
		
		Regarding the last component of the $\SOne_\lambda$-norm,
we can expand it as in \eqref{eq:ProductNoiseXSchrOpHL}
and note that the lower order terms in $(\imu \partial_t + \Delta)(P_{\leq \frac{\lambda}{2^8}} w  P_\lambda u)$ are controlled by
		\begin{align}
			\label{eq:ProductNoiseXSchrOpLowerOrd}
			\|\nabla P_{\leq \frac{\lambda}{2^8}} w \cdot \nabla P_{\lambda} u\|_{L^2_{t,x}} + \|\Delta P_{\leq \frac{\lambda}{2^8}} w P_{\lambda} u\|_{L^2_{t,x}}
			&\lesssim  \|\nabla w\|_{L^2_t L^\infty_x} \lambda \|P_\lambda u\|_{L^\infty_t L^2_x} + \|\Delta w\|_{L^2_t L^4_x} \|P_\lambda u\|_{L^\infty_t L^4_x} \nonumber \\
			&\lesssim |I|^{\frac{1}{2}} \|w_0\|_{H^4_x} \lambda \|P_\lambda u\|_{L^\infty_t L^2_x}
		\end{align}
		and
		\begin{align}
			\label{eq:ProductNoiseXSchrOpLowerOrdTimeDer}
			\| \imu \partial_t P_{\leq \frac{\lambda}{2^8}} w P_{\lambda} u\|_{L^2_{t,x}}
			\lesssim \|\partial_t \rho P_{\leq \frac{\lambda}{2^8}} w_0\|_{L^2_t L^\infty_x} \|P_{\lambda} u\|_{L^\infty_t L^2_x}
			\lesssim  \|w_0\|_{H^4_x} \|P_{\lambda} u\|_{L^\infty_t L^2_x}.
		\end{align}
		Next, we fix $\mu \in \{\frac{\lambda}{2}, \lambda, 2\lambda\}$.
It remains to treat the term $P_{\leq \frac{\mu}{2^8}} w (\imu \partial_t + \Delta) P_\mu u$.
 Splitting $P_\mu u$ in low and high temporal frequencies and applying the product estimate for fractional derivatives from Lemma~2.7 in~\cite{CHN23}, we obtain
		\begin{align}
		\label{eq:ProductNoiseXSNormL2txLHMain}
		 &\Big\|\Big(\frac{\lambda + |\partial_t|}{\lambda^2 + |\partial_t|}\Big)^{\frac{1}{4}} (P_{\leq \frac{\mu}{2^8}} w (\imu \partial_t + \Delta) P_\mu u) \Big\|_{L^2_{t,x}} \nonumber\\
		 &\lesssim \Big\|\Big(\frac{\lambda + |\partial_t|}{\lambda^2 + |\partial_t|}\Big)^{\frac{1}{4}} (P_{\leq \frac{\mu}{2^8}} w (\imu \partial_t + \Delta) P_{\leq (\frac{\lambda}{2^8})^2}^{(t)} P_\mu u) \Big\|_{L^2_{t,x}} + \| P_{\leq \frac{\mu}{2^8}} w (\imu \partial_t + \Delta) P_{> (\frac{\lambda}{2^8})^2}^{(t)} P_\mu u \|_{L^2_{t,x}} \nonumber \\
		 &\lesssim \lambda^{-\frac{1}{2}} \|(\lambda + |\partial_t|)^{\frac{1}{4}} (P_{\leq \frac{\mu}{2^8}} w (\imu \partial_t + \Delta)P_{\leq (\frac{\lambda}{2^8})^2}^{(t)} P_\mu u) \|_{L^2_{t,x}}  + \|w\|_{L^\infty_{t,x}} \|  (\imu \partial_t + \Delta) P_{> (\frac{\lambda}{2^8})^2}^{(t)} P_\mu u \|_{L^2_{t,x}} \nonumber\\
		 &\lesssim \lambda^{-\frac{1}{2}} \cdot \lambda^{-\frac{1}{4}} \|(\lambda + |\partial_t|)^{\frac{1}{4}}P_{\leq \frac{\mu}{2^8}} w\|_{L^\infty_{t,x}} \|(\lambda + |\partial_t|)^{\frac{1}{4}} (\imu \partial_t + \Delta)P_{\leq (\frac{\lambda}{2^8})^2}^{(t)} P_\mu u\|_{L^2_{t,x}} \nonumber \\
		 &\qquad + \|w_0\|_{H^4_x} \|  (\imu \partial_t + \Delta) P_{> (\frac{\lambda}{2^8})^2}^{(t)} P_\mu u \|_{L^2_{t,x}} \nonumber \\
		 &\lesssim  \|\lambda^{-\frac{1}{4}}(\lambda + |\partial_t|)^{\frac{1}{4}} \rho\|_{L^\infty_t} \|w_0\|_{L^\infty_x} \|\lambda^{-\frac{1}{2}}(\lambda + |\partial_t|)^{\frac{1}{4}} (\imu \partial_t + \Delta)P_{\leq (\frac{\lambda}{2^8})^2}^{(t)} P_\mu u\|_{L^2_{t,x}} \nonumber \\
		 &\qquad + \|w_0\|_{H^4_x} \|  (\imu \partial_t + \Delta) P_{> (\frac{\lambda}{2^8})^2}^{(t)} P_\mu u \|_{L^2_{t,x}} \nonumber \\
		 &\lesssim (1 + |I|^{\frac{1}{2}}) \|w_0\|_{H^4_x} \Big\|\Big(\frac{\lambda + |\partial_t|}{\lambda^2 + |\partial_t|}\Big)^{\frac{1}{4}} (\imu \partial_t + \Delta) P_\mu u \Big\|_{L^2_{t,x}},
		\end{align}
		where we used that $\|\lambda^{-\frac{1}{4}}(\lambda + |\partial_t|)^{\frac{1}{4}} \rho\|_{L^\infty_t}$ is uniformly bounded in $\lambda$ by $1 + |I|^{\frac{1}{2}}$.
		From~\eqref{eq:ProductNoiseXSchrOpLowerOrd}, \eqref{eq:ProductNoiseXSchrOpLowerOrdTimeDer}, and~\eqref{eq:ProductNoiseXSNormL2txLHMain} we infer
		\begin{align}
			\label{eq:ProductNoiseXSNormL2tx}
			\Big(\sum_{\lambda \in 2^\N} \Big\|\Big(\frac{\lambda + |\partial_t|}{\lambda^2 + |\partial_t|}\Big)^{\frac{1}{4}} (\imu \partial_t + \Delta) P_\lambda (w u)_{LH} \Big\|_{L^2_{t,x}}^2 \Big)^{\frac{1}{2}}
			\lesssim \|w_0\|_{H^4_x} (1 + |I|^{\frac{1}{2}}) \|u\|_{\XOne(I)}.
		\end{align}

Thus, combining this estimate with~\eqref{eq:ProductNoiseXAdpatStrichartzLH1} and~\eqref{eq:ProductNoiseXAdpatStrichartzLH2},
we conclude that
		\begin{align}
			\label{eq:ProductNoiseXSNormLH}
			\Big(\sum_{\lambda \in 2^\N} \|P_\lambda(w u)_{LH}\|_{\SOne_\lambda}^2 \Big)^{\frac{1}{2}}
			\lesssim \|w_0\|_{H^4_x} (1 + |I|^{\frac{1}{2}}) \|u\|_{\XOne(I)}.
		\end{align}

Finally, the combination of~\eqref{eq:ProductNoiseXSnormHL}, the corresponding estimate for the high-high interaction
and~\eqref{eq:ProductNoiseXSNormLH} yields
		\begin{align*}
			\Big(\sum_{\lambda \in 2^{\N_0}} \|P_\lambda(w u)\|_{\SOne_\lambda}^2 \Big)^{\frac{1}{2}}
			\lesssim \|w_0\|_{H^4_x} (1 + |I|^{\frac{1}{2}}) \|u\|_{\XOne(I)}.
		\end{align*}
		Together with~\eqref{eq:ProductNoiseLocalSmoothing},
this estimate implies~\eqref{eq:EstwuRealline} and thus the assertion of the lemma.
\end{proof}

As a consequence of Lemma~\ref{lem:ProductNoiseInX},
we get the following corollary.

\begin{corollary}
	\label{cor:RefinedRescalingInX}
	Let $\sigma \geq 0$ and $\tau > 0$.
	\begin{enumerate}
		\item \label{it:RefinedRescalingInX1} If $u_\sigma \in \XOne([0,\tau])$, then $u$ defined by $u(t) = e^{-W_1(\sigma)} u_\sigma(t - \sigma)$ for $t \in [\sigma, \sigma + \tau]$ belongs to $\XOne([\sigma, \sigma + \tau])$.
		\item \label{it:RefinedRescalingInX2} If $u \in \XOne([\sigma, \sigma + \tau])$, then $u_\sigma$ defined by $u_\sigma(t) = e^{W_1(\sigma)}u(t + \sigma)$ for $t \in [0,\tau]$ belongs to $\XOne([0,\tau])$.
	\end{enumerate}
	The statements in~\ref{it:RefinedRescalingInX1} and~\ref{it:RefinedRescalingInX2} remain true if we replace $[0,\tau]$ and $[\sigma, \sigma + \tau]$ by $[0,\tau)$ and $[\sigma, \sigma + \tau)$, respectively.
\end{corollary}

\begin{proof}
	We start with part~\ref{it:RefinedRescalingInX1}. Let $u_\sigma \in \XOne([0,\tau])$. By definition there exists an extension $\tilde{u}_\sigma$ of $u_\sigma$ which belongs to $\XOne(\R)$. Since the $\XOne(\R)$-norm is time-translation invariant, we also have $\tilde{u}_\sigma(\cdot - \sigma) \in \XOne(\R)$. Since $\tilde{u}_\sigma(t-\sigma) = u_\sigma(t-\sigma)$ for all $t \in [\sigma, \sigma + \tau]$, we obtain $u_\sigma(\cdot - \sigma) \in \XOne([\sigma, \sigma + \tau])$. Lemma~\ref{lem:ProductNoiseInX} thus implies $e^{-W_1(\sigma)} u_\sigma(\cdot - \sigma) \in \XOne([\sigma, \sigma + \tau])$.
	
	Part~\ref{it:RefinedRescalingInX2} follows in the same way. Moreover, we can replace  $[0,\tau]$ and $[\sigma, \sigma + \tau]$ by $[0,\tau)$ and $[\sigma, \sigma + \tau)$, respectively, in the above proof.
\end{proof}

\section{LWP and blow-up alternative}
\label{Sec-LWP}

The aim of this section is to prove the local well-posedness and blow-up alternative
in Theorem \ref{Thm-LWP} for the energy-critical Zakharov system \eqref{eq:StoZak}.
Theorem \ref{Thm-Zakharov-Low} 
can be proved in a similar manner. 

We first collect some H{\"o}lder continuity properties of the noise terms. This is just Lemma~6.1 from~\cite{HRSZ23} adapted to the regularity assumptions for $W_1$ and $W_2$ we make in dimension four.
\begin{lemma} \label{lem:PropNoise}
Let $T\in (0,\infty)$ and $\kappa\in (0,\frac{1}{2})$.
Then, $W_1$ is $C^\kappa$-H\"older continuous in $H^4$ and
$W_2$ and the process $t \mapsto \int_0^t e^{-\imu s|\nabla|} \dd W_2(s)$ are $C^\kappa$-H\"older
continuous in $H^2$.
Moreover, for every $j \in \{1, \ldots, 4\}$ and
for $\bbp$-a.e. $\omega\in \Omega$,
there exists a sequence $(n_l(\omega))_{l \in \N}$ in $\N$ with $n_l(\omega) \rightarrow \infty$ as $l \rightarrow \infty$ such that
\begin{align}  \label{eq:limphi1kbeta1k}
   \sum\limits_{k = n_l}^\infty \int  \sup_{y\in \R^3} |\nabla \phi^{(1)}_k(r \vece_j+y)| \dd r
   \sup\limits_{t\in [0,T]} |\beta^{(1)}_k(t,\omega)|
  \longrightarrow 0,\ \ \text{as}\ l \to \infty.
\end{align}
\end{lemma}

\subsection{Linear equation with potential}

To begin with, let us
first develop the well-posedness theory for the linear Schr\"odinger equation with forcing in $\GOne(I)$
and a potential,
which is a perturbation of a free wave.

With the estimates from Section~\ref{Sec:MultilinEst},
the proof of the following result follows along the same lines as the proof of~\cite[Theorem~7.1]{CHN23}.

\begin{lemma}
	\label{lem:LinSchrPotSmallTime}
	There exist $\epsilon > 0$ and $C > 0$ such that for any interval $I \subset \R$, $t_0 \in I$, $f \in H^1(\R^4)$, $F \in \GOne(I)$, and $V \in \Y(I)$ satisfying
	\begin{align*}
		\|V\|_{\Y(I) + L^2_t W_x^{1,4}(I \times \R^4)} < \epsilon,
	\end{align*}
	the Cauchy problem
	\begin{align}
		\label{eq:SchroedingerPotential}
		(\imu \partial_t  + \Delta - \Re(V))u = F, \qquad u(t_0) = f
	\end{align}
	has a unique solution $u \in C(I, L^2(\R^4)) \cap L^2_t L^{4}_x(I \times \R^4)$, which satisfies
	\begin{align*}
		\|u\|_{\XOne(I)} \leq C(\|f\|_{H^1} + \|F\|_{\GOne(I)}), 
    \end{align*} 
    and 
    \begin{align*}
     \|u\|_{L^2_t L^4_x(I \times \R^4)} \leq C( \|f\|_{L^2} + \|F\|_{L^2_t L^{\frac{4}{3}}_x(I \times \R^4)}).
	\end{align*}
\end{lemma}

\begin{proof}
	We first note that $u$ solves~\eqref{eq:SchroedingerPotential} if and only if
	\begin{equation}
		\label{eq:SchroedingerPotentialDuhamel}
		u(t) = e^{\imu (t - t_0) \Delta} f -i \int_{t_0}^t e^{\imu (t-s) \Delta} (\Re(V) u + F)(s) \dd s.
	\end{equation}
	Define the operator $\Psi(f, F,  V; u)$ by the right-hand side of~\eqref{eq:SchroedingerPotentialDuhamel}, for which we want to construct a fixed point. Let $R > 0$ and set
	\begin{align*}
		B_R = \{u \in \XOne(I) \colon \|u\|_{\XOne(I)} \leq R\}
	\end{align*}
	equipped with the metric induced by the $\XOne(I)$-norm. Lemmas~\ref{lem:LinEstimates} and \ref{lem:BilinearEstimates} imply
	\begin{align}
	\label{eq:SchrPotEstBound}
		\|\Psi(f, F,  V; u)\|_{\XOne(I)}
        &\lesssim \|f\|_{H^1} + \|\Re(V) u + F\|_{\GOne(I)}  \nonumber\\
		&\lesssim  \|f\|_{H^1} + \|V\|_{\Y(I) + L^2_t W^{1,4}_x(I \times \R^4)}\|u\|_{\XOne(I)} + \| F\|_{\GOne(I)}  \nonumber\\
        & \lesssim \|f\|_{H^1} +\epsilon \|u\|_{\XOne(I)} + \| F\|_{\GOne(I)}
	\end{align}
	and
	\begin{align}
	\label{eq:SchrPotEstContr}
		\|\Psi(f, F,  V; u) - \Psi(f, F,  V; w)\|_{\XOne(I)}
         &\lesssim \|\Re(V)(u - w)\|_{\GOne(I)}  \nonumber \\
		&\lesssim \|V\|_{\Y(I) + L^2_t W^{1,4}_x(I \times \R^4)}\|u - w\|_{\XOne(I)} \lesssim \epsilon \|u - w\|_{\XOne(I)}.
	\end{align} 

	We fix the maximum $C_0$ of the implicit constants on the right-hand sides of~\eqref{eq:SchrPotEstBound} and~\eqref{eq:SchrPotEstContr}, set $R = 2 C_0 (\|f\|_{H^1} + \|F\|_{\GOne(I)})$, and choose $\epsilon > 0$ so small that $C_0 \epsilon < \frac{1}{2}$. Estimates~\eqref{eq:SchrPotEstBound} and~\eqref{eq:SchrPotEstContr} thus yield that $\Psi(f, F,  V; \cdot)$ is a contractive self-mapping on the complete metric space $B_R$. Hence, the Cauchy problem~\eqref{eq:SchroedingerPotential} has a unique solution in $\XOne(I)$. Using~\eqref{eq:SchrPotEstBound} for this solution, we further get
	\begin{align*}
		\|u\|_{\XOne(I)} \leq 2 C_0 (\|f\|_{H^1} + \|F\|_{\GOne(I)}).
	\end{align*}
	
	The uniqueness in the larger space $C(I, L^2(\R^4)) \cap L^2_tL^4_x(I \times \R^4)$ follows from standard arguments and the Strichartz estimate
	\begin{align*}
		\|\cI_0(\Re(v) u)\|_{L^\infty_t L^2_x \cap L^2_t L^4_x} &\lesssim \|\Re(V) u\|_{L^2_t L^{\frac{4}{3}}_x} \lesssim \|V\|_{L^\infty_t L^2_x + L^2_t L^4_x} \|u\|_{L^\infty_t L^2_x \cap L^2_t L^4_x} \\
		&\lesssim \|V\|_{\Y(I) + L^2_t W^{1,4}_x(I \times \R^4)}  \|u\|_{L^\infty_t L^2_x \cap L^2_t L^4_x}.
	\end{align*}
	Applying this estimate to~\eqref{eq:SchroedingerPotentialDuhamel}, we also obtain
	\begin{align*}
		 \|u\|_{L^\infty_t L^2_x \cap L^2_t L^4_x} \lesssim \|f\|_{L^2} + \|\Re(V) u\|_{L^2_t L^{\frac{4}{3}}_x} + \|F\|_{L^2_t L^{\frac{4}{3}}_x} \lesssim \|f\|_{L^2} + \epsilon \|u\|_{L^\infty_t L^2_x \cap L^2_t L^4_x} + \|F\|_{L^2_t L^{\frac{4}{3}}_x}.
	\end{align*}
	Assuming that the implicit constant is smaller or equal than $C_0$, the last part of the assertion follows.
\end{proof}

In order to apply the previous lemma when the potential $V$ is a linear wave, 
we recall 
from \cite[Lemma~7.5]{CHN23} with parameters $d = 4$, $s = 1$, $l = 0$, $a = \frac{1}{4}$ and $\beta = \frac{1}{2}$ that the smallness condition in Lemma~\ref{lem:LinSchrPotSmallTime} is satisfied for linear waves on small time intervals. 

\begin{lemma}
	\label{lem:SmallnessWavePotential}
	Let $g \in L^2(\R^4)$, $V_L(t) = e^{\imu t |\nabla|} g$, and $\epsilon > 0$. There exist finitely many intervals $(I_j)_{j = 1, \ldots, N}$ such that $\R = \cup_{j = 1}^N I_j$, $\min|I_j \cap I_{j+1}| > 0$, and
	\begin{align*}
		\sup_{j = 1, \ldots, N} \|V_L\|_{\Y(I_j) + L^2_t W^{1,4}_x(I_j \times \R^4)} < \epsilon.
	\end{align*}
\end{lemma}

\begin{remark}
    \label{rem:WDFreeWaveProp}
    As in the proof of Theorem~7.1 in~\cite{CHN23}, the combination of Lemma~\ref{lem:LinSchrPotSmallTime}, Lemma~\ref{lem:SmallnessWavePotential}, and Lemma~\ref{lem:DecompX} shows that~\eqref{eq:SchroedingerPotential} has a unique solution in $C(I, L^2(\R^4)) \cap L^2_t L^4_x(I \times \R^4)$ on any interval $I$, which also satisfies the estimates in Lemma~\ref{lem:LinSchrPotSmallTime}. This particular shows that the propagation operators $\cU_V$ and $\cI_V$ introduced in Subsection~\ref{sec:Notation}, are well-defined for free wave potentials $V$.
\end{remark}

We are now ready to prove Theorem \ref{thm:LocalWP}.
The proof mainly proceeds in three steps
in Subsections \ref{Subsec-LWP}, \ref{Subsec-Max-Exist} and \ref{Subsec-Blowup} below,
respectively.

\subsection{Local well-posedness}  \label{Subsec-LWP}

We first prove that system \eqref{eq:StoZak} is locally well-posed up to some stopping time.

Set
\begin{align*}
   (u_0, v_0) := (X_0, Y_0),\ \ v_L := e^{\imu t |\nabla|} v_0,\ \ {\rm and}\  \ \rho := v - v_L.
\end{align*}
Then $(u,v)$ is a solution of~\eqref{eq:RanZakbc} if and only if $(u,\rho)$ solves
\begin{equation}  \label{eq:RanZakWavePot}
	\left\{\aligned
		(\imu \partial_t + \Delta - \Re(v_L)) u &= \Re(\rho) u - b\cdot \na u - cu + \Re(\cT_t(W_2)) u, \qquad &u(0) &= u_0, \\
		(\imu \partial_t + |\nabla|) \rho &= - |\nabla| |u|^2, &\rho(0) &= 0.
    \endaligned
    \right.
\end{equation}
	Noting that
	\begin{align*}
		\rho(t) =  - \cJ_0[|\nabla| |u|^2],
	\end{align*}
	we obtain a solution of~\eqref{eq:RanZakWavePot} - and thus of~\eqref{eq:RanZakbc} - if and only if
	\begin{align}
		u(t) = \cU_{v_L} [u_0](t) - \cI_{v_L}[\Re(\cJ_0[|\nabla| |u|^2])u](t) - \cI_{v_L}[b \cdot \nabla u + c u - \Re(\cT_{\cdot}(W_2)) u](t).
		\label{eq:LWPFixedPoint}
	\end{align}

	 Let us define the fixed point operator $\Phi(u_0, v_0; u)$ by the right-hand side of~\eqref{eq:LWPFixedPoint}.
     Let $\delta, R > 0$. Set
	\begin{align*}
		B_{R, \delta}(\tau) = \{u \in \XOne([0,\tau]) \colon \|u\|_{L^2_tL^4_x([0,\tau] \times \R^4)} \leq \delta, \, \|u\|_{\XOne([0,\tau])} \leq R\},
	\end{align*}		
	 where $\tau > 0$ is a stopping time to be fixed below. In this step all space-time norms are taken over $[0,\tau] \times \R^4$ so that we drop $[0,\tau] \times \R^4$ from the notation in the following. Equipped with the metric induced by $\| \cdot \|_{\XOne([0,\tau])}$ the set $B_{R, \delta}(\tau)$ is a complete metric space, cf. Remark~\ref{rem:NormComp}.

     Below we show that $\Phi(u_0, v_0; \cdot)$ is a contractive self-mapping on the ball $B_{R, \delta}(\tau)$.
     For this purpose, we fix $\epsilon > 0$ from Lemma~\ref{lem:LinSchrPotSmallTime}, a time $T > 0$, and define
     \begin{equation}
     	\label{eq:DefTauTilde}
     	\tau_0 := \inf\{t \in [0,T] \colon \|v_L\|_{\Y([0,t]) + L^2_t W^{1,4}_x([0,t] \times \R^4)} 
      \geq \epsilon\} \wedge \min\{2,T\}.
     \end{equation}
We point out that $\tau_0 > 0$ by Lemma~\ref{lem:SmallnessWavePotential}. In the following we assume $\tau \leq \tau_0$.

\medskip
\paragraph{\bf $\bullet$ Self-mapping}
	 Lemma~\ref{lem:LinSchrPotSmallTime} yields that
	 \begin{align*}
	 	\| \cU_{v_L}[ u_0 ]\|_{\XOne([0,\tau])} \leq C \|u_0\|_{H^1}, \qquad \text{and} \qquad
	 	\| \cI_{v_L} [F] \|_{\XOne([0,\tau])} \leq C \|F\|_{\GOne([0,\tau])}.
	 \end{align*}
	 We thus obtain
	 \begin{align}
	 	\| \Phi(u_0, v_0; u) \|_{\XOne([0,\tau])}
	 	&\leq C \|u_0\|_{H^1} + C \|\Re(\cJ_0[|\nabla| |u|^2])u\|_{\GOne([0,\tau])} + C \|b \cdot \nabla u\|_{\GOne([0,\tau])} + C \|c u\|_{\GOne([0,\tau])} \nonumber\\
	 	&\qquad + C \|\Re(\cT_{\cdot}(W_2)) u\|_{\GOne([0,\tau])}. \label{eq:EstFixedPointOp1}
	 \end{align}
	 The definition of $\GOne$ and Lemma~\ref{lem:BilinearEstimates} yield
	 \begin{align}
	 	\|\Re(\cJ_0[|\nabla| |u|^2])u\|_{\GOne([0,\tau])}
         &\lesssim \| \Re(\cJ_0[|\nabla| |u|^2])u\|_{\NOne([0,\tau])}    \nonumber\\
         & \lesssim \| \cJ_0[|\nabla| |u|^2] \|_{\Y([0,\tau])} \|u\|_{\SOne(I)} \nonumber\\
	 	&\lesssim \|u\|_{L^2_t L^4_x}^{2 \theta} \|u\|_{\XOne([0,\tau])}^{3 - 2 \theta}. \label{eq:FixedPointOpNonlin}
	 \end{align}
	 Thus, inserting this estimate into~\eqref{eq:EstFixedPointOp1} and employing Lemma~\ref{lem:BilinearLowerOrder}, we arrive at
	 \begin{align}
	 	\| \Phi(u_0, v_0; u) \|_{\XOne([0,\tau])}
	 	&\leq C \|u_0\|_{H^1} +  C \|u\|_{L^2_t L^4_x}^{2 \theta} \|u\|_{\XOne([0,\tau])}^{3 - 2 \theta}   + C \Big(\sum_{j = 1}^4 \|b\|_{L^{1,\infty}_{\vece_j}} + \|b\|_{L^\infty_t H^3_x} \Big) \|u\|_{\XOne([0,\tau])} \nonumber \\
	 	&\qquad + C \tau^{\frac{1}{2}}(\|b\|_{L^\infty_t H^2_x} + \|c\|_{L^\infty_t H^2_x} + \|\cT_{\cdot}(W_2)\|_{L^\infty_t H^2_x}) \|u\|_{\XOne([0,\tau])}. \label{eq:EstFixedPointOp2}
	 \end{align}
	
	 Concerning the estimate of $L^2_tL^4_x$-norm,
we write the linear propagator $\cU_{v_L}$ as
	 \begin{align*}
	 	\cU_{v_L}[u_0](t) = e^{\imu t \Delta} u_0 + \cI_{v_L}[\Re(v_L) e^{\imu (\cdot) \Delta} u_0](t).
	 \end{align*}
	 Using Lemma~\ref{lem:LinSchrPotSmallTime}, we thus obtain
	 \begin{align*}
	 	\|\cU_{v_L}[u_0] \|_{L^2_t L^4_x} &\leq \| e^{\imu t \Delta} u_0 \|_{L^2_t L^4_x} + \|\Re(v_L) e^{\imu t \Delta} u_0\|_{L^2_t L^{\frac{4}{3}}_x} \notag \\
        & \leq \| e^{\imu t \Delta} u_0 \|_{L^2_t L^4_x} + \|v_L\|_{L^\infty_t L^2_x} \| e^{\imu t \Delta} u_0 \|_{L^2_t L^4_x} \\
	 	&\leq (1 + \|v_0\|_{L^2}) \| e^{\imu t \Delta} u_0 \|_{L^2_t L^4_x}.
	 \end{align*}
	 Lemmas \ref{lem:BilinearEstimates} and \ref{lem:LinSchrPotSmallTime} also show that
	 \begin{align*}
	 	\| \cI_{v_L}[\cJ_0[|\nabla| |u|^2] u] \|_{L^2_t L^4_x}
        & \lesssim \|\cJ_0[|\nabla| |u|^2]\|_{L^\infty_t L^2_x} \|u\|_{L^2_t L^4_x} \\
        & \lesssim  \|\cJ_0[|\nabla| |u|^2]\|_{\Y([0,\tau])} \|u\|_{L^2_t L^4_x} \\
	 	&\lesssim  \|u\|_{\XOne([0,\tau])}^{2(1-\theta)} \|u\|_{L^2_t L^4_x}^{2\theta + 1}. 
	 \end{align*} 

     Thus, combining the last two estimates with Lemma~\ref{lem:BilinearLowerOrder}, we derive
	 \begin{align}
	 	\|\Phi(u_0, v_0; u)\|_{L^2_t L^4_x} &\leq C  \| e^{\imu t \Delta} u_0 \|_{L^2_t L^4_x} +  C \|u\|_{\XOne([0,\tau])}^{2(1-\theta)} \|u\|_{L^2_t L^4_x}^{2\theta + 1} \nonumber \\
	 	&\qquad + C \tau^{\frac{1}{2}}(\|b\|_{{L^\infty_t H^2_x}} + \|c\|_{{L^\infty_t H^2_x}} + \|\cT_{\cdot}(W_2)\|_{{L^\infty_t H^2_x}}) \|u\|_{\XOne([0,\tau])}  \nonumber\\
	 	&\qquad + C \Big(\sum_{j = 1}^4 \|b\|_{L^{1,\infty}_{\vece_j}} + \|b\|_{{L^\infty_t H^3_x}} \Big) \|u\|_{\XOne([0,\tau])}. 	\label{eq:EstFixedPointOp3}
	 \end{align}

We also note that, by the definition of the random coefficients $b$ and $c$ in~\eqref{eq:Defb} and~\eqref{eq:Defc},  
	 \begin{align}
	 	\label{eq:DefWstar}
	 	&\|b(t)\|_{H^3} + \sum_{j = 1}^4 \| b\|_{L^{1,\infty}_{\vece_j}((0,t) \times \R^4)} + \|c(t)\|_{{H^2}} + \| \cT_t(W_2) \|_{{H^2}} \nonumber\\
	 	&\leq C\Big(\| \nabla W_1(t)\|_{H^3} + \sum_{j = 1}^4 \sum_{k = 1}^\infty \int \sup_{y \in \R^3} |\nabla \phi^{(1)}_k(r \vece_j + y)| \dd r \sup_{s \in [0,t]} |\beta_k^{(1)}(s)| + \|W_1(t)\|_{{H^4}}^2 + \|W_1(t)\|_{{H^4}} \nonumber\\
	 	&\qquad + \|\cT_t(W_2)\|_{{H^2}}\Big) =: W^*(t).
	 \end{align}
	
	  Fixing now $C = C(\|v_0\|_{L^2})$ as the maximum of the generic constants in~\eqref{eq:EstFixedPointOp2} and~\eqref{eq:EstFixedPointOp3} and setting $R = 2 C \|u_0\|_{H^1}$, we get from~\eqref{eq:EstFixedPointOp2} and~\eqref{eq:EstFixedPointOp3}
	 \begin{align*}
	 	\| \Phi(u_0, v_0; u) \|_{\XOne([0,\tau])} &\leq \frac{R}{2} + C R^{2(1 -  \theta)} \delta^{2 \theta} R + C R  \Big(\sum_{j = 1}^4 \|b\|_{L^{1,\infty}_{\vece_j}} + \|b\|_{L^\infty_t {H^2_x}}\Big) \\
	 	&\qquad + C R \tau^{\frac{1}{2}}(\|b\|_{L^\infty_t {H^2_x}} + \|c\|_{L^\infty_t {H^2_x}} + \|\cT_{\cdot}(W_2)\|_{L^\infty_t {H^2_x}}), \\
	 	\| \Phi(u_0, v_0; u) \|_{L^2_t L^4_x} &\leq  C  \| e^{\imu t \Delta} u_0 \|_{L^2_t L^4_x} +  C R^{2(1-\theta)} \delta^{2\theta} \delta + C R  \Big(\sum_{j = 1}^4 \|b\|_{L^{1,\infty}_{\vece_j}} + \|b\|_{L^\infty_t {H^2_x}}\Big) \\
	 	&\qquad + C R \tau^{\frac{1}{2}}(\|b\|_{L^\infty_t {H^2_x}} + \|c\|_{L^\infty_t {H^2_x}} + \|\cT_{\cdot}(W_2)\|_{L^\infty_t {H^2_x}})
	 \end{align*}
	 for every $u \in B_{R, \delta}(\tau)$. Choosing $\delta \in (0,R)$ so small that $4 C R^{2(1 - \theta)} \delta^{2 \theta} \leq 1$ and defining the stopping time $\tilde{\tau}$ by
	 \begin{align}
	 \label{eq:DefFixedPointStoppingTimeSelfmap}
	 	 \tilde{\tau} := \inf\Big\{t \in [0,T] \colon C  \| e^{\imu (\cdot) \Delta} u_0 \|_{L^2_t L^4_x((0,t) \times \R^4)} +	2 C R  W^*(t) \geq \frac{\delta}{4}\Big\} \wedge \tau_0,
	 \end{align}
	 we conclude that $\tilde{\tau} > 0$ $\PP$-a.s.,
since $\lim_{t \rightarrow 0} W^*(t) = 0$ $\PP$-a.s., and that $\Phi(u_0, v_0; \cdot)$ maps $B_{R,\delta}(\tilde{\tau})$ into itself.

\medskip
\paragraph{\bf $\bullet$ Contraction}
	 In order to show that $\Phi(u_0, v_0; \cdot)$ is a contraction, we first argue as in~\eqref{eq:FixedPointOpNonlin} to derive
	 \begin{align*}
	 	&\|\Re(\cJ_0[|\nabla| |u|^2]) u - \Re(\cJ_0[|\nabla| |w|^2]) w\|_{\GOne([0,\tau])} \\
	 	&= \|\Re(\cJ_0[|\nabla|((\overline{u} - \overline{w})u)]) u + \Re(\cJ_0[|\nabla| (\overline{w}(u - w))]) u  + \Re(\cJ_0[|\nabla| |w|^2]) (u - w)\|_{\GOne([0,\tau])} \\
	 	&\lesssim (\|u-w\|_{\XOne([0,\tau])} \|u\|_{\XOne([0,\tau])} )^{1-\theta} (\|u-w\|_{L^2_t L^4_x} \|u\|_{L^2_t L^4_x})^\theta \|u\|_{\XOne([0,\tau])}   \\
	 	&\qquad + (\|w\|_{\XOne([0,\tau])} \|u-w\|_{\XOne([0,\tau])} )^{1-\theta} (\|w\|_{L^2_t L^4_x} \|u-w\|_{L^2_t L^4_x})^\theta \|u\|_{\XOne([0,\tau])}  \\
        & \qquad +\|w\|_{L^2_t L^4_x}^{2\theta} \|w\|_{\XOne([0,\tau])}^{2(1-\theta)} \|u - w\|_{\XOne([0,\tau])} \\
	 	&\lesssim \delta^{\theta} R^{2 - \theta} \|u-w\|_{\XOne([0,\tau])} + \delta^{2 \theta} R^{2(1 - \theta)} \|u - w\|_{\XOne([0,\tau])}
	 \end{align*}
	 for all $u,w \in B_{R,\delta}(\tau)$, where we also used Remark~\ref{rem:NormComp}. In the same way as we derived~\eqref{eq:EstFixedPointOp2}, we thus get
	 \begin{align}
	 	\|\Phi(u_0, v_0; u) - \Phi(u_0, v_0; w)\|_{\XOne([0,\tau])} &\leq C \delta^{\theta} R^{2 - \theta} \|u-w\|_{\XOne([0,\tau])} + C \delta^{2 \theta} R^{2(1 - \theta)} \|u - w\|_{\XOne([0,\tau])}  \notag  \\
	 	&\qquad +  C \Big(\sum_{j = 1}^4 \|b\|_{L^{1,\infty}_{\vece_j}}+ \|b\|_{L^\infty_t {H^2_x}} \Big) \|u-w\|_{\XOne([0,\tau])}  \label{eq:EstFixedPointOpDiff}  \\
	 	&\qquad + C \tau^{\frac{1}{2}}(\|b\|_{L^\infty_t {H^2_x}} + \|c\|_{L^\infty_t {H^2_x}} + \|\cT_{\cdot}(W_2)\|_{L^\infty_t {H^2_x}}) \|u-w\|_{\XOne([0,\tau])}.  \nonumber
	 \end{align}
	 Fixing the generic constant $C$ and taking $\delta > 0$ possibly smaller such that additionally
	 \begin{align*}
	 	C \delta^\theta R^{2- \theta} + C \delta^{2 \theta} R^{2(1 - \theta)} \leq \frac{1}{4},
	 \end{align*}
	  we update the definition of $\tilde{\tau}$ in~\eqref{eq:DefFixedPointStoppingTimeSelfmap} and set
		\begin{align*}
			\tilde{\tau}_1 := \inf\Big\{t \in [0,T]: W^*(t) \geq \frac{1}{4}\Big\} \wedge \tilde{\tau}.
		\end{align*}			
	  Then $\tilde{\tau}$ is a stopping time, $\tilde{\tau} > 0$ $\PP$-a.s., 
and $\Phi(u_0, v_0; \cdot)$ is a contractive self-mapping on $B_{R,\delta}(\tilde{\tau}_1)$.
	
	  As the constant $C$ and the radius $R$ are increasing in $\|u_0\|_{H^1}$ and $\|v_0\|_{L^2}$, we note that there is a small constant $\delta_*(\|u_0\|_{H^1}, \|v_0\|_{L^2}) > 0$, which is decreasing in both its arguments, such that
	  \begin{align}
	  \label{eq:FixedPointDefStoppingTimeFirstIntervalExt}
	  	\tilde{\tau}_1 = \inf\{t \in [0, T] \colon  \|e^{\imu (\cdot) \Delta} u_0\|_{L^2_t L^4_x([0,t] \times \R^4)} + W^*(t) \geq 4 \delta_*(\|u_0\|_{H^1}, \|v_0\|_{L^2})\} \wedge \tau_0.
	  \end{align}
	  Moreover, we define the stopping time
	  \begin{align}
	  \label{eq:FixedPointDefStoppingTimeFirstInterval}
	  	\tau_1 &= \inf\{t \in [0, T] \colon  \|e^{\imu (\cdot) \Delta} u_0\|_{L^2_t L^4_x([0,t] \times \R^4)} + W^*(t) \geq 2 \delta_*(\|u_0\|_{H^1}, \|v_0\|_{L^2})\} \nonumber\\
	  	&\qquad \wedge \inf\Big\{t \in [0,T] \colon \|v_L\|_{\Y([0,t]) + L^2_t W^{1,4}_x([0,t] \times \R^4)} \geq  \frac{\epsilon}{2}\Big\} \wedge \min\{1,T\}.
	  \end{align}
	Using Lemma~\ref{lem:LinSchrPotSmallTime} once again, we note that $\tau_1 > 0$ $\PP$-a.s. Moreover, for continuity reasons (employing also Lemma~\ref{lem:Continuity}~\ref{it:ContSumY}), we have $\tau_1 < \tilde{\tau}_1$ or $\tau_1 = T$ $\PP$-a.s.

	Since $\Phi(u_0, v_0; \cdot)$ is a contractive self-mapping on the complete metric space $B_{R,\delta}(\tilde{\tau}_1)$, Banach's fixed point theorem yields a unique solution $\tilde{u}_1 \in B_{R,\delta}(\tilde{\tau}_1)$ of~\eqref{eq:LWPFixedPoint}. Standard arguments show that $\tilde{u}_1$
is the unique solution of~\eqref{eq:LWPFixedPoint} in $\XOne([0,\tilde{\tau}_1])$.
	Then, setting
	\begin{align*}
		\tilde{v}_1(t) := v_L(t) - \cJ_0[|\nabla| |\tilde{u}_1|^2](t) = e^{\imu t |\nabla|} v_0  - \cJ_0[|\nabla| |\tilde{u}_1|^2](t), \ \
    t\in [0,\tilde{\tau}_1],
	\end{align*}
   we obtain that $\tilde{v}_1 \in \Y([0,\tilde{\tau}_1])$ by Lemma~\ref{lem:BilinearEstimates}.
    Thus, $(\tilde{u}_1, \tilde{v}_1)$ is the unique solution of~\eqref{eq:RanZakbc}, where uniqueness holds in $\XOne([0,\tilde{\tau}_1]) \times L^\infty_t L^2_x([0,\tilde{\tau}_1] \times \R^4)$.

   Finally, set
	\begin{align*}
		(u_1,v_1)(t) := (\tilde{u}_1(t \wedge \tilde{\tau}_1), \tilde{v}_1(t \wedge \tilde{\tau}_1)),\ \ t\in [0,T].
	\end{align*}
	Then, $(u_1, v_1)$ is an $\{\cF_t\}$-adapted process in $C([0,T], H^1(\R^4) \times L^2(\R^4))$ (see e.g.~\cite{BRZ14} for the relevant arguments)
and solves~\eqref{eq:RanZakbc} on $[0,\tilde{\tau}_1]$.

\begin{remark} \label{Rem-twostopping}
   We introduced two stopping times above
 as a preparation for the gluing procedure in 
 the next step below,
where an overlap with positive measure of two intervals is required
to conclude that the glued solution belongs to $\XOne$ on the union of these intervals, cf. Proposition~\ref{prop:GluingSolutions} and Lemma~\ref{lem:DecompX}.
\end{remark}

\subsection{Extension to maximal existence time}  \label{Subsec-Max-Exist}
In this step,
we extend the solution from Step~1 to its maximal existence time.
The proof relies crucially on the inductive application of refined rescaling transforms and the gluing procedure.
	
	 Let $n \in \N$ and assume $(u_n, v_n)$ is an $\{\cF_t\}$-adapted continuous process
in $H^1 \times L^2$ and that $\sigma_n \leq \tilde{\sigma}_n$ are $\{\cF_t\}$-stopping times,
such that $\sigma_n < \tilde{\sigma}_n$ or $\sigma_n = T$ $\PP$-a.s.,
and that $(u_n, v_n)$ is the unique solution of~\eqref{eq:RanZakbc} on $[0,\tilde{\sigma}_n]$ in $\XOne([0,\tilde{\sigma}_n]) \times \Y([0,\tilde{\sigma}_n])$
satisfying $(u_n, v_n) \equiv (u_n(\tilde{\sigma}_n), v_n(\tilde{\sigma}_n))$ on $[\tilde{\sigma}_n, T]$.
	
In view of Propositions \ref{prop:RefinedRescaling} and \ref{prop:GluingSolutions},
we aim to solve \eqref{eq:RanZakbsigmacsigma} with the initial data
	 \begin{align}
	 \label{eq:DefInitialDataRescaled}
	 	(u_{0,n}, v_{0,n}) = (e^{W_1(\sigma_n)} u_n(\sigma_n), v_n(\sigma_n) + \cT_{\sigma_n}(W_2)),
	 \end{align}
	 i.e.  the system
	 \begin{equation}   \label{eq:RanZakbsigmacsigmaInProof}
	\left\{\aligned
	  \imu \partial_t u_\sigma + \Delta u_\sigma
	   &= \Re(v_\sigma) u_\sigma - b_\sigma \cdot \nabla u_\sigma - c_\sigma u_\sigma + \Re(\cT_{\sigma + \cdot, \sigma}(W_2)) u_\sigma,  \\
	 \imu \partial_t v_\sigma + |\nabla |v_\sigma  &= - |\nabla||u_\sigma|^2, \\
	 (u_\sigma(0), v_\sigma(0)) &= (u_{0,n}, v_{0,n}),
	\endaligned
	\right.
\end{equation}
where
\begin{align}
\label{eq:DefbsigmacsigmaInProof}
	&b_\sigma = 2 \nabla W_{1,\sigma}, \quad c_\sigma = |\nabla W_{1,\sigma}|^2 + \Delta W_{1,\sigma}, \quad W_{1,\sigma}(t)= W_1(\sigma+t) - W_1(\sigma), \qquad \text{and } \nonumber\\
	&\cT_{\sigma+t, \sigma} (W_2)
     = -\imu \int_\sigma^{\sigma+t} e^{\imu (\sigma+t-s)|\na|} \dd W_2(s) \quad (t \in [0,T]).
\end{align}

Proceeding as in Step~1,
we define $v_{L,n}(t) := e^{\imu t |\nabla|} v_{0,n}$
and $\rho_{\sigma,n} := v_\sigma - v_{L,n}$.
Similarly to \eqref{eq:LWPFixedPoint},
we see that $(u_\sigma, v_\sigma)$ is a solution of~\eqref{eq:RanZakbsigmacsigmaInProof} if and only if $u_\sigma$ solves
\begin{equation}
	\label{eq:FixedPointusigma}
	u_\sigma(t) = \cU_{v_{L,n}}[u_{0,n}](t) - \cI_{v_{L,n}}[\Re(\cJ_0[|\nabla| |u_\sigma|^2])u_\sigma](t) - \cI_{v_{L,n}}[b_\sigma \cdot \nabla u_\sigma + c_\sigma u_\sigma - \Re(\cT_{\sigma + \cdot,\sigma}(W_2)) u_\sigma](t).
\end{equation}

We define the fixed point operator $\Phi_\sigma(u_{0,n}, v_{0,n}; u_\sigma)$ by the right-hand side of~\eqref{eq:FixedPointusigma} as well as
\begin{align}
	\label{eq:DefWStarsigman}
	W^*_{\sigma_n}(t) &= \| \nabla W_{1,\sigma_n}(t)\|_{H^3} + \sum_{j = 1}^4 \sum_{k = 1}^\infty \int \sup_{y \in \R^3} |\nabla \phi^{(1)}_k(r \vece_j + y)| \dd r \sup_{s \in [0,t]} |\beta_k^{(1)}(\sigma_n + s) - \beta_k^{(1)}(\sigma_n)| \nonumber\\
	&\qquad + \|W_{1,\sigma_n}(t)\|_{{H^4}}^2 + \|W_{1,\sigma_n}(t)\|_{{H^4}} + \|\cT_{\sigma_n + t,\sigma_n}(W_2)\|_{{H^2}},
\end{align}
the $\{\cF_{\sigma_n + t}\}$-stopping times
\begin{align*}
	\tilde{\tau}_{n+1} &=\inf\{t \in [0, T] \colon  \|e^{\imu (\cdot) \Delta} u_{0,n}\|_{L^2_t L^4_x([0,t] \times \R^4)} + W_{\sigma_n}^*(t) \geq 4 \delta_*(\|u_{0,n}\|_{H^1}, \|v_{0,n}\|_{L^2})\} \\
	&\qquad \wedge \inf\{t \in [0,T] \colon \|v_{L,n}\|_{\Y([0,t]) + L^2_t W^{1,4}_x([0,t] \times \R^4)} \geq \epsilon\} \wedge \min\{2,T - \sigma_n\}, \\
	\tau_{n+1} &= \inf\{t \in [0, T] \colon  \|e^{\imu (\cdot) \Delta} u_{0,n}\|_{L^2_t L^4_x([0,t] \times \R^4)} + W_{\sigma_n}^*(t) \geq 2 \delta_*(\|u_{0,n}\|_{H^1}, \|v_{0,n}\|_{L^2})\} \\
	&\qquad \wedge \inf\Big\{t \in [0,T] \colon \|v_{L,n}\|_{\Y([0,t]) + L^2_t W^{1,4}_x([0,t] \times \R^4)} \geq  \frac{\epsilon}{2}\Big\} \wedge \min\{1,T - \sigma_n\}
\end{align*}
with $\delta_*$ from Subsection \ref{Subsec-LWP}, and
\begin{align*}
	\sigma_{n + 1} := \sigma_n + \tau_{n+1}, \qquad \tilde{\sigma}_{n+1} := \sigma_n + \tilde{\tau}_{n+1}.
\end{align*}
Note that $t \mapsto \|v_{L,n}\|_{\Y([0,t]) + L^2_t W^{1,4}_x([0,t] \times \R^4)}$ is continuous by Lemma~\ref{lem:Continuity}~\ref{it:ContSumY} so that $\tilde{\tau}_{n+1}$ and $\tau_{n+1}$ are indeed $\{\cF_{\sigma_n + t}\}$-stopping times.
Then, $\sigma_{n+1}$ and $\tilde{\sigma}_{n+1}$ are $\{\cF_t\}$-stopping times
(see \cite{BRZ14,LR15} for the relevant arguments)
with $\sigma_{n+1} \leq \tilde{\sigma}_{n+1} \leq T$,
as well as $\sigma_{n+1} < \tilde{\sigma}_{n+1}$ or $\sigma_{n+1} = \tilde{\sigma}_{n+1} = T$, $\PP$-a.s.

We note that, as mentioned in Subsection \ref{Subsec-LWP},
we need two stopping times in order to show later that
the glued solutions belong to $\XOne([0,\tilde{\sigma}_{n+1}]) \times \Y([0,\tilde{\sigma}_{n+1}])$.

Employing the estimates from Subsection~\ref{Subsec-LWP}, i.e.,
\eqref{eq:EstFixedPointOp1} to~\eqref{eq:EstFixedPointOp3} and~\eqref{eq:EstFixedPointOpDiff},
we derive as in Subsection~\ref{Subsec-LWP} that the operator $\Phi_\sigma(u_{0,n}, v_{0,n}; \cdot)$
is a contractive self-mapping on a closed subset of $\XOne([0,\tilde{\tau}_{n+1}])$.
Hence, there is a unique solution $\tilde{u}_{\sigma_{n+1}}$ of~\eqref{eq:FixedPointusigma} in this closed subset.
Setting $\tilde{v}_{\sigma_{n+1}} := v_{L,n} - \cJ_0[|\nabla| |\tilde{u}_{\sigma_{n+1}}|^2]$,
we thus obtain a solution $(\tilde{u}_{\sigma_{n+1}}, \tilde{v}_{\sigma_{n+1}})$
of~\eqref{eq:RanZakbsigmacsigmaInProof} in $\XOne([0,\tilde{\tau}_{n+1}]) \times \Y([0,\tilde{\tau}_{n+1}])$, which is unique in $\XOne([0,\tilde{\tau}_{n+1}]) \times L^\infty_t L^2_x([0,\tilde{\tau}_{n+1}] \times \R^4)$.

Next, define
\begin{align*}
	(u_{\sigma_{n+1}}(t), v_{\sigma_{n+1}}(t)) = (\tilde{u}_{\sigma_{n+1}}(t \wedge \tilde{\tau}_{n+1}), \tilde{v}_{\sigma_{n+1}}(t \wedge \tilde{\tau}_{n+1})),
    \ \ t \in [0,T].
\end{align*}
Then, $(u_{\sigma_{n+1}}, v_{\sigma_{n+1}})$ is an $\{\cF_{\sigma_{n} + t}\}$-adapted continuous process in $H^1 \times L^2$ which solves~\eqref{eq:RanZakbsigmacsigmaInProof} on $[0,\tilde{\tau}_{n+1}]$.

Finally, we use the gluing procedure to define
\begin{align*}
	u_{n+1}(t) &:= u_n(t) \chi_{[0,\sigma_n)}(t) + e^{-W_1(\sigma_n)}u_{\sigma_{n+1}}((t-\sigma_{n}) \wedge \tilde{\tau}_{n+1})\chi_{[\sigma_n, T]}(t), \\
	v_{n+1}(t) &:= v_n(t) \chi_{[0,\sigma_n)}(t) + \Big(v_{\sigma_{n+1}}((t-\sigma_n) \wedge \tilde{\tau}_{n+1}) - e^{\imu ((t - \sigma_n) \wedge \tilde{\tau}_{n+1}) |\nabla|} \cT_{\sigma_n}(W_2)\Big) \chi_{[\sigma_n, T]}(t)
\end{align*}
for all $t \in [0,T]$. By Proposition~\ref{prop:GluingSolutions},
$(u_{n+1}, v_{n+1})$ solves~\eqref{eq:RanZakbc} on $[0, \tilde{\sigma}_{n+1}]$.

\medskip
\paragraph{\bf Claim:}
$(u_{n+1}, v_{n+1})$ belongs to $\XOne([0,\tilde{\sigma}_{n+1}]) \times \Y([0,\tilde{\sigma}_{n+1}])$.

\medskip
To this end,
if $\sigma_n(\omega) = \tilde{\sigma}_n(\omega) = T$, there is nothing to show.
So it suffices to consider $\omega$ with $\sigma_n(\omega) < \tilde{\sigma}_n(\omega)$ in the following.

On the one hand, since $(u_n, v_n)$ solves~\eqref{eq:RanZakbc} on $[0,\tilde{\sigma}_n]$,
it also solves~\eqref{eq:RanZakbc} on $[\sigma_n, \tilde{\sigma}_n]$ with the initial data
$(u_n(\sigma_n), v_n(\sigma_n)) = (e^{-W_1(\sigma_n)} u_{0,n}, v_{0,n} - \cT_{\sigma_n}(W_2))$.
Proposition~\ref{prop:RefinedRescaling}~\ref{it:RefRescSigTo0} thus yields that
\begin{align*}
	(\underline{u}_n, \underline{v}_n):= (e^{W_1(\sigma_n)} u_n(\sigma_n + \cdot), v_n(\sigma_n + \cdot) + e^{\imu (\cdot) |\nabla|} \cT_{\sigma_n}(W_2))
\end{align*}
solves~\eqref{eq:RanZakSigma} on $[0,\tilde{\sigma}_n - \sigma_n]$ with the initial data $(u_{0,n}, v_{0,n})$.
Since the $\Y(\R)$-norm is time translation invariant and $v_n \in \Y([0,\tilde{\sigma}_n])$, we infer that $v_n(\sigma_n + \cdot)$ belongs to $\Y([-\sigma_n, \tilde{\sigma}_n - \sigma_n]) \subseteq \Y([0,\tilde{\sigma}_n - \sigma_n])$. As $\cT_{\sigma_n}(W_2) \in L^2(\R^4)$, we also have $e^{\imu (\cdot) |\nabla|} \cT_{\sigma_n}(W_2) \in \Y(\R) \subseteq  \Y([0,\tilde{\sigma}_n - \sigma_n])$. Moreover, Corollary~\ref{cor:RefinedRescalingInX}~\ref{it:RefinedRescalingInX2} yields $e^{W_1(\sigma_n)} u_n(\sigma_n + \cdot) \in \XOne([0, \tilde{\sigma}_n - \sigma_n])$,
as $u_n \in \XOne([0, \tilde{\sigma}_n]) \subseteq \XOne([\sigma_n, \tilde{\sigma}_n])$.
Thus, we conclude that $(\underline{u}_n, \underline{v}_n) \in \XOne([0, \tilde{\sigma}_n - \sigma_n]) \times \Y([0, \tilde{\sigma}_n - \sigma_n])$.

On the other hand,
$(u_{\sigma_{n+1}}, v_{\sigma_{n+1}})$ also solves~\eqref{eq:RanZakbsigmacsigmaInProof} with the initial data $(u_{0,n}, v_{0,n})$ in $\XOne([0,\tilde{\tau}_{n+1}]) \times \Y([0,\tilde{\tau}_{n+1}])$ and it is unique in $\XOne([0,\tilde{\tau}_{n+1}]) \times L^\infty_t L^2_x([0,\tilde{\tau}_{n+1}] \times \R^4)$.
We thus infer
\begin{align*}
	(\underline{u}_n, \underline{v}_n) = (u_{\sigma_{n+1}}, v_{\sigma_{n+1}}) \qquad \text{on } [0, (\tilde{\sigma}_n - \sigma_n) \wedge \tilde{\tau}_{n+1}].
\end{align*}
Via the definition of $(\underline{u}_n, \underline{v}_n)$, this
yields
\begin{align*}
	(u_n(t), v_n(t)) = (e^{-W_1(\sigma_n)} u_{\sigma_{n+1}}(t - \sigma_n), v_{\sigma_{n+1}}(t - \sigma_n) - e^{\imu(t - \sigma_n)|\nabla|} \cT_{\sigma_n}(W_2))
\end{align*}
for all $t \in [\sigma_n, ((\tilde{\sigma}_n - \sigma_n) \wedge \tilde{\tau}_{n+1}) + \sigma_n] = [\sigma_n, \tilde{\sigma}_n \wedge \tilde{\sigma}_{n+1}]$.
In view of the definition of $(u_{n+1}, v_{n+1})$, we thus infer that
\begin{align*}
	(u_{n+1}, v_{n+1})_{|[0,\tilde{\sigma}_n \wedge \tilde{\sigma}_{n+1}]} &= (u_n, v_n), \\
	(u_{n+1}, v_{n+1})_{|[\sigma_n,\tilde{\sigma}_{n+1}]} &= (e^{-W_1(\sigma_n)} u_{\sigma_{n+1}}(t - \sigma_n), v_{\sigma_{n+1}}(t - \sigma_n) - e^{\imu(t - \sigma_n)|\nabla|} \cT_{\sigma_n}(W_2)).
\end{align*}
Arguing as for $\underline{v}_n$ above, we infer that $(v_{n+1})_{|[\sigma_n,\tilde{\sigma}_{n+1}]} \in \Y([\sigma_n,\tilde{\sigma}_{n+1}])$. Corollary~\ref{cor:RefinedRescalingInX}~\ref{it:RefinedRescalingInX1} further shows that $e^{-W_1(\sigma_n)} u_{\sigma_{n+1}}(\cdot - \sigma_n) \in \XOne([\sigma_n, \tilde{\sigma}_{n+1}])$. Consequently, $(u_{n+1}, v_{n+1})_{|[\sigma_n,\tilde{\sigma}_{n+1}]} \in \XOne([\sigma_n,\tilde{\sigma}_{n+1}]) \times \Y([\sigma_n,\tilde{\sigma}_{n+1}])$, and $(u_{n+1}, v_{n+1})_{|[0,\tilde{\sigma}_n \wedge \tilde{\sigma}_{n+1}]} \in \XOne([0,\tilde{\sigma}_n \wedge \tilde{\sigma}_{n+1}]) \times \Y([0,\tilde{\sigma}_n \wedge \tilde{\sigma}_{n+1}])$ because of the properties of $(u_n, v_n)$.

Finally,
since $\tilde{\sigma}_n \wedge \tilde{\sigma}_{n+1} > \sigma_n>0$,
Lemma~\ref{lem:DecompX} yields that $(u_{n+1}, v_{n+1}) \in \XOne([0, \tilde{\sigma}_{n+1}]) \times \Y([0, \tilde{\sigma}_{n+1}])$,
as claimed.

\medskip
Now, combining the uniqueness properties of $(u_n, v_n)$ and $(u_{\sigma_{n+1}}, v_{\sigma_{n+1}})$, Proposition~\ref{prop:RefinedRescaling}, Corollary~\ref{cor:RefinedRescalingInX}, and standard arguments, we infer that $(u_{n+1}, v_{n+1})$ is the unique solution of~\eqref{eq:RanZakbc} in $\XOne([0,\tilde{\sigma}_{n+1}]) \times L^\infty_t L^2_x([0,\tilde{\sigma}_{n+1}] \times \R^4)$. Moreover, $(u_{n+1}, v_{n+1})$ is an $\{\cF_t\}$-adapted continuous process in $H^1 \times L^2$ (see e.g.~\cite{BRZ14} for the relevant arguments),
which coincides with $(u_n, v_n)$ on $[0,\sigma_n]$ and satisfies $(u_{n+1}(t), v_{n+1}(t)) = (u_{n+1}(\tilde{\sigma}_{n+1}), v_{n+1}(\tilde{\sigma}_{n+1}))$ for all $t \in [\tilde{\sigma}_{n+1}, T]$.

Inductively, we thus obtain an increasing sequence of $\{\cF_t\}$-adapted stopping times $(\sigma_n)$ as well as corresponding $\{\cF_t\}$-adapted processes $(u_n, v_n)$,
such that $(u_n, v_n)$ is the solution of~\eqref{eq:RanZakbc}
in $\XOne([0,\sigma_n]) \times \Y([0,\sigma_n])$, unique in $\XOne([0,\tilde{\sigma}_{n+1}]) \times L^\infty_t L^2_x([0,\tilde{\sigma}_{n+1}] \times \R^4)$,
and $(u_{n+1}, v_{n+1})$ coincides with $(u_n, v_n)$ on $[0,\sigma_n]$.

Setting
\begin{align*}
	\tau_T^* := \lim_{n \rightarrow \infty} \sigma_n \qquad \text{and} \qquad (u^T, v^T) := \lim_{n \rightarrow \infty} (u_n \chi_{[0, \tau_T^*)}, v_n \chi_{[0, \tau_T^*)}),
\end{align*}
we obtain an $\{\cF_t\}$-adapted stopping time $\tau_T^*$ as well as an $\{\cF_t\}$-adapted process $(u^T, v^T)$, which is
the solution of~\eqref{eq:RanZakbc} on $[0,\tau_T^*)$.

Since $\{\tau_T^*\}$ is increasing in $T$, and for $T' > T$,
the process $(u^{T'}, v^{T'})$ coincides with $(u^T, v^T)$ on $[0,\tau_T^*)$ by the uniqueness property,
we can define
\begin{align*}
	\tau^* := \lim_{T \rightarrow \infty} \tau_T^* \qquad \text{and} \qquad (u,v) := \lim_{T \rightarrow \infty} (u^T \chi_{[0,\tau^*)}, v^T \chi_{[0,\tau^*)}).
\end{align*}
The resulting process $(u,v)$ is thus $\{\cF_t\}$-adapted,
continuous in $H^1 \times L^2$ on $[0,\tau^*)$,
and uniquely solves \eqref{eq:RanZakbc}.

Finally, we use the rescaling transformation again to define
\begin{align*}
   (X,Y) := (e^{W_1} u, v + \cT_{\cdot}(W_2)).
\end{align*}
The equivalence result in Theorem~\ref{thm:Rescaling}
shows that $(X,Y)$ is the unique solution of~\eqref{eq:StoZak} on $[0,\tau^*)$ in the sense of Definition~\ref{def:Solution}.

\subsection{Blow-up alternative}   \label{Subsec-Blowup}
It remains to prove the blow-up alternative in Theorem~\ref{thm:LocalWP}. We argue by contradiction and assume that it is not true. Also employing Lemma~\ref{lem:PropNoise}, we thus find a set $\Omega'$ of positive measure such that for every $\omega \in \Omega'$ we have
\begin{enumerate}
	\item $\tau^*(\omega) < \infty$,
	\item $\limsup_{t \rightarrow \tau^*(\omega)} (\|X(t,\omega)\|_{H^1} + \|Y(t,\omega)\|_{L^2}) < \infty$,
	\item $\|X(\cdot, \omega)\|_{L^2_t W^{\frac{1}{2},4}_x([0,\tau^*(\omega)) \times \R^4)} < \infty$,
\end{enumerate}
and~\eqref{eq:limphi1kbeta1k} is satisfied.
We fix such an $\omega$ in the following. Let $T \in (0,\infty)$ such that $T > \tau^*(\omega)$. Let $\{\sigma_n, \tau_n\}$ be constructed as in Subsection \ref{Subsec-Max-Exist} for this $T$. For convenience  the dependence on $\omega$ is dropped in the following. 

Since $(u_{0,n}, v_{0,n}) = (X(\sigma_n), Y(\sigma_n))$ for the initial data from~\eqref{eq:DefInitialDataRescaled}, (ii) above implies $\limsup_{n \rightarrow \infty}(\|u_{0,n}\|_{H^1} + \|v_{0,n}\|_{L^2}) < \infty$. In particular, there exists $r > 0$ such that
\begin{equation}
	\label{eq:BoundInitialDataRescaled}
	\|u_{0,n}\|_{H^1} + \|v_{0,n}\|_{L^2} \leq r
\end{equation}
for all $n \in \N$.

Note that $\lim_{n \rightarrow \infty} \tau_n = 0$ since $\tau^* < \infty$. We further assume without loss of generality that $T - \sigma_n < 1$ for all $n \in \N$. Because of $T - \sigma_n \geq T - (\tau^* - \tau_{n+1}) > \tau_{n+1}$, we infer that
\begin{align}
\label{eq:StoppingTimeBlowUp}
	\tau_{n+1} &= \inf\{t \in [0, T] \colon  \|e^{\imu (\cdot) \Delta} u_{0,n}\|_{L^2_t L^4_x([0,t] \times \R^4)} + W_{\sigma_n}^*(t) \geq 2 \delta_*(\|u_{0,n}\|_{H^1}, \|v_{0,n}\|_{L^2})\} \nonumber\\
	&\qquad \wedge \inf\Big\{t \in [0,T] \colon \|v_{L,n}\|_{\Y([0,t]) + L^2_t W^{1,4}_x([0,t] \times \R^4)} \geq \frac{\epsilon}{2}\Big\}
\end{align}
for all $n \in \N$.
Recall that $\delta_*$ is decreasing in both its arguments. Setting $\delta = \delta_*(r,r) > 0$, we obtain from~\eqref{eq:BoundInitialDataRescaled} that
\begin{align*}
	\delta_*(\|u_{0,n}\|_{H^1}, \|v_{0,n}\|_{L^2}) \geq \delta_*(r,r) = \delta >0
\end{align*}
for all $n \in \N$. By continuity (where we use Lemma~\ref{lem:Continuity}~\ref{it:ContSumY} again) and~\eqref{eq:StoppingTimeBlowUp}, we thus get
\begin{align}
\label{eq:AlternativeTaunplus1}
	\|e^{\imu (\cdot) \Delta} u_{0,n}\|_{L^2_t L^4_x([0,\tau_{n+1}] \times \R^4)}  \geq \delta \quad \text{or} \quad W_{\sigma_n}^*(\tau_{n+1}) \geq \delta  \quad \text{or} \quad \|v_{L,n}\|_{\Y([0,\tau_{n+1}]) + L^2_t W^{1,4}_x([0,\tau_{n+1}] \times \R^4)} \geq \frac{\epsilon}{2}
\end{align}
for all $n \in \N$.

We next demonstrate that the second alternative in~\eqref{eq:AlternativeTaunplus1} is never satisfied if $n \in \N$ is large enough. To that purpose we show that
\begin{equation}
	\label{eq:WStarsigmanLimZero}
	\lim_{n \rightarrow \infty} \sup_{t \in [0, \tau^* - \sigma_n]} W_{\sigma_n}^*(t) = 0.
\end{equation} 

To prove this claim, we exploit Lemma~\ref{lem:PropNoise}. Let $\zeta > 0$.
Since the convergence~\eqref{eq:limphi1kbeta1k} holds for the $\omega$ we fixed,
there exists an index $n_l \in \N$ such that
\begin{align*}
    \sum\limits_{j=1}^4 \sum_{k=n_l+1}^\infty \int  \sup_{y\in \R^3}
    |\nabla \phi^{(1)}_k(r \vece_j + y)| \dd r \sup\limits_{t\in [0,T]} |\beta^{(1)}_k(t)|
    < \frac{\zeta}{4}.
\end{align*}
We now fix $\kappa \in (0,\frac{1}{2})$. The $C^\kappa$-H\"older continuity of Brownian motions yields for the first $n_l$ modes of the noise
\begin{align*}
     & \sum_{j=1}^4 \sum_{k=1}^{n_l}
    \int  \sup_{y\in \R^3} |\nabla \phi^{(1)}_k(r \vece_j+y)| \dd r
    \sup_{s \in [0,\tau^* - \sigma_n]}|\beta^{(1)}_k(\sigma_n+s) - \beta^{(1)}_k(\sigma_n)|  \\
   &\leq \sum_{j=1}^4 \sum_{k=1}^{n_l}
    \int  \sup_{y\in \R^3} |\nabla \phi^{(1)}_k(r \vece_j+y)| \dd r\,
          \tilde{C}(k,\kappa,T) (\tau^* - \sigma_n)^{\kappa}
    =: \tilde{C}(n_l,\kappa, T)  (\tau^* - \sigma_n)^{\kappa},
\end{align*}
where $\tilde{C}(k,\kappa,T) $ is the $C^\kappa$-H\"older norm of $\beta^{(1)}_k$ on $[0,T]$ for $1\leq k\leq n_l$.
Consequently, we have
\begin{align*}
   & \sum\limits_{j=1}^4 \sum_{k=1}^\infty \int \sup_{y \in \R^3}
        |\nabla \phi^{(1)}_k(r \vece_j+y)| \dd r \sup_{s\in[0, \tau^* - \sigma_n]}  |\beta^{(1)}_k(\sigma_n+s) - \beta^{(1)}_k(\sigma_n)|  \leq \frac{\zeta}{2} + \tilde{C}(n_l, \kappa, T)  (\tau^* - \sigma_n)^{\kappa}.
\end{align*}
The H\"older continuity of the noise provided by Lemma~\ref{lem:PropNoise}
also shows that
there exists $\tilde{C}'(\kappa,T)$ such that
\begin{align*}
    &\|\nabla W_{1,\sigma_n}(t) \|_{H^3}
     + \|W_{1,\sigma_n}(t) \|_{{H^4}}^2 + \|W_{1,\sigma_n}(t) \|_{{H^4}}
  + \|\cT_{\sigma_n+t, \sigma_n}(W_2)\|_{H^2}  \\
   &\leq   2 \|W_1(\sigma_{n}+t) - W_1(\sigma_n)\|_{{H^4}}
     + \| W_1(\sigma_{n} + t) -  W_1(\sigma_n)\|_{{H^4}}^2  \\
    &\qquad + \Big \|\int_{\sigma_n}^{\sigma_{n} + t} e^{\imu(\sigma_n + t -s)|\nabla|} \dd W_2(s) \Big\|_{H^2}   \\
  &\leq \tilde{C}'(\kappa,T) (\tau^* - \sigma_n)^\kappa
\end{align*}
for all $t \in [0, \tau^* - \sigma_n]$.
The last two estimates and the definition of $W^*_{\sigma_n}$ in \eqref{eq:DefWStarsigman} now yield
\begin{align*}
   \sup_{t \in [0, \tau^* - \sigma_n]} W^*_{\sigma_n}(t)
   \leq  \frac{\zeta}{2}
           + (\tilde{C}(n_l, \kappa, T)   +  \tilde{C}'(\kappa,T)) (\tau^* - \sigma_n)^\kappa.
\end{align*}
Since $\sigma_{n} \rightarrow \tau^*$ as $n \rightarrow \infty$, we conclude that there exists $N \in \N$ such that
\begin{equation}
	\label{eq:BoundWsigmanStar}
	\sup_{t \in [0, \tau^* - \sigma_n]} W_{\sigma_n}^*(t) < \zeta
\end{equation}
for all $n \geq N$, which implies~\eqref{eq:WStarsigmanLimZero}.
As $\tau_{n+1} \in [0, \tau^* - \sigma_n)$ for all $n \in \N$, we particularly find an index $n_0 \in \N$ such that $W^*_{\sigma_n}(\tau_{n+1}) < \delta$ for all $n \geq n_0$.

We next show that the first alternative in~\eqref{eq:AlternativeTaunplus1} is not satisfied for large enough $n$, i.e. that
\begin{equation}
	\label{eq:ExcludingBlowUpSchroedingerProfile}
	\|e^{\imu(\cdot)\Delta} u_{0,n}\|_{L^2_t L^4_x([0,\tau_{n+1}] \times \R^4)} < \delta
\end{equation}
if $n$ is large enough.

To that purpose, we define
\begin{align*}
	(u_{\sigma_n}(t), v_{\sigma_n}(t)) = (e^{W_1(\sigma_n)} u(\sigma_n + t), v(\sigma_n + t) + e^{\imu t |\nabla|} \cT_{\sigma_n}(W_2))
\end{align*}
for all $t \in [0, \tau^* - \sigma_n)$. By Proposition~\ref{prop:RefinedRescaling},
$(u_{\sigma_n}, v_{\sigma_n})$ solves~\eqref{eq:RanZakbsigmacsigmaInProof} on $[0, \tau^* - \sigma_n)$ with initial data $(u_{0,n}, v_{0,n})$. We further note that there exists a number $R > 0$ such that
\begin{align}
	\|u_{\sigma_n}\|_{L^\infty_t H^1_x([0,\tau^* - \sigma_n) \times \R^4)} &\leq \|e^{W_1(\sigma_n)} u\|_{L^\infty_t H^1_x([0,\tau^*) \times \R^4)} = \|e^{W_1(\sigma_n)} e^{-W_1} X\|_{L^\infty_t H^1_x([0,\tau^*) \times \R^4)} \nonumber \\
	&\leq C \| X \|_{L^\infty_t H^1_x([0,\tau^*) \times \R^4)} \leq R, \nonumber\\
	\|v_{\sigma_n}\|_{L^\infty_t L^2_x([0,\tau^* - \sigma_n) \times \R^4)} &\leq \|v\|_{L^\infty_t L^2_x([0,\tau^*) \times \R^4)} + \|\cT_{\sigma_n}\|_{L^2_x} \nonumber\\
	&\leq \|Y\|_{L^\infty_t L^2_x([0,\tau^*) \times \R^4)} + 2 \| \cT_{\cdot}(W_2)\|_{L^\infty_t L^2_x([0,\tau^*] \times \R^4)} \leq R, \label{eq:BlowUpBoundv}
\end{align}
for all $n \in \N$, where we employed~(ii) as well as $W_1 \in C([0,\tau^*], H^4(\R^4))$ and $\cT_{\cdot}(W_2) \in C([0,\tau^*], L^2(\R^4))$.

Using now that $(u_{\sigma_n}, v_{\sigma_n})$ solves~\eqref{eq:RanZakbsigmacsigmaInProof} on $[0, \tau^* - \sigma_n)$, we infer
\begin{align*}
	u_{\sigma_n}(t) = e^{\imu t \Delta} u_{0,n} - \imu \int_0^t e^{\imu (t-s) \Delta}(\Re(v_{\sigma_n}) u_{\sigma_n} - b_{\sigma_n} \cdot \nabla u_{\sigma_n} - c_{\sigma_n} u_{\sigma_n} + \Re(\cT_{\sigma_n + \cdot, \sigma_n}) u_{\sigma_n})(s) \dd s
\end{align*}
and thus
\begin{align*}
	&\|e^{\imu t \Delta} u_{0,n}\|_{L^2_t L^4_x([0,\tau_{n+1}] \times \R^4)} \leq \|u_{\sigma_n}\|_{L^2_t L^4_x([0,\tau^* - \sigma_n) \times \R^4)} \\
	&\qquad + \Big\| \int_0^t e^{\imu (t-s) \Delta}(\Re(v_{\sigma_n}) u_{\sigma_n} - b_{\sigma_n} \cdot \nabla u_{\sigma_n} - c_{\sigma_n} u_{\sigma_n} + \Re(\cT_{\sigma_n + \cdot, \sigma_n}) u_{\sigma_n})(s) \dd s \Big\|_{L^2_t L^4_x([0,\tau^* - \sigma_n) \times \R^4)} \\
	&\lesssim \|u\|_{L^2_t L^4_x([\sigma_n, \tau^*) \times \R^4)} + \|\Re(v_{\sigma_n}) u_{\sigma_n}\|_{L^2_t L^{\frac{4}{3}}_x([0, \tau^* - \sigma_n) \times \R^4)} \\
	&\qquad + \| b_{\sigma_n} \cdot \nabla u_{\sigma_n} + c_{\sigma_n} u_{\sigma_n} - \Re(\cT_{\sigma_n + \cdot, \sigma_n}) u_{\sigma_n} \|_{L^1_t L^2_x([0, \tau^* - \sigma_n) \times \R^4)} \\
	&\lesssim \|X\|_{L^2_t L^4_x([\sigma_n, \tau^*) \times \R^4)} + \|v_{\sigma_n}\|_{L^\infty_t L^2_x([0, \tau^* - \sigma_n) \times \R^4)} \| u_{\sigma_n}\|_{L^2_t L^{4}_x([0, \tau^* - \sigma_n) \times \R^4)} \\
	&\qquad + (\tau^* - \sigma_n) (\| b_{\sigma_n} \|_{L^\infty_t H^3_x([0, \tau^* - \sigma_n) \times \R^4)} + \|c_{\sigma_n}\|_{L^\infty_t H^1_x([0, \tau^* - \sigma_n) \times \R^4)} \\
	&\hspace{10em} + \|\cT_{\sigma_n + \cdot, \sigma_n}\|_{L^\infty_t H^1_x([0, \tau^* - \sigma_n) \times \R^4)}) \| u_{\sigma_n} \|_{L^\infty_t H^1_x([0, \tau^* - \sigma_n) \times \R^4)} \\
	&\lesssim (1+ R) \|X\|_{L^2_t L^4_x([\sigma_n, \tau^*) \times \R^4)} + (\tau^* - \sigma_n) R \sup_{t \in [0, \tau^* - \sigma_n]}W^*_{\sigma_n}(t)
\end{align*}
for all $n \in \N$, where we used Strichartz estimates as well as $|e^{\pm W_1}| = 1$. In~\eqref{eq:BoundWsigmanStar} we have seen that $\sup_{t \in [0, \tau^* - \sigma_n]}W^*_{\sigma_n}(t) \leq 1$ for all $n \geq N$. By assumption~(iii) and the dominated convergence theorem, we thus conclude
\begin{align*}
	\|e^{\imu t \Delta} u_{0,n}\|_{L^2_t L^4_x([0,\tau_{n+1}] \times \R^4)} \longrightarrow 0
\end{align*}
as $n \rightarrow \infty$. In particular, there exists $n_0 \in \N$ such that~\eqref{eq:ExcludingBlowUpSchroedingerProfile} is satisfied for all $n \geq n_0$.

Finally, we show that the third alternative in~\eqref{eq:AlternativeTaunplus1} cannot hold, i.e., we show that
\begin{equation}
\label{eq:BlowUpAltLinWavePotSmall}
	\|v_{L,n}\|_{\Y([0,\tau_{n+1}]) + L^2_t W^{1,4}_x([0,\tau_{n+1}] \times \R^4)} < \frac{\epsilon}{2}
\end{equation}
if $n \in \N$ is large enough. To that purpose we first recall that
\begin{align*}
	v_{L,n}(t) = e^{\imu t |\nabla|} v_{0,n} = e^{\imu t |\nabla|}(v_n(\sigma_n) + \cT_{\sigma_n}(W_2)) = e^{\imu t |\nabla|}(v(\sigma_n) + \cT_{\sigma_n}(W_2))
\end{align*}
from the construction in Subsection \ref{Subsec-Max-Exist}. By Sobolev's embedding we get for the second summand
\begin{align}
	&\|e^{\imu t |\nabla|}\cT_{\sigma_n}(W_2)\|_{\Y([0,\tau_{n+1}]) + L^2_t W^{1,4}_x([0,\tau_{n+1}] \times \R^4)} \lesssim \|e^{\imu t |\nabla|}\cT_{\sigma_n}(W_2)\|_{L^2_t W^{1,4}_x([0,\tau_{n+1}] \times \R^4)} \nonumber\\
	&\lesssim \|e^{\imu t |\nabla|}\cT_{\sigma_n}(W_2)\|_{L^2_t H^2_x([0,\tau_{n+1}] \times \R^4)} \lesssim \tau_{n+1}^{\frac{1}{2}} \|\cT_{\sigma_n}(W_2)\|_{H^2_x}
	\lesssim (\tau^* - \sigma_n)^{\frac{1}{2}} \sup_{t \in [0,\tau*]} \|\cT_t(W_2)\|_{H^2}. \label{eq:BlowUpWaveProfileSecondSummand}
\end{align}

For the first summand, we exploit Lemma~\ref{lem:ImprovementRegularity}. Note that by~(ii) and~(iii) the assumptions of that lemma are satisfied so that we can extend $v$ to a function in $C([0,\tau^*], L^2(\R^4))$. Now let $\phi \in C_c^\infty(\R^4)$ with $0 \leq \phi \leq 1$, $\phi = 1$ on $B_1(0)$ and $\phi = 0$ on $B_2(0)^c$. 
Setting $\phi_\nu(x) = \nu^{-4} \phi(\frac{x}{\nu})$ for $x \in \R^4$ and $\nu > 0$, we obtain the kernel of a standard mollifier on $\R^4$. Since $v$ is continuous on the compact interval $[0,\tau^*]$, we have
\begin{align*}
	\|v - v \ast \phi_\nu\|_{L^\infty_t L^2_x([0,\tau^*] \times \R^4)} \longrightarrow 0
\end{align*}
as $\nu \rightarrow 0$. Hence, we can fix $\nu > 0$ such that $\|v - v \ast \phi_\nu\|_{L^\infty_t L^2_x([0,\tau^*] \times \R^4)} < \frac{\epsilon}{4C'}$, where $C'$ is the constant from Lemma~\ref{lem:LinearEstimateHalfWave}. Using Lemma~\ref{lem:LinearEstimateHalfWave} and Sobolev's embedding again as in~\eqref{eq:BlowUpWaveProfileSecondSummand}, we thus infer
\begin{align}
	\label{eq:BlowUpWaveProfileFirstSummand}
	&\|e^{\imu t |\nabla|}v(\sigma_n)\|_{\Y([0,\tau_{n+1}]) + L^2_t W^{1,4}_x([0,\tau_{n+1}] \times \R^4)} \nonumber\\
	& \leq \|e^{\imu t |\nabla|}(v(\sigma_n) - v \ast \phi_\nu(\sigma_n))\|_{\Y([0,\tau_{n+1}])} + \|e^{\imu t |\nabla|}(v \ast \phi_\nu(\sigma_n))\|_{L^2_t W^{1,4}_x([0,\tau_{n+1}] \times \R^4)} \nonumber \\
	&\leq C'\|v(\sigma_n) - v \ast \phi_\nu(\sigma_n)\|_{L^2_x(\R^4)} + C\|e^{\imu t |\nabla|}(v(\sigma_n) \ast \phi_\nu)\|_{L^2_t H^2_x([0,\tau_{n+1}] \times \R^4)} \nonumber \\
	&\leq C'\|v - v \ast \phi_\nu\|_{L^\infty_t L^2_x([0,\tau^*] \times \R^4)} + C\tau_{n+1}^{\frac{1}{2}} \|v(\sigma_n) \ast \phi_\nu\|_{H^2_x(\R^4)}  \nonumber \\ 
	&\leq \frac{\epsilon}{4} + C(\tau^* - \sigma_n)^{\frac{1}{2}} \nu^{-2} \|v(\sigma_n)\|_{L^2_x(\R^4)} \nonumber \\
	&\leq \frac{\epsilon}{4} + C(\tau^* - \sigma_n)^{\frac{1}{2}} \nu^{-2} R,
\end{align}
where we also employed~\eqref{eq:BlowUpBoundv} in the last step. 

Combining~\eqref{eq:BlowUpWaveProfileFirstSummand} and~\eqref{eq:BlowUpWaveProfileSecondSummand} and using that $\lim_{n \rightarrow \infty} \sigma_n = \tau^*$, we conclude that there is $n_0 \in \N$ such that~\eqref{eq:BlowUpAltLinWavePotSmall} is satisfied for all $n \geq n_0$. 
Finally, \eqref{eq:WStarsigmanLimZero}, \eqref{eq:ExcludingBlowUpSchroedingerProfile}, and~\eqref{eq:BlowUpAltLinWavePotSmall} contradict~\eqref{eq:AlternativeTaunplus1} and thus the blow-up alternative in Theorem~\ref{thm:LocalWP} holds true. \hfill \qed

\section{GWP below the ground state}
\label{sec:WellPoBelowGrSt}

This section is devoted to the proof of the global well-posedness
below the ground state.
Two crucial ingredients of this proof are the variational properties of the ground state and a uniform estimate for solutions of a Schrödinger equation with a free-wave potential and lower order perturbations in the adapted space $\SHalf(I)$. By uniform we mean in this context that the involved constant does not depend on the free-wave profile, but only on its $L^2$-norm.

We first recall some consequences of the variational properties of the ground state $W$. These properties have been studied in~\cite{GNW13} and have been further developed in~\cite{GN21}. We exploit them in the form of Lemma~7.3 in~\cite{CHN21}.

\begin{lemma}[Variational constraints below the ground state, {\cite[Lemma~7.3]{CHN21}}]
	\label{lem:VariationalConstraints}
	Let $f \in \dot{H}^1(\R^4)$ and $g \in L^2(\R^4)$ with
	\begin{align*}
		e_Z(f,g) < e_Z(W, -W^2) = \frac{1}{4} \|W^2\|_{L^2_x}^2, \qquad \|g\|_{L^2_x} \leq \|W^2\|_{L^2_x}.
	\end{align*}
	We then have
	\begin{align*}
		\|g\|_{L^2_x}^2 \leq 4 e_Z(f,g), \qquad \| \nabla f\|_{L^2_x}^2 \leq \frac{1}{2} \frac{\|W^2\|_{L^2_x}}{\|W^2\|_{L^2_x} - \|g\|_{L^2_x}}(4 e_Z(f,g) - \|g\|_{L^2_x}^2) \leq \|W^2\|_{L^2_x}^2.
	\end{align*}
\end{lemma}

The following result gives the uniform estimate in the case of lower order perturbations.

\begin{proposition} [Uniform estimates]
	\label{prop:UniformEstimateBelowGroundState}
	Let $I$ be an interval with $t_0 = \min I$, $0 < B < \|W^2\|_{L^2(\R^4)}$, and $v_0 \in L^2(\R^4)$ with $\|v_0\|_{L^2(\R^4)} \leq B$. Let $u_0\in H^{1}(\R^4)$ and $f \in \NHalf(I)$.
Let $u\in C(I, H^1(\R^4))$
solve the equation
\begin{align*}
   \imu \partial_t u + \Delta u  - \Re(v_L) u + b\cdot \na u + c u - \Re(\cT_t(W_2)) u = f,
\end{align*}
with initial condition $u(t_0) = u_0$,
where $v_L:= e^{\imu (t-t_0)|\na|} v_0$ . Assume that there is a constant $A > 0$ such that
	\begin{align*}
		\|u \|_{L^\infty_t H^1_x(I \times \R^4)}  + a^*(I) \leq A,
	\end{align*}		
	where
\begin{align*}
	a^*(I) :=  \sum_{j = 1}^4 \|b\|_{L^{1,\infty}_{\vece_j}(I \times \R^4)} + \|b\|_{L^\infty_t H^3_x(I \times \R^4)} + \|c\|_{L^\infty_t H^2_x(I \times \R^4)} + \|\Re(\cT_t(W_2))\|_{L^\infty_t H^2_x(I \times \R^4)}.
\end{align*}
	 Then there is a constant $C = C(A,B) > 0$ such that
	\begin{align}  \label{unif-esti}
		\|u\|_{\SHalf(I)}
     \leq C \|u_0\|_{H^{\frac{1}{2}}_x(\R^4)} +  C\|f\|_{\NHalf(I)} + C a^*(I) |I|^{\frac{1}{2}}.
	\end{align}
\end{proposition}

\begin{remark}
	\label{rem:SubNorm}
	In Theorem~6.1 in~\cite{CHN21} the norm $\|\cdot\|_{\underline{S}^\frac{1}{2}(I)}$ is used. This norm is stronger than our $\SHalf$-norm, i.e. $\|u\|_{\SHalf(I)} \leq \|u\|_{\underline{S}^{\frac{1}{2}}(I)}$, see~\cite[Lemma~2.1]{CHN21}.
\end{remark}

\begin{proof}
We assume $\|u\|_{\SHalf(I)} < \infty$ in the following.
	We rewrite
	\begin{align}
		\label{eq:uInUniformEstimate}
			u(t) &= \cU_{v_L}[u_0](t) + \cI_{v_L}[f](t) - \cI_{v_L}[b \cdot \nabla u + c u - \Re(\cT_t(W_2)) u](t)  \nonumber \\
			 &= \cU_{v_L}[u_0](t) - \cI_{v_L}[-f + (b \cdot \nabla u)_{HL + HH} + c u - \Re(\cT_t(W_2)) u](t) - \cI_{v_L}[(b \cdot \nabla u)_{LH} ](t).
	\end{align}
	Set $F= (b \cdot \nabla u)_{LH}$. The key point is to prove a uniform estimate for the lower order perturbation term $\cI_{v_L}[F]$. For the other two terms in~\eqref{eq:uInUniformEstimate} we can directly apply the uniform Strichartz estimate from~\cite[Theorem~6.1]{CHN21}.
	
	We first note that a simple computation shows
	\begin{equation}
		\label{eq:IdDuhamelWithPotential}
		\cI_{v_L} = [I + \cI_{v_L} \Re(v_L)] \cI_0,
	\end{equation}
	see~(6.9) in~\cite{CHN21}. Consequently, we have
	\begin{align}
	\label{eq:SplittingIvLF}
		\| \cI_{v_L}[F] \|_{\SHalf(I)} \leq \|\cI_0 F\|_{\SHalf(I)} + \| \cI_{v_L}[\Re(v_L)\cI_0[F]]\|_{\SHalf(I)}.
	\end{align}
	For the first summand we apply Lemma~\ref{lem:LinEstimates} to deduce
	\begin{align}
	\label{eq:S12LSInProof}
		\|\cI_0[F]\|_{\SHalf(I)} \lesssim \|F\|_{\G^{\frac{1}{2}}(I)} \lesssim \sum_{j = 1}^4 \Big(\sum_{\lambda \in 2^\N} \|F_\lambda\|_{L^{1,2}_{\vece_j}}^2 \Big)^{\frac{1}{2}},
	\end{align}
 where we also exploited that $P_1 F = 0$.
	For the second term on the right-hand side of~\eqref{eq:SplittingIvLF}, Theorem~6.1 in~\citep{CHN21} (see also Remark~\ref{rem:SubNorm}) yields
	\begin{align*}
		\| \cI_{v_L}[\Re(v_L)\cI_0[F]]\|_{\SHalf(I)} \lesssim_B \|\Re(v_L) \cI_0[F]\|_{\NHalf(I)}.
	\end{align*}
	Applying Lemmas~\ref{lem:LinearEstimateHalfWave} 
 and \ref{lem:BilinearEstimates}, we further deduce
	\begin{align*}
		\|\Re(v_L) \cI_0[F]\|_{\NHalf(I)} &\lesssim \|\Re(v_L)\|_{W^{0,0,0}(I)} \|\cI_0[F]\|_{\SHalf(I)}  \notag \\
        &\lesssim \|v_0\|_{L^2_x(\R^4)} \|\cI_0[F]\|_{\SHalf(I)}  \\
		&\lesssim B \sum_{j = 1}^4 \Big(\sum_{\lambda \in 2^\N} \|F_\lambda\|_{L^{1,2}_{\vece_j}}^2 \Big)^{\frac{1}{2}},
	\end{align*}
	where we also used~\eqref{eq:S12LSInProof} in the last step. Combining the last two estimates with~\eqref{eq:S12LSInProof} and~\eqref{eq:SplittingIvLF} 
 we obtain 
	\begin{equation}
		\label{eq:EstIvLF}
		\| \cI_{v_L}[F] \|_{\SHalf(I)} \lesssim_B \sum_{j = 1}^4 \Big(\sum_{\lambda \in 2^\N} \|F_\lambda\|_{L^{1,2}_{\vece_j}}^2 \Big)^{\frac{1}{2}}.
	\end{equation}
	For the above right-hand side we estimate
	\begin{align}
		\Big(\sum_{\lambda \in 2^\N} \|F_\lambda\|_{L^{1,2}_{\vece_j}}^2 \Big)^{\frac{1}{2}}
		&\lesssim \Big(\sum_{\lambda \in 2^\N} \|P_{\leq \frac{\lambda}{2^8}} b P_\lambda \nabla u\|_{L^{1,2}_{\vece_j}}^2 \Big)^{\frac{1}{2}} \nonumber\\
		&\lesssim \Big(\sum_{\lambda \in 2^\N} \||P_{\leq \frac{\lambda}{2^8}} b |^{\frac{1}{2}} \|_{L^{2,\infty}_{\vece_j}}^2 \||P_{\leq \frac{\lambda}{2^8}} b |^{\frac{1}{2}} |P_\lambda \nabla u |\|_{L^2_{t,x}}^2 \Big)^{\frac{1}{2}} \nonumber\\
		&\lesssim \|b\|_{L^{1,\infty}_{\vece_j}}^{\frac{1}{2}} \|b\|_{L^\infty_{t,x}}^{\frac{1}{2}} \| \nabla u\|_{L^2_{t,x}} \nonumber \\
		&\lesssim |I|^{\frac{1}{2}} \|b\|_{L^{1,\infty}_{\vece_j}}^{\frac{1}{2}} \|b\|_{L^\infty_t H^3_x}^{\frac{1}{2}} \| \nabla u\|_{L^\infty_t L^2_x} 
  \lesssim_A a^*(I) |I|^{\frac{1}{2}} \label{eq:EstbnuLHLS}
	\end{align}
	for every $j \in \{1, \ldots, 4\}$, which finally shows
	\begin{equation}
		\label{eq:EstIvLbnabuUniform}
		\| \cI_{v_L}[F] \|_{\SHalf(I)} \lesssim_{A,B} a^*(I) |I|^{\frac{1}{2}}.
	\end{equation}

	We next turn to the second term in~\eqref{eq:uInUniformEstimate}. Here we simply apply Theorem~6.1 from~\cite{CHN21} to infer
	\begin{align}
		&\| \cI_{v_L}[-f + (b \cdot \nabla u)_{HL + HH} + c u - \Re(\cT_t(W_2)) u] \|_{\SHalf(I)} \nonumber\\
		&\lesssim_B \|f\|_{\NHalf(I)} + \|(b \cdot \nabla u)_{HL + HH} + c u - \Re(\cT_t(W_2)) u\|_{\NHalf(I)}. \label{eq:EstIvLRest}
	\end{align}
	Lemma~\ref{lem:BilinearLowerOrder} implies for the remaining components of the lower order perturbation
	\begin{align}
		\label{eq:UniformEstLowerOrderN12}
		&\|(b \cdot \nabla u)_{HL + HH} + c u - \Re(\cT_t(W_2)) u\|_{\NHalf(I)} \nonumber \\
		&\lesssim |I|^{\frac{1}{2}}(\|b\|_{L^\infty_t H^2_x} + \|c\|_{L^\infty_t H^2_x} + \|\Re(\cT_t(W_2))\|_{L^\infty_t H^2_x}) \|u\|_{L^\infty_t H^1_x} \lesssim_A a^*(I) |I|^{\frac{1}{2}}.
	\end{align}

	In order to estimate the linear propagator in~\eqref{eq:uInUniformEstimate}, we employ the identity
	\begin{align*}
		\cU_{v_L}[u_0](t) = \cI_{v_L}[\Re(v_L) e^{\imu (\cdot - t_0) \Delta} u_0](t) +  e^{\imu (t-t_0) \Delta} u_0.
	\end{align*}
	Another application of Theorem~6.1 from~\cite{CHN23} as well as the energy estimates from Lemmas~\ref{lem:LinEstimates}, \ref{lem:LinearEstimateHalfWave} 
 and \ref{lem:BilinearEstimates} yield
	\begin{align}
		\label{eq:UniformEstLinear}
		\|\cU_{v_L}[u_0]\|_{\SHalf(I)} &\lesssim_B \|\Re(v_L) e^{\imu (\cdot - t_0) \Delta} u_0\|_{\NHalf} + \|e^{\imu (\cdot - t_0) \Delta} u_0\|_{\SHalf(I)} \nonumber\\
		&\lesssim_B (1 + \|v_L\|_{W^{0,0,0}})\|e^{\imu (\cdot - t_0) \Delta} u_0\|_{\SHalf(I)} \notag \\
        &\lesssim_B (1 + \|v_0\|_{L^2_x(\R^4)}) \|u_0\|_{H^{\frac{1}{2}}_x(\R^4)} \lesssim_B \|u_0\|_{H^{\frac{1}{2}}_x(\R^4)}.
	\end{align}
	The combination of estimates~\eqref{eq:EstIvLbnabuUniform} to~\eqref{eq:UniformEstLinear} yields the assertion of the lemma.	
\end{proof}

We are now in position to prove the global well-posedness result
below the ground state.
An important fact is that controlling the endpoint critical $L^2_tW^{\frac 12, 4}$-norm only needs $\frac{1}{2}$ derivative while our solution $u$ belongs to $H^1(\R^4)$. We will exploit this observation via the estimate
\begin{align}
	\label{eq:EstProfileControllingNorm}
	\|e^{\imu t \Delta} u_0\|_{L^2_t W^{\frac{1}{2},4}_x(I \times \R^4)} \lesssim |I|^{\frac{1}{4}}\|e^{\imu t \Delta} u_0\|_{L^4_t W^{1,\frac{8}{3}}_x(I \times \R^4)} \lesssim |I|^{\frac{1}{4}} \|u_0\|_{H^1_x(\R^4)},
\end{align}
which follows from Sobolev's embedding and Strichartz estimates.

\medskip
\paragraph{\bf Proof of Theorem \ref{Thm-GWP-Ground}}

Let $(X,Y)$ be the unique maximal solution of~\eqref{eq:StoZak} on the maximal interval of existence $[0,\tau^*)$ provided by Theorem~\ref{thm:LocalWP}. We have to prove that $\sigma^* \leq \tau^*$ $\bbp$-a.s. We argue by contradiction and assume that $\bbp(\{\sigma^* > \tau^*\}) > 0$. We fix an element $\omega \in \{\sigma^* > \tau^*\}$ in the following but do not denote the dependence of the considered quantities on $\omega$ for the ease of notation. Note that in particular $\tau^* < \infty$.
By the definition of $\sigma^*$, there exists an $n \in \N$ such that $\tau^* < \sigma_n^*$. We fix such an $n$ in the following.

Setting $u := e^{-W_1} X$ and $v := Y - \cT_t(W_2)$, Theorem~\ref{thm:Rescaling} shows that $(u,v)$ solves~\eqref{eq:RanZakbc} on $[0,\tau^*)$. Moreover, the definition of $\sigma_n^*$ implies that
\begin{equation}
    \label{eq:EnergyBdStoppingTime}
    e_Z(u(t), v(t))  < e_Z(W, W^2) - \frac{1}{n} \qquad \text{for all } t \in [0, \tau^*).
\end{equation}
We define
\begin{align*}
    B := \max\left\{\Big(\|W^2\|_{L^2_x}^2 - \frac{4}{n}\Big)^{\frac{1}{2}}, \| Y_0 \|_{L^2_x}\right\}  < \| W^2 \|_{L^2_x}.
\end{align*}
Lemma~\ref{lem:VariationalConstraints} and a continuity argument yield that
\begin{align*}
    \| v(t) \|_{L^2_x} \leq B \qquad \text{for all } t \in [0,\tau^*).
\end{align*}
Hence, we can combine~\eqref{eq:EnergyBdStoppingTime} with Lemma~\ref{lem:VariationalConstraints} again to infer
\begin{align*}
    \| \nabla u(t) \|_{L^2_x} \leq \|W^2\|_{L^2_x} \qquad \text{for all } t \in [0,\tau^*),
\end{align*}
which implies
\begin{equation}
    \label{eq:H1BounduBelowGroundState}
    \| u \|_{L^\infty_t H^1_x([0,\tau^*) \times \R^4)} < \infty
\end{equation}
in view of the conservation of $\|u(t)\|_{L^2_x}$.
Using that $\tau^* < \infty$, we also have
\begin{align*}
    a^*([0,\tau^*)) &:=  \sum_{j = 1}^4 \|b\|_{L^{1,\infty}_{\vece_j}([0,\tau^*) \times \R^4)} + \|b\|_{L^\infty_t H^3_x([0,\tau^*) \times \R^4)} + \|c\|_{L^\infty_t H^2_x([0,\tau^*) \times \R^4)} \\
    &\qquad + \|\Re(\cT_t(W_2))\|_{L^\infty_t H^2_x([0,\tau^*) \times \R^4)}  < \infty,
\end{align*}
as well as
\begin{align*}
    \|X\|_{L^\infty_t H^1_x([0,\tau^*) \times \R^4)} \lesssim \|u\|_{L^\infty_t H^1_x([0,\tau^*) \times \R^4)}.
\end{align*}
Consequently, there is a constant $A > 0$ such that
\begin{equation}
    \label{eq:BounduH1aStar}
    \| u \|_{L^\infty_t H^1_x([0,\tau^*) \times \R^4)} + a^*([0,\tau^*)) \leq A
\end{equation}
and a constant $r > 0$ such that
\begin{align}
\label{eq:EnergyBoundMaxExTime}
    \|X\|_{L^\infty_t H^1_x([0,\tau^*) \times \R^4)} + \| Y \|_{L^\infty_t L^2_x([0,\tau^*) \times \R^4)} \leq r.
\end{align}
The blow-up alternative in Theorem~\ref{Thm-LWP} now implies that
\begin{equation}
	\label{eq:BlowupDispersiveNormX}
	\|X\|_{L^2([0,\tau^*), W^{\frac{1}{2},4}(\R^4))} = \infty.
\end{equation}
By standard product estimates, we also have
\begin{align*}
	\|X\|_{L^2_t W^{\frac{1}{2},4}_x([0,\tau^*) \times \R^4))} \lesssim (\|e^{W_1} - 1\|_{L^\infty_t H^3_x([0,\tau^*) \times \R^4))} + 1) \|u\|_{L^2_t W^{\frac{1}{2},4}_x([0,\tau^*) \times \R^4))},
\end{align*}
which leads to
\begin{equation}
	\label{eq:BlowupDispersiveNormu}
	\|u\|_{L^2([0,\tau^*), W^{\frac{1}{2},4}(\R^4))} = \infty.
\end{equation}
By the construction in the proof of Theorem \ref{Thm-LWP},
we further obtain a sequence of stopping times $(\tau_n)$
such that $\tau_n <\tau^*$, $\tau_n \to \tau^*$ as $n\to \infty$,
and
\begin{align}  \label{u-S1-taun}
   \|u\|_{\SOne ([0,\tau_n])} <\infty.
\end{align}

Let $R:= 4 A C(A,B) + 1$, where $C(A,B)$ is the constant from Proposition~\ref{prop:UniformEstimateBelowGroundState}. Choose $\epsilon, \sigma \in (0,1)$, depending only on $A$ and $B$, so small that
\begin{align*}
   C^2 C(A,B) \epsilon^{\frac{1}{2}} R^{\frac{3}{2}} \leq \frac{1}{8}, \qquad
   C(A,B) A \sigma^{\frac{1}{2}} + C \sigma^{\frac{1}{4}} A + C C'(A,B) A^{\frac{1}{2}} \sigma^{\frac{1}{8}} \leq \frac{\epsilon}{8},
\end{align*}
where $C$ is the maximum of the implicit constants in~\eqref{eq:BilinvuEndpoint}, \eqref{eq:BilinnablauwEndpoint}, and~\eqref{eq:EstProfileControllingNorm} and $C'(A,B)$ the constant arising in~\eqref{eq:BootstrapDHomCor} below. In particular, $\sigma$ is independent of $n$.

\medskip
\paragraph{\bf Claim:}
For any $n\geq 1$ and any $\tau < \tau^*$, we have
\begin{align}  \label{D-bdd-boot}
  \|u\|_{\SHalf([\tau_n, (\tau_n + \sigma)\wedge \tau))} \leq R,\ \
  \|u\|_{D([\tau_n, (\tau_n + \sigma)\wedge \tau))} \leq \ve.
\end{align}

We use a bootstrap argument to prove \eqref{D-bdd-boot}.
To that purpose, we first note that the claim holds on some interval $[\tau_n, t')$. To see this, we consider the extension
\begin{align*}
    \tilde{u}(t) = \one_{(-\infty, \tau_n)}(t) e^{\imu (t - \tau_n) \Delta} u(\tau_n) + \one_{[\tau_n, t']}(t) u(t) + \one_{(t', \infty)}(t) e^{\imu(t - t')\Delta} u(t').
\end{align*}
Using Strichartz estimates, we thus infer
\begin{align*}
    &\| u\|_{\SHalf([\tau_n, t'))} \leq \| \tilde{u} \|_{\SHalf(\R)} \leq \tilde{C} \| u \|_{L^\infty_t H^{\frac{1}{2}}_x([\tau_n, t'] \times \R^4)} \\
    & \quad + \Big( \sum_{\lambda \in 2^{\N_0}} (\lambda^{\frac{1}{2}} \| u_\lambda \|_{L^\infty_t L^2_x([\tau_n, t') \times \R^4)} + \lambda^{\frac{1}{2}} \| u_\lambda \|_{L^2_t L^4_x([\tau_n, t') \times \R^4)} + \lambda^{-\frac{1}{2}} \| (\imu \partial_t + \Delta) u_\lambda \|_{L^2_t L^2_x([\tau_n, t') \times \R^4)} )^2 \Big)^{\frac{1}{2}} \\
    &\leq \frac{R}{4} + \Big( \sum_{\lambda \in 2^{\N_0}} ( \lambda^{\frac{1}{2}} \| u_\lambda \|_{L^2_t L^4_x([\tau_n, t') \times \R^4)} + \lambda^{-\frac{1}{2}} \| (\imu \partial_t + \Delta) u_\lambda \|_{L^2_t L^2_x([\tau_n, t') \times \R^4)} )^2 \Big)^{\frac{1}{2}},
\end{align*}
where we assumed without loss of generality that $\tilde{C} + 2 \leq C(A,B)$. Dominated convergence thus implies that~\eqref{D-bdd-boot} is satisfied on $[\tau_n, t')$ if $t'$ is close enough to $\tau_n$.

Now assume that the claim is true on some subinterval $I:=[t_0,t_0'] \subseteq [\tau_n, (\tau_n + \sigma)\wedge \tau)$. Since $(u,v)$ solves~\eqref{eq:RanZakbc},
we have
\begin{align}
\label{eq:DuhameluPotProp}
   u(t)= \cU_{v_L}[u(t_0)](t)
         - \cI_{v_L}[\Re(\cJ_0[|\nabla| |u|^2])u
         + b \cdot \nabla u + c u - \Re(\cT_{\cdot}(W_2)) u](t),
\end{align}
where $v_L:= e^{\imu (t-t_0)|\na|} v(t_0)$ and the propagation operators $\cU_{v_L}$ and $\cI_{v_L}$ are used with initial time $t_0$.

Using the uniform estimate \eqref{unif-esti} in Proposition \ref{prop:UniformEstimateBelowGroundState},
we infer
\begin{align}
    \| u \|_{\SHalf(I)} &\leq C(A,B)(\|u(t_0)\|_{H^{\frac{1}{2}}_x} + \| \cJ_0[|\nabla| |u|^2]u \|_{\NHalf(I)} + a^*(I) |I|^{\frac{1}{2}}) \nonumber\\
    &\leq \frac{R}{4} + C(A,B) A \sigma^{\frac{1}{2}} + C(A,B) \| \cJ_0[|\nabla| |u|^2]u \|_{\NHalf(I)}. \label{eq:BootstrapS12}
\end{align}
We next apply the endpoint estimates \eqref{eq:BilinvuEndpoint}
and \eqref{eq:BilinnablauwEndpoint}
to infer
\begin{align}   \label{esti-uD-2}
   \|\Re(\cJ_0[|\nabla| |u|^2])u\|_{\NHalf(I)}
   &\leq  C \|\Re(\cJ_0[|\nabla| |u|^2])\|_{\WEndp(I)}
            \|u\|_{D(I)}^\frac 12  \|u\|_{\SHalf(I)}^\frac 12    \notag \\
   &\leq C^2 \|u\|_{D(I)}^\frac 32  \|u\|_{\SHalf(I)}^\frac 32.
\end{align}
Combining the previous two estimates, we arrive at
\begin{align}
    \label{eq:BootstrapS12Final}
    \| u \|_{\SHalf(I)} \leq \frac{R}{4} + C(A,B) A \sigma^{\frac{1}{2}} + C(A,B) C^2 \epsilon^{\frac{3}{2}} R^{\frac{3}{2}} \leq \frac{R}{2}.
\end{align}

To estimate the $D(I)$-norm of $u$, we begin with the homogeneous propagation operator. Recalling that 
\begin{align}
\label{eq:HomPropPot}
    \cU_{v_L}[u(t_0)](t) = e^{\imu (t - t_0) \Delta} u(t_0) + \cI_{v_L}[\Re(v_L)e^{\imu (t - t_0) \Delta} u(t_0)],
\end{align}
we obtain
\begin{align}
\label{eq:BootstrapDHomPot}
    \| \cU_{v_L}[u(t_0)] \|_{D(I)} \leq \| e^{\imu (t - t_0) \Delta} u(t_0) \|_{D(I)} + \| \cI_{v_L}[\Re(v_L)e^{\imu (t - t_0) \Delta} u(t_0)] \|_{\SHalf(I)}. 
\end{align}
For the second term on the right-hand side, we apply the uniform estimate from Theorem~6.1 in~\cite{CHN21} and the endpoint estimates~\eqref{eq:BilinvuEndpoint} and~\eqref{eq:BilinnablauwEndpoint} to infer
\begin{align}
    \| \cI_{v_L}[\Re(v_L)e^{\imu (t - t_0) \Delta} u(t_0)] \|_{\SHalf(I)} &\leq C(B) \|\Re(v_L)e^{\imu (t - t_0) \Delta} u(t_0) \|_{\NHalf} \nonumber\\
    &\leq C(B) C \| v_L \|_{\WEndp} \| e^{\imu (t - t_0) \Delta} u(t_0) \|_{D(I)}^{\frac{1}{2}} \| e^{\imu (t - t_0) \Delta} u(t_0) \|_{\SHalf(I)}^{\frac{1}{2}} \nonumber\\
    &\leq C(B) C \| v(t_0) \|_{L^2} \|u(t_0)\|_{H^1}^{\frac{1}{2}}  \| e^{\imu (t - t_0) \Delta} u(t_0) \|_{D(I)}^{\frac{1}{2}} \nonumber\\
    &\leq C'(A,B) \| e^{\imu (t - t_0) \Delta} u(t_0) \|_{D(I)}^{\frac{1}{2}} . \label{eq:BootstrapDHomCor}
\end{align}
For the inhomogeneous part in~\eqref{eq:DuhameluPotProp}, we use again that the $D(I)$-norm is controlled by the $\SHalf(I)$-norm and the estimates in~\eqref{eq:BootstrapS12} to~\eqref{eq:BootstrapS12Final}, which yields
\begin{align}
\label{eq:BootstrapDInhom}
    \| \cI_{v_L}[\Re(\cJ_0[|\nabla| |u|^2])u]
         + b \cdot \nabla u + c u - \Re(\cT_{\cdot}(W_2)) u] \|_{D(I)} \leq C(A,B) A \sigma^{\frac{1}{2}} + C(A,B) C^2 \epsilon^{\frac{3}{2}} R^{\frac{3}{2}} \leq \frac{\epsilon}{4}.
\end{align}
Combining~\eqref{eq:HomPropPot} to~\eqref{eq:BootstrapDInhom} and employing estimate~\eqref{eq:EstProfileControllingNorm}, we arrive at
\begin{align}
    \| u \|_{D(I)} &\leq \| e^{\imu (t - t_0) \Delta} u(t_0) \|_{D(I)} + C'(A,B) \| e^{\imu (t - t_0) \Delta} u(t_0) \|_{D(I)}^{\frac{1}{2}} + \frac{\epsilon}{4} \nonumber\\
    &\leq C \sigma^{\frac{1}{4}} \| u(t_0) \|_{H^1_x} + C C'(A,B) \sigma^{\frac{1}{8}} \| u(t_0) \|_{H^1_x}^{\frac{1}{2}} + \frac{\epsilon}{4} \leq \frac{\epsilon}{2}. \label{eq:BootstrapDFinal}
\end{align}
Lemma~\ref{lem:Continuity} and the fact that $u \in \SHalf(J)$ for every compact subinterval $J \subseteq [t_0,\tau^*)$ imply that $t \mapsto \| u \|_{\SHalf([\tau_n, t))}$ is continuous in every $t \in [\tau_n, \tau^*)$. Hence, \eqref{eq:BootstrapS12Final}, \eqref{eq:BootstrapDFinal} and a continuity argument imply \eqref{D-bdd-boot}, as claimed.

\medskip
Now, since $\sigma$ is independent of $n$,
we can take $\tau_n$ such that $\tau^*- \tau_n <\sigma$,
i.e., $\tau_n + \sigma > \tau^*$.
In view of \eqref{D-bdd-boot},
we then get
\begin{align*}
   \|u\|_{D([\tau_n, \tau))} \leq R
\end{align*}
for every $\tau < \tau^*$. The Lemma of Fatou thus yields
\begin{align*}
    \| u \|_{D([\tau_n, \tau^*))} \leq R.
\end{align*}
Since we have
\begin{align*}
   \|u\|_{D([0, \tau_n])}
   \leq C \|u\|_{\SOne([0, \tau_n])}
   <\infty
\end{align*}
by \eqref{u-S1-taun}, we arrive at
\begin{align*}
   \|u\|_{D([0, \tau^*))}<\infty,
\end{align*}
which however contradicts~\eqref{eq:BlowupDispersiveNormu}.

We thus conclude that $\sigma_* \leq \tau^*$, completing the proof of Theorem~\ref{Thm-GWP-Ground}.
\hfill $\square$

\section{GWP and scattering via regularization by noise} 
\label{Sec-Noise-Regular} 

In this section we prove the regularization by noise result in Theorem~\ref{thm:RegNoise}.

We set $c := \Im \phi_1^{(1)}$ to ease the notation in the following.  
Since $W_1$ is a one-dimensional Brownian motion, there are no lower order perturbations in~\eqref{eq:RanZakNoncons}. Consequently, 
the local smoothing component is not needed in our functional setting from Subsection~\ref{Subsec:FunctFrame}, i.e., we will solve~\eqref{eq:RanZakNoncons} in $S^{1, \frac{1}{4}} \times W^{0,\frac{1}{4},\frac{1}{2}}$.

In order to prove Theorem~\ref{thm:RegNoise} we shall show that the probability of the event
\begin{align}
	\Upsilon := \{\omega \in \Omega \colon &(z,v) \text{ solution of~\eqref{eq:RanZakNoncons} exists on } [0,\infty) \text{ and there is } (z_+, v_+) \in H^1 \times L^2 \text{ s.t. } \nonumber\\ 
    &\lim_{t \rightarrow \infty} \|e^{-\imu t \Delta} z(t) - z_+\|_{H^1} = 0 \text{ and } \lim_{t \rightarrow \infty} \|e^{-\imu t |\nabla|} v(t) - v_+\|_{L^2} = 0\} \label{eq:DefUpsilon}
\end{align}
converges to $1$ as $c \rightarrow \infty$. To that purpose, we have to substantiate
the heuristic expectation that the asymptotic exponential decay of the geometric Brownian motion stabilizes the system and facilitates to get global results. 
As metioned in Subsection \ref{Subsec-Novelties}, 
the key point is to derive a global-in-time $V^p$ control of geometric Brownian motions.

\subsection{Global-in-time $V^p$ control}
Let us first introduce the $V^p$ spaces. Define the set of partitions
\begin{align*}
    \cP = \{(t_j)_{j=1}^N \colon N \in \N, t_j \in \R, t_j < t_{j+1}\},
\end{align*}
i.e., a partition is a finite increasing sequence in $\R$. Let $1 \leq p < \infty$. For every function $v \colon \R \rightarrow \R$,  define the $p$-variation
\begin{align*}
    |v|_{V^p} = \sup_{(t_j)_{j = 1}^N \in \cP} \Big( \sum_{j = 1}^{N-1} |v(t_{j+1}) - v(t_j)|^p\Big)^{\frac{1}{p}}
\end{align*}
and the $V^p$-norm
\begin{align*}
	\| v \|_{V^p} = \sup_{(t_j)_{j = 1}^N \in \cP} \Big( \sum_{j = 1}^{N-1} |v(t_N)|^p + |v(t_{j+1}) - v(t_j)|^p\Big)^{\frac{1}{p}}.
\end{align*}
The space $V^p = V^p(\R, \C)$ consists of the functions with finite $V^p$-norm, i.e.,
\begin{align*}
	V^p := \{v \colon \R \rightarrow \C \, | \, \|v\|_{V^p} < \infty\}.
\end{align*}
Finally, let 
\begin{align*}
	V^p_0 := \{v \colon \R \rightarrow \C \, | \, v \text{ is right-continuous}, \, \lim_{t \rightarrow - \infty} v(t) = 0, \, \| v \|_{V^p} < \infty\}.
\end{align*}
Both $V^p$ and $V^p_0$ are Banach spaces when equipped with the $V^p$-norm, see~\cite{CH18} and the references therein. On $V_0^p$ the $p$-variation $|\cdot|_{V^p}$ is an equivalent norm to $\|\cdot\|_{V^p}$ and we will mainly use $|\cdot|_{V^p}$ on this space.

Let us prepare the proof of the geometric Brownian motion being in $V^p_0$ by noting that $| \cdot |_{V^p}$ is measurable on the space of continuous functions.

\begin{lemma}
	\label{lem:MeasurabilityVpnorm}
	The functional
	\begin{align*}
		| \cdot |_{V^p} \colon (C(\R), \cB(C(\R))) \rightarrow (\overline{\R}, \cB(\overline{\R}))
	\end{align*}
	is measurable, where $\cB(C(\R))$ and $\cB(\overline{\R})$ denote the Borel-$\sigma$-algebras on $C(\R)$ and $\overline{\R}$, respectively.
\end{lemma}

\begin{proof}
	Let $\cP_{\Q} := \{ (t_j)_{j = 1}^N \colon N \in \N, t_j \in \Q, t_j < t_{j+1}\}$ denote the set of rational partitions of $\R$. Note that $\cP_\Q$ is countable. Let $v \in C(\R)$ and $(t_j)_{j = 1}^N $ be a partition in $\cP$. For any $\epsilon > 0$ there exists a rational partition $(s_j)_{j = 1}^N \in \cP_\Q$ such that
 \begin{align*}
    \Big| \Big( \sum_{j = 1}^{N-1} |v(t_{j+1}) - v(t_j)|^p\Big)^{\frac{1}{p}} - \Big( \sum_{j = 1}^{N-1} |v(s_{j+1}) - v(s_j)|^p\Big)^{\frac{1}{p}} \Big| < \epsilon
\end{align*}
as $v$ is continuous. We infer that
\begin{align*}
    \Big( \sum_{j = 1}^{N-1} |v(t_{j+1}) - v(t_j)|^p\Big)^{\frac{1}{p}} \leq \sup_{(s_j)_{j=1}^N \in \cP_\Q} \Big( \sum_{j = 1}^{N-1} |v(s_{j+1}) - v(s_j)|^p\Big)^{\frac{1}{p}}.
\end{align*}
Hence, $|v|_{V^p} \leq \sup_{(s_j)_{j=1}^N \in \cP_\Q} \Big( \sum_{j = 1}^{N-1} |v(s_{j+1}) - v(s_j)|^p\Big)^{\frac{1}{p}}$ and since the reverse inequality is trivial, we conclude that
\begin{align*}
    |v|_{V^p} = \sup_{(s_j)_{j=1}^N \in \cP_\Q} \Big( \sum_{j = 1}^{N-1} |v(s_{j+1}) - v(s_j)|^p\Big)^{\frac{1}{p}}.
\end{align*}
Denote the point evaluations by $\pi_t \colon C(\R) \rightarrow \C$, $\pi_t(v) = v(t)$ and introduce the map
\begin{align*}
    \pi_P \colon C(\R) \rightarrow \R, \quad \pi_P(v) = \Big( \sum_{j = 1}^{N-1} |v(t_{j+1}) - v(t_j)|^p\Big)^{\frac{1}{p}}
\end{align*}
for every partition $P = (t_j)_{j = 1}^N \in \cP$. Clearly, $\pi_P$ is measurable for every $P \in \cP$ so that $| \cdot |_{V^p} = \sup_{P \in \cP_\Q} \pi_P$ implies that $| \cdot |_{V^p}$ is measurable from $(C(\R), \cB(C(\R)))$ to $(\overline{\R}, \cB(\overline{\R}))$ as a countable supremum of measurable functions.
\end{proof}

We will also exploit the H{\"o}lder-continuity property of Brownian motions. To that purpose, 
let us denote the $C^{0,\alpha}$-H{\"o}lder norm over an interval $I$ by $\|\cdot\|_{0,\alpha,I}$, i.e.,
\begin{align*}
	\|v\|_{0,\alpha,I} = \sup_{s,t \in I, s \neq t} \frac{|v(t) - v(s)|}{|t - s|^\alpha}.
\end{align*}
Note that $\| \cdot \|_{0,\alpha,I}$ is measurable on $C(I)$ by a similar argument as in Lemma~\ref{lem:MeasurabilityVpnorm} for every interval~$I$. 
The following lemma shows the H{\"o}lder norm of Brownian motion $\beta$ over intervals of constant length is uniformly bounded in expectation, 
due to the invariance 
$\PP \circ (\beta(\cdot))^{-1} = \PP \circ (\beta(\cdot+n) - \beta(n))^{-1}$ on $C([0,\infty))$.

\begin{lemma}
	\label{lem:FiniteExpectationHoelderNorm}
	Let $\beta$ be a one-dimensional Brownian motion and $\alpha \in (0,\frac{1}{2})$. Then
	\begin{align*}
		\sup_{n \in \N_0} \bbe (\| \beta\|_{0,\alpha,[n,n+1]}) 
      =\bbe (\| \beta\|_{0,\alpha,[0,1]}) 
      < \infty.
	\end{align*}
\end{lemma}

The main result of this subsection is formulated below which shows that the geometric Brownian motion belongs to $V^p_0$ for every $p > 2$.

\begin{proposition} [Global-in-time $V^p$ control]
	\label{prop:GeometricBrownianMotionVp}
	Let $\beta$ be a one-dimensional Brownian motion. Let $h$ be the geometric Brownian motion
	\begin{align*}
		h(t) = e^{-2 \beta(t) - 2 t} \qquad \text{on } [0,\infty),
	\end{align*}
 extended by $h(t) = t + 1$ for $-1 \leq t < 0$ and $h(t) = 0$ for $t < -1$ to the real line.
 Then, for every $p \in (2,\infty)$ we have $h \in V^p_0$, $\PP$-a.s.
\end{proposition}

\begin{proof}
	For every interval $I$,  
 set $\cP_I := \{\tau = (t_j)_{j = 1}^N \colon N \in \N, \,  t_j \in I, \, t_{j -1} < t_j\}$ and then for every $p \in (2,\infty)$
	\begin{align*}
		|v|_{V^p_I} = \sup_{(t_j)_{j = 1}^N \in \cP_I} \Big(\sum_{j = 1}^{N-1} |v(t_{j+1}) - v(t_j)|^p \Big)^{\frac{1}{p}}.
	\end{align*}
  Note that $|\cdot|_{V^p_I}$ is measurable by the same argument as in Lemma~\ref{lem:MeasurabilityVpnorm}.
    
	With the chosen extension, we have $h \in C(\R)$ $\PP$-a.s. and $\lim_{t \rightarrow - \infty} h(t) = 0$. 
 It is thus sufficient to show $|h|_{V^p} < \infty$ $\PP$-a.s. In view of our extension, for the latter it is sufficient to show
	\begin{align*}
		|h|_{V^p_{[0,\infty)}} < \infty \qquad \PP\text{-a.s.}
	\end{align*} 

 For this purpose, 
 we fix $p \in (2,\infty)$ and define the set
	\begin{align*}
		A &:= \{\omega \in \Omega \colon \exists n_0(\omega) \in \N,\, \forall n \geq n_0(\omega) \colon \|\beta\|_{0,\frac{1}{p},[n,n+1]} < e^{\frac{n}{16}}\} = \bigcup_{k \in \N} \bigcap_{n = k}^\infty \{\omega \in \Omega \colon \|\beta\|_{0,\frac{1}{p},[n,n+1]} < e^{\frac{n}{16}}\}.
	\end{align*}
	Writing $B_n := \{\omega \in \Omega \colon \|\beta\|_{0,\frac{1}{p},[n,n+1]} \geq e^{\frac{n}{16}}\}$, we thus have 
	\begin{align*}
		A^c = \bigcap_{k \in \N} \bigcup_{n = k}^\infty B_n = \limsup_{n \rightarrow \infty} B_n.
	\end{align*}
	By Lemma~\ref{lem:FiniteExpectationHoelderNorm}, $C_0 = \sup_{n \in \N_0} \bbe (\| \beta \|_{0,\frac{1}{p},[n,n+1]}) < \infty$, so that Markov's inequality yields
	\begin{align*}
		\PP(B_n) \leq e^{-\frac{n}{16}} \bbe \|\beta \|_{0,\frac{1}{p},[n,n+1]} \leq C_0 e^{-\frac{n}{16}}
	\end{align*}
	for all $n \in \N$. In particular, we have $\sum_{n \in \N} \PP(B_n) < \infty$ and thus
	\begin{align*}
		\PP(A^c) = \PP(\limsup_{n \rightarrow \infty} B_n) = 0
	\end{align*}
	by the Borel-Cantelli lemma. 
 It follows that $\PP(A) = 1$.
	
	Combining this result with the H{\"o}lder continuity properties and the iterated law of the logarithm of Brownian motions, we find a set $\tilde{\Omega} \subseteq \Omega$ with full measure, i.e. $\PP(\tilde{\Omega}) = 1$, such that for every $\omega \in \tilde{\Omega}$ there exists an index $n_0(\omega) \in \N$ such that
	\begin{enumerate}
		\item $\beta(\cdot, \omega)$ is $\frac{1}{p}$-H{\"o}lder continuous on $[0,n]$ for all $n \in \N$,
		\item $\|\beta(\cdot, \omega)\|_{0,\frac{1}{p},[n,n+1]} \leq e^{\frac{n}{16}}$ for all $n \geq n_0(\omega)$,
		\item $|\beta(t, \omega)| \leq 2 \sqrt{2 t \log(\log(t))} \leq \frac{t}{16}$ for all $t \geq n_0(\omega)$.
	\end{enumerate}
	We now fix an $\omega \in \tilde{\Omega}$ and claim that 
	\begin{equation}
		\label{eq:FiniteVarGeomBM}
		|h(\cdot, \omega)|_{V^p_{[0,\infty)}} < \infty,
	\end{equation}		
	which will imply the statement of the proposition. From now on the analysis will be pathwise for this fixed $\omega$ and 
 the $\omega$ dependence is dropped in order to ease the notation in the following.
	
	Let $n_0$ be as above. We denote the $\frac{1}{p}$-H{\"o}lder constant of $\beta$ on $[0,n_0]$ by $C_1$ and the maximum of $|\beta|$ on $[0,n_0]$ by $M$. Let $(t_j)_{j = 1}^N \in \cP_{[0,n_0]}$. We then infer
	\begin{align*}
		\sum_{j = 1}^{N-1} |h(t_{j+1}) - h(t_j)|^p &= \sum_{j = 1}^{N-1} \Big|e^{- 2 \beta(t_{j+1}) - 2 t_{j+1}} - e^{-2\beta(t_j) - 2 t_j}\Big|^p \\
		&\lesssim \sum_{j = 1}^{N-1} e^{- 2 t_{j+1} p} \Big|e^{- 2 \beta(t_{j+1})} - e^{-2\beta(t_j)}\Big|^p 
			+ \sum_{j = 1}^{N-1} e^{- 2 \beta(t_j)p} \Big| e^{- 2 t_{j+1}} - e^{-2 t_j} \Big|^p \\
		&\lesssim \sum_{j = 1}^{N-1} e^{2 M p} (|\beta(t_{j+1}) - \beta(t_j)|^p + |t_{j+1} - t_j|^p) \\
		&\lesssim e^{2 M p} \sum_{j = 1}^{N-1} (C_1^p + n_0^{p-1}) |t_{j+1} - t_j|\lesssim e^{2 M p} (C_1^p + n_0^{p-1}) n_0.
	\end{align*}
	Taking the supremum over all partitions in $\cP_{[0,n_0]}$, we obtain
	\begin{equation}
	\label{eq:FiniteVarGeomBMCompact}
		|h|_{V^p_{[0,n_0]}} \lesssim e^{2M}(C_1 + n_0) n_0^{\frac{1}{p}} < \infty.
	\end{equation} 
 
	To prove~\eqref{eq:FiniteVarGeomBM} it thus remains to show that $|h|_{V^p_{[n_0,\infty)}} < \infty$. We claim that for this statement it is actually enough to show
	\begin{equation}
		\label{eq:FiniteVarGeomBmSum}
		\sum_{n = n_0}^\infty |h|_{V^p_{[n,n+1]}} < \infty.
	\end{equation} 
 
	To see this claim, let $(t_j)_{j = 1}^N \in \cP_{[n_0,\infty)}$. Then there is an index $K \in \N$ and an increasing sequence of natural numbers $(n_k)_{k = 1}^{K}$ and $(l_k)_{k = n_1}^{n_{K-1}}$ such that
	\begin{align*}
		t_1, \ldots, t_{l_{n_1}} \in [n_1, n_1 + 1], \quad t_{l_{n_1} + 1}, \ldots, t_{l_{n_2}} \in [n_2, n_2 + 1], \quad \ldots, \quad t_{l_{n_{K-1} + 1}}, \ldots, t_N \in [n_K, n_K + 1].
	\end{align*}
	Setting $l_{n_0} = 0$ and $l_{n_K} = N$, we get
	\begin{equation}
	\label{eq:FiniteVarGeomBMExtrJumps}
		\sum_{j = 1}^{N-1} |h(t_{j+1}) - h(t_j)|^p = \sum_{k = 1}^{K} \sum_{l = l_{n_{k-1}}+1}^{l_{n_k} - 1} |h(t_{l+1}) - h(t_l)|^p + \sum_{k = 1}^{K-1} |h(t_{l_{n_k}+1}) - h(t_{l_{n_k}})|^p.
	\end{equation}
	For the second summand we estimate
	\begin{align*}
		\sum_{k = 1}^{K-1} |h(t_{l_{n_k}+1}) - h(t_{l_{n_k}})|^p &\lesssim \sum_{k = 1}^{K-1} \Big(e^{(-2\beta(t_{l_{n_k}+1}) - 2 t_{l_{n_k}+1}) p} + e^{(-2\beta(t_{l_{n_k}}) - 2 t_{l_{n_k}}) p}\Big)
		\lesssim \sum_{k = 1}^{K-1} \Big(e^{-t_{l_{n_k}+1} p} + e^{- t_{l_{n_k}} p}\Big) \\
		&\lesssim \sum_{k = 1}^{K-1} \Big(e^{- n_{k+1} p} + e^{- n_k p}\Big) \lesssim \sum_{n \in \N} e^{-n p} < \infty,
	\end{align*}
	where we used that $|\beta(t)| \leq {t}/{16}$ for all $t \geq n_0$ by our choice of the set $\tilde{\Omega}$. Consequently, if~\eqref{eq:FiniteVarGeomBmSum} is true, we obtain from~\eqref{eq:FiniteVarGeomBMExtrJumps}
	\begin{align*}
		\sum_{j = 1}^{N-1} |h(t_{j+1}) - h(t_j)|^p \lesssim  \sum_{k = 1}^{K} |h|_{V^p_{[n_k, n_k + 1]}}^p + \sum_{n \in \N} e^{-n p} \lesssim \sum_{n = n_0}^\infty |h|_{V^p_{[n, n + 1]}}^p + \sum_{n \in \N} e^{-n p} < \infty.
	\end{align*}
	Taking the supremum over all partitions $(t_j)_{j = 1}^N \in \cP_{[n_0,\infty)}$ thus yields $|h|_{V^p_{[n_0,\infty)}} < \infty$.
	
	It is now remains to prove~\eqref{eq:FiniteVarGeomBmSum}. Let $n \geq n_0$ and $(t_j)_{j = 1}^N$ be a partition in $\cP_{[n, n + 1]}$. We then estimate
	\begin{align*}
		\sum_{j = 1}^{N-1} |h(t_{j+1}) - h(t_j)|^p &\lesssim \sum_{j = 1}^{N-1} e^{- 2 t_{j+1} p} \Big|e^{- 2 \beta(t_{j+1})} - e^{-2\beta(t_j)}\Big|^p 
			+ \sum_{j = 1}^{N-1} e^{- 2 \beta(t_j)p} \Big| e^{- 2 t_{j+1}} - e^{-2 t_j} \Big|^p \\
			&\lesssim \sum_{j = 1}^{N-1} e^{- 2 t_{j+1} p} e^{\frac{t_{j+1}}{8} p} \Big|\beta(t_{j+1}) - \beta(t_j)\Big|^p 
			+ \sum_{j = 1}^{N-1} e^{\frac{t_j}{8} p} e^{- 2 t_j p} | t_{j+1} - t_j |^p,
	\end{align*}
	where we again used that $|\beta(t)| \leq {t}/{16}$ for all $t \geq n_0$. 
 It follows that 
	\begin{align*}
		\sum_{j = 1}^{N-1} |h(t_{j+1}) - h(t_j)|^p &\lesssim \sum_{j = 1}^{N-1} e^{- \frac{3}{2} n p} \|\beta \|_{0,\frac{1}{p},[n,n+1]}^p |t_{j+1} - t_j| 
			+ \sum_{j = 1}^{N-1} e^{-\frac{3}{2} n p} | t_{j+1} - t_j | \\
			&\lesssim e^{-\frac{3}{2} n p} (e^{\frac{n}{16} p} + 1) \lesssim e^{- n p},
	\end{align*}
	where we again exploited the definition of $\tilde{\Omega}$ in order to estimate $\|\beta\|_{0,\frac{1}{p},[n,n+1]}$. Taking the supremum over all these partitions yields
	\begin{align*}
		| h |_{V^p_{[n,n+1]}} \lesssim e^{-n}, 
	\end{align*} 
 where the implicit constant is independent of $n$. 
	Since $n \geq n_0$ was arbitrary, this implies~\eqref{eq:FiniteVarGeomBmSum} and thus the assertion of the proposition.
\end{proof}

\subsection{Trilinear estimates of wave nonlinearity}

We next estimate the nonlinearity of the wave equation in~\eqref{eq:RanZakNoncons}. The situation here is different from the deterministic one because of the presence of the geometric Brownian motion $h$. Although $h$ is independent of the spatial variable, the modulation components of the $W^{0,\frac{1}{4}, \frac{1}{2}}$-norm lead to intricate trilinear interactions involving the geometric Brownian motion.

One of the key steps in the proof is to uncover a subtle {\it nonresonance identity} 
that allows us to transfer some spatial regularity to temporal regularity of $h_c$, 
at the cost of the $1/8$-temproal regularity 
$h_c$ in the Besov space $B^{{1}/{8}}_{6,\infty}$, 
which is acceptable thanks to the global $V^p$ control derived in the previous subsection.

\begin{theorem} [Trilinear estimate for wave nonlinearity]
	\label{thm:TrilinearEstWave}
	Let $I \subseteq \R$ be an interval. If $\varphi, \psi \in \SOne(I)$ and $h \in L^6(I) \cap B^{\frac{1}{8}}_{6,\infty}(I)$, then
	\begin{align*}
		\|\cJ_0(h |\nabla|(\overline{\varphi} \psi)\|_{\WEner(I)} \lesssim (\|h\|_{L^6(I)} + \|h\|_{B^{\frac{1}{8}}_{6,\infty}(I)}) \| \varphi \|_{\SOne(I)} \| \psi \|_{\SOne(I)}.
	\end{align*}
\end{theorem}

\begin{proof}
By the definition of the involved norms, it is sufficient to prove the assertion for $I = \R$.
	We first note that $\lambda^{\frac{3}{4}}\|u\|_{L^2_t L^4_x} \lesssim \|u\|_{\SOne_\lambda}$ by~\eqref{eq:DefSsablambda} and Remark~\ref{rem:NormComp}. Interpolating with $\lambda \|u\|_{L^\infty_t L^2_x} \lesssim \|u\|_{\SOne_\lambda}$, we obtain
	\begin{equation}
		\label{eq:ControlL3txS1}
		\lambda^{\frac{5}{6}} \|u\|_{L^3_t L^3_x} \lesssim \|u\|_{\SOne_\lambda},
	\end{equation}
	which we will use frequently in the following without further reference.
	
	We will show the estimates
	\begin{align}
		\Big( \sum_{\lambda \in 2^{\N_0}} \|P_\lambda \cJ_0(h |\nabla|(\overline{\varphi} \psi)\|_{L^\infty_t L^2_x}^2\Big)^{\frac{1}{2}} &\lesssim \|h\|_{L^6_t} \|\varphi\|_{\SOne} \|\psi\|_{\SOne}, \label{eq:EstTrilinWaveLinfL2} \\
		\Big( \sum_{\lambda \in 2^{\N_0}} \lambda^{-1} \|P_\lambda (h |\nabla|(\overline{\varphi} \psi)\|_{L^2_t L^2_x}^2\Big)^{\frac{1}{2}} &\lesssim \|h\|_{L^6_t} \|\varphi\|_{\SOne} \|\psi\|_{\SOne}, \label{eq:EstTrilinWaveL2L2}  \\ 
  \Big( \sum_{\lambda \in 2^{\N_0}} \lambda^{-\frac{1}{2}} \|(\lambda + |\partial_t|)^{\frac{1}{4}} P_{\leq (\frac{\lambda}{2^8})^2}^{(t)} P_\lambda \cJ_0(h |\nabla|(\overline{\varphi} \psi)\|_{L^\infty_t L^2_x}^2\Big)^{\frac{1}{2}} &\lesssim (\|h\|_{L^6_t} + \|h\|_{B^{\frac{1}{8}}_{6,\infty}}) \|\varphi\|_{\SOne} \|\psi\|_{\SOne}, \label{eq:EstTrilinWaveTempLinfL2} 
	\end{align}
	which imply the assertion in view of the definition of $\|\cdot\|_{\WEner}$.
	
	Throughout the proof we will employ the usual paraproduct decomposition
	\begin{align*}
			P_\lambda(\overline{\varphi} \psi) = \sum_{\lambda/2 \leq \mu \leq 2 \lambda}  \overline{\varphi}_\mu \psi_{\ll \mu} + \sum_{\lambda_1 \sim \lambda_2 \gtrsim \lambda} P_\lambda(\overline{\varphi}_{\lambda_1}  \psi_{\lambda_2}) + \sum_{\lambda/2 \leq \mu \leq 2 \lambda}  \overline{\varphi}_{\ll \mu} \psi_\mu,
\end{align*}
where we set $\varphi_\lambda = P_\lambda \varphi$, $\varphi_{\ll \lambda} = P_{\ll \lambda} \varphi$, etc. As the estimates~\eqref{eq:EstTrilinWaveLinfL2} to~\eqref{eq:EstTrilinWaveL2L2} are invariant under complex conjugation, it is sufficient to prove them for the high-low and the high-high contributions.
	
	\textbf{Proof of~\eqref{eq:EstTrilinWaveLinfL2}:} We first use the energy estimate
		\begin{equation}
			\label{eq:EnergyEstWave}
				\|P_\lambda \cJ_0(h |\nabla|(\overline{\varphi} \psi))\|_{L^\infty_t L^2_x} \lesssim \lambda \| h P_\lambda (\overline{\varphi} \psi)\|_{L^1_t L^2_x}.
		\end{equation}
		Then decompose the high-low contribution as
		\begin{align*}
			\overline{\varphi}_{\mu} \psi_{\ll \mu} = (C_{\ll \mu^2} \overline{\varphi}_{\mu}) \psi_{\ll \mu} + (C_{\gtrsim \mu^2} \overline{\varphi}_{\mu}) \psi_{\ll \mu}.
		\end{align*}
	For the low modulation part we estimate
	\begin{align*}
		\lambda \|h (C_{\ll \mu^2} \overline{\varphi}_{\mu}) \psi_{\ll \mu}\|_{L^1_t L^2_x} &\lesssim \lambda \|h\|_{L^6_t} \|C_{\ll \mu^2} \overline{\varphi}_{\mu} \|_{L^2_t L^4_x} \| \psi_{\ll \mu}\|_{L^3_t L^4_x} \\
		&\lesssim \|h\|_{L^6_t} \lambda \|C_{\ll \mu^2} \varphi_{\mu} \|_{L^2_t L^4_x} \sum_{\nu \ll \mu} \nu^{-\frac{1}{2}} \nu^{\frac{5}{6}} \|\psi_\nu\|_{L^3_t L^3_x} \\ 
		&\lesssim \|h\|_{L^6_t} \|\varphi_\mu\|_{\SOne_\mu} \|\psi\|_{\SOne},
	\end{align*}
	while we infer for the high modulation part 
	\begin{align*}
		\lambda \| h (C_{\gtrsim \mu^2} \overline{\varphi}_{\mu}) \psi_{\ll \mu} \|_{L^1_t L^2_x} &\lesssim \|h\|_{L^6_t} \lambda \|C_{\gtrsim \mu^2} \overline{\varphi}_{\mu}\|_{L^2_t L^2_x} \| \psi_{\ll \mu} \|_{L^3_t L^\infty_x} \\
		&\lesssim \|h\|_{L^6_t} \mu \cdot \mu^{-2} \cdot \mu^{\frac{1}{4}} \Big\|\Big( \frac{\mu + |\partial_t|}{\mu^2 + |\partial_t|}\Big)^{\frac{1}{4}} (\imu \partial_t + \Delta) \varphi_{\mu} \Big\|_{L^2_t L^2_x}  \sum_{\nu \ll \mu} \nu^{\frac{1}{2}} \nu^{\frac{5}{6}} \|\psi_\nu\|_{L^3_t L^3_x} \\
		&\lesssim \|h\|_{L^6_t} \|\varphi_\mu\|_{\SOne_\mu} \|\psi\|_{\SOne}. 
	\end{align*}
	 Taking the $l^2$-norm in $\lambda$, we thus obtain~\eqref{eq:EstTrilinWaveLinfL2} for the high-low contribution. The standard adaptions yield the high-high case, which finishes the proof of~\eqref{eq:EstTrilinWaveLinfL2}.
	
	\textbf{Proof of~\eqref{eq:EstTrilinWaveL2L2}:} Here we have to estimate
	\begin{align*}
		\lambda^{-\frac{1}{2}} \| h P_\lambda(|\nabla|(\overline{\varphi} \psi))\|_{L^2_t L^2_x}
			\lesssim \lambda^{\frac{1}{2}} \|h P_\lambda (\overline{\varphi} \psi)\|_{L^2_t L^2_x}.
	\end{align*}
	We again start with the high-low contribution and infer
	\begin{align*}
		\mu^{\frac{1}{2}} \|h \overline{\varphi}_\mu \psi_{\ll \mu} \|_{L^2_t L^2_x}
		&\lesssim \mu^{\frac{1}{2}} \|h\|_{L^6_t} \|\overline{\varphi}_\mu\|_{L^\infty_t L^2_x} \|\psi_{\ll \mu} \|_{L^3_t L^\infty_x}  \\ 
		&\lesssim \|h\|_{L^6_t}  \mu \|\varphi_\mu\|_{L^\infty_t L^2_x} \mu^{-\frac{1}{2}} \sum_{\nu \ll \mu} \nu^{\frac{1}{2}} \nu^{\frac{5}{6}} \|\psi_\nu\|_{L^3_t L^3_x} \\
		&\lesssim \|h\|_{L^6_t} \|\varphi_\mu\|_{\SOne_\mu} \|\psi\|_{\SOne}.
	\end{align*}
	The standard adaptions also imply the corresponding estimate for the high-high contribution. Taking the $l^2$-norm in $\lambda$, we thus obtain~\eqref{eq:EstTrilinWaveL2L2}.
	
	\textbf{Proof of~\eqref{eq:EstTrilinWaveTempLinfL2}:}  
 Now we come to the most delicate estimate \eqref{eq:EstTrilinWaveTempLinfL2}. 
 First, we split the required estimate into
	\begin{align*}
		\lambda^{-\frac{1}{4}} \|(\lambda + |\partial_t|)^{\frac{1}{4}} P_{\ll \lambda^2}^{(t)}P_\lambda \cJ_0(h|\nabla|(\overline{\varphi} \psi) \|_{L^\infty_t L^2_x} 
		&\lesssim \lambda^{-\frac{1}{4}} \|(\lambda + |\partial_t|)^{\frac{1}{4}} P^{(t)}_{\lesssim \lambda} P_{\ll \lambda^2}^{(t)} P_\lambda \cJ_0(h|\nabla|(\overline{\varphi} \psi)) \|_{L^\infty_t L^2_x} \\
		&\qquad + \lambda^{-\frac{1}{4}} \|(\lambda + |\partial_t|)^{\frac{1}{4}} P^{(t)}_{\gg \lambda} P_{\ll \lambda^2}^{(t)} P_\lambda \cJ_0(h|\nabla|(\overline{\varphi} \psi)) \|_{L^\infty_t L^2_x}  \\ 
         &=: I_\lambda + II_\lambda.
	\end{align*}
	Since the temporal frequencies are bounded by $\lambda$ in the first summand, we can simply estimate
	\begin{align*}
		I_\lambda \lesssim \| P_\lambda \cJ_0(h|\nabla|(\overline{\varphi} \psi)) \|_{L^\infty_t L^2_x}.
	\end{align*}
	Taking the $l^2$-norm in $\lambda$, estimate~\eqref{eq:EstTrilinWaveLinfL2} yields the assertion for this part of the nonlinearity.
	
	It remains to estimate $II_\lambda$. Here we first employ the energy type inequality for the wave equation in the $W^{0,\frac{1}{4},\frac{1}{2}}$-norm from Lemma~2.6 in~\cite{CHN23}. To be more precise, the last estimate in the proof of~\cite[Lemma~2.6]{CHN23} yields
	\begin{align*}
		\|(\lambda + |\partial_t|)^{\frac{1}{4}} P^{(t)}_{\gg \lambda} P_{\ll \lambda^2}^{(t)} P_\lambda \cJ_0(h|\nabla|(\overline{\varphi} \psi)) \|_{L^\infty_t L^2_x} \lesssim \|(\lambda + |\partial_t|)^{\frac{1}{4}}  P_{\ll \lambda^2}^{(t)} P_\lambda (h|\nabla|(\overline{\varphi} \psi)) \|_{L^1_t L^2_x}
	\end{align*}
	and hence,
	\begin{equation}
		\label{eq:ReductionIILambda}
		II_\lambda \lesssim \lambda^{\frac{3}{4}} \|(\lambda + |\partial_t|)^{\frac{1}{4}}  P_{\ll \lambda^2}^{(t)} P_\lambda (h\overline{\varphi} \psi) \|_{L^1_t L^2_x}.
	\end{equation}

 In order to bound the above right-hand side, let us start with the high-low part. 
 
 {\bf $\bullet$ High-low part:} 
 Decompose into a low modulation and a high modulation contribution, i.e., for $\frac{\lambda}{2} \leq \mu \leq 2 \lambda$ we estimate
	\begin{align*}
		\lambda^{\frac{3}{4}} \|(\lambda + |\partial_t|)^{\frac{1}{4}}  P_{\ll \lambda^2}^{(t)} (h \overline{\varphi}_\mu \psi_{\ll \mu}) \|_{L^1_t L^2_x}
		&\lesssim \lambda^{\frac{3}{4}} \|(\lambda + |\partial_t|)^{\frac{1}{4}}  P_{\ll \lambda^2}^{(t)} (h C_{\gtrsim \mu^2} \overline{\varphi}_\mu \psi_{\ll \mu}) \|_{L^1_t L^2_x} \\  
        & \quad + \lambda^{\frac{3}{4}} \|(\lambda + |\partial_t|)^{\frac{1}{4}}  P_{\ll \lambda^2}^{(t)} (h C_{\ll \mu^2} \overline{\varphi}_\mu \psi_{\ll \mu}) \|_{L^1_t L^2_x} \\ 
      &=: II^{HL}_{HM,\lambda} 
         + II^{HL}_{LM,\lambda}.
	\end{align*}
	For the high modulation contribution we infer
	\begin{align}
		II^{HL}_{HM,\lambda} &\lesssim \lambda^{\frac{5}{4}} \| h C_{\gtrsim \mu^2} \overline{\varphi}_\mu \psi_{\ll \mu} \|_{L^1_t L^2_x} 
		\lesssim \lambda^{\frac{5}{4}} \|h\|_{L^6_t} \|C_{\gtrsim \mu^2} \varphi_\mu\|_{L^2_t L^2_x} \|\psi_{\ll \mu}\|_{L^3_t L^\infty_x} \nonumber\\
		&\lesssim \lambda^{\frac{5}{4}} \|h\|_{L^6_t} \mu^{-2} \mu^{\frac{1}{4}} \Big\|\Big( \frac{\mu + |\partial_t|}{\mu^2 + |\partial_t|} \Big)^{\frac{1}{4}} (\imu \partial_t + \Delta) \varphi_\mu \Big\|_{L^2_t L^2_x} \sum_{\nu \ll \mu} \nu^{\frac{1}{2}} \nu^{\frac{5}{6}} \|\psi_\nu\|_{L^3_t L^3_x} \nonumber\\
		&\lesssim \|h\|_{L^6_t} \|\varphi_\mu\|_{\SOne_\mu} \|\psi\|_{\SOne}. \label{eq:EstIIHLHM}
	\end{align}
	The low modulation part is the most subtle one. Here we employ the nonresonance identity
	\begin{align}
		\label{eq:NonresonanceId}
		P^{(t)}_{\ll \lambda^2}(h C_{\ll \mu^2} \overline{\varphi}_\mu \psi_{\ll \mu})
		= P^{(t)}_{\ll \lambda^2}(h C_{\ll \mu^2} \overline{\varphi}_\mu P^{(t)}_{> (\frac{\mu}{2^4})^2 }\psi_{\ll \mu}) + P^{(t)}_{\ll \lambda^2}(P^{(t)}_{\sim \mu^2} h C_{\ll \mu^2} \overline{\varphi}_\mu P^{(t)}_{\leq  (\frac{\mu}{2^4})^2 }\psi_{\ll \mu}),
	\end{align}
	which follows from the fact that $C_{\ll \mu^2} \overline{\varphi}_\mu$ has temporal frequency of size $\mu^2$. Recalling that $\psi_{\ll \mu} = P_{\leq \frac{\mu}{2^8}}\psi$, we obtain for the first summand
	\begin{align*}
		&\lambda^{\frac{3}{4}} \| (\lambda + |\partial_t|)^{\frac{1}{4}} P^{(t)}_{\ll \lambda^2}(h C_{\ll \mu^2} \overline{\varphi}_\mu P^{(t)}_{>(\frac{\mu}{2^4})^2}\psi_{\ll \mu}) \|_{L^1_t L^2_x}  \\ 
		&\lesssim \lambda^{\frac{5}{4}} \|h\|_{L^6_t} \| C_{\ll \mu^2} \overline{\varphi}_\mu\|_{L^2_t L^4_x} \| P^{(t)}_{>(\frac{\mu}{2^4})^2} \psi_{\ll \mu}\|_{L^3_t L^4_x} \\
		&\lesssim \sum_{\nu > (\frac{\mu}{2^4})^2} \|h\|_{L^6_t} \mu \|C_{\ll \mu^2} \varphi_\mu\|_{L^2_t L^4_x} \mu^{\frac{1}{4}}  \nu^{\frac{1}{6}} \mu \nu^{-1} \| P^{(t)}_{\nu} (\imu \partial_t + \Delta)\psi_{\ll \mu}\|_{L^2_t L^2_x} \\ 
       & \lesssim \|h\|_{L^6_t} \|\varphi_\mu\|_{\SOne_\mu} \|\psi\|_{\SOne},
	\end{align*}
	where we also employed Bernstein's inequality in space and time and used that
	\begin{align*}
		\| P^{(t)}_{\nu} (\imu \partial_t + \Delta)\psi_{\ll \mu}\|_{L^2_t L^2_x} 
		 \lesssim \Big( \sum_{\kappa \ll \mu} \Big\| P^{(t)}_{\nu} \Big(\frac{\kappa + |\partial_t|}{\kappa^2 + |\partial_t|} \Big)^{\frac{1}{4}} (\imu \partial_t + \Delta)\psi_{\kappa}\Big\|_{L^2_t L^2_x}^2\Big)^{\frac{1}{2}} 
		 \lesssim \|\psi\|_{\SOne}
	\end{align*}
	as $\nu \gtrsim \mu^2 \gg \kappa^2$.
	For the second summand in~\eqref{eq:NonresonanceId}, we infer
	\begin{align*}
		&\lambda^{\frac{3}{4}} \|(\lambda + |\partial_t|)^{\frac{1}{4}} P^{(t)}_{\ll \lambda^2}(P^{(t)}_{\sim \mu^2} h C_{\ll \mu^2} \overline{\varphi}_\mu P^{(t)}_{\leq  (\frac{\mu}{2^4})^2 }\psi_{\ll \mu}) \|_{L^1_t L^2_x} \\ 
		&\lesssim \lambda^{\frac{5}{4}} \|P^{(t)}_{\sim \mu^2} h\|_{L^6_t} \|C_{\ll \mu^2} \varphi_\mu\|_{L^2_t L^4_x} \|\psi_{\ll \mu}\|_{L^3_t L^4_x} \\
		&\lesssim (\mu^2)^{\frac{1}{8}} \|P^{(t)}_{\sim \mu^2} h\|_{L^6_t} \mu \|C_{\ll \mu^2} \varphi_\mu\|_{L^2_t L^4_x} \sum_{\nu \ll \mu} \nu^{-\frac{1}{2}} \nu^{\frac{5}{6}} \|\psi_\nu\|_{L^3_t L^3_x}  \\ 
		&\lesssim \|h\|_{B^{\frac{1}{8}}_{6,\infty}} \|\varphi_\mu\|_{\SOne_\mu} \|\psi\|_{\SOne}.
	\end{align*}
	In view of~\eqref{eq:NonresonanceId}, the last two estimates and~\eqref{eq:EstIIHLHM} imply
	\begin{equation}
		\label{eq:EstIIHL}
		\lambda^{\frac{3}{4}} \|(\lambda + |\partial_t|)^{\frac{1}{4}}  P_{\ll \lambda^2}^{(t)} (h \overline{\varphi}_\mu \psi_{\ll \mu}) \|_{L^1_t L^2_x}
		\lesssim  (\|h\|_{L^6_t} + \|h\|_{B^{\frac{1}{8}}_{6,\infty}}) \|\varphi_\mu\|_{\SOne_\mu} \|\psi\|_{\SOne}.
	\end{equation}
	Taking the $l^2$-norm in $\lambda$, we obtain the desired bound for the high-low part of $II$.

 {\bf $\bullet$ High-high part:}
	To control the high-high contribution of the right-hand side of~\eqref{eq:ReductionIILambda}, we proceed similarly as in the proof of~\eqref{eq:EstTrilinWaveLinfL2}. We estimate
	\begin{align}
		&\lambda^{\frac{3}{4}} \|(\lambda + |\partial_t|)^{\frac{1}{4}} P_{\ll \lambda^2}^{(t)} P_\lambda (h(\overline{\varphi} \psi)_{HH}) \|_{L^1_t L^2_x}
		\lesssim \lambda^{\frac{5}{4}} \sum_{\lambda_1 \sim \lambda_2 \gtrsim \lambda} \|h \varphi_{\lambda_1} \psi_{\lambda_2}\|_{L^1_t L^2_x} \nonumber\\
		&\lesssim  \lambda^{\frac{5}{4}} \sum_{\lambda_1 \sim \lambda_2 \gtrsim \lambda} (\|h C_{\ll \lambda_1^2}\varphi_{\lambda_1} \psi_{\lambda_2}\|_{L^1_t L^2_x} + \|h C_{\gtrsim \lambda_1^2}\varphi_{\lambda_1} \psi_{\lambda_2}\|_{L^1_t L^2_x}) =: II^{HH}_{LM, \lambda} + II^{HH}_{HM,\lambda}. \label{eq:EstIIHH}
	\end{align}
	For the low modulation contribution we then derive
	\begin{align*}
		II^{HH}_{LM, \lambda} &\lesssim \lambda^{\frac{5}{4}} \sum_{\lambda_1 \sim \lambda_2 \gtrsim \lambda} \|h\|_{L^6_t} \|C_{\ll \lambda_1^2} \varphi_{\lambda_1} \|_{L^2_t L^4_x} \|\psi_{\lambda_2}\|_{L^3_t L^4_x}\\
		&\lesssim \|h\|_{L^6_t} \lambda^{\frac{5}{4}} \sum_{\lambda_1 \sim \lambda_2 \gtrsim \lambda} \lambda_1^{-1} (\lambda_1 \|C_{\ll \lambda_1^2} \varphi_{\lambda_1} \|_{L^2_t L^4_x}) \lambda_2^{-\frac{1}{2}} (\lambda_2^{\frac{5}{6}} \|\psi_{\lambda_2}\|_{L^3_t L^3_x}) \\ 
		&\lesssim \|h\|_{L^6_t} \|\varphi\|_{\SOne} \|\psi\|_{\SOne} \lambda^{\frac{5}{4}} \sum_{\lambda_1 \gtrsim \lambda} \lambda_1^{-\frac{3}{2}} \\
		&\lesssim \lambda^{-\frac{1}{4}} \|h\|_{L^6_t} \|\varphi\|_{\SOne} \|\psi\|_{\SOne},
	\end{align*}
	while we get
	\begin{align*}
		II^{HH}_{HM,\lambda} &\lesssim  \lambda^{\frac{5}{4}} \sum_{\lambda_1 \sim \lambda_2 \gtrsim \lambda} \|h\|_{L^6_t} \|C_{\gtrsim \lambda_1^2} \varphi_{\lambda_1} \|_{L^2_t L^2_x} \|\psi_{\lambda_2}\|_{L^3_t L^\infty_x} \\
		&\lesssim \|h\|_{L^6_t} \lambda^{\frac{5}{4}} \sum_{\lambda_1 \sim \lambda_2 \gtrsim \lambda} \lambda_1^{-2} \lambda_1^{\frac{1}{4}} \Big\| \Big( \frac{\lambda_1 + |\partial_t|}{\lambda_1^2 + |\partial_t|}\Big)^{\frac{1}{4}} (\imu \partial_t + \Delta) \varphi_{\lambda_1} \Big\|_{L^2_t L^2_x} \lambda_2^{\frac{1}{2}} \lambda_2^{\frac{5}{6}} \|\psi_{\lambda_2}\|_{L^3_t L^3_x}\\
		&\lesssim \|h\|_{L^6_t} \|\psi\|_{\SOne} \sum_{\lambda_1  \gtrsim \lambda} \Big(\frac{\lambda}{\lambda_1}\Big)^{\frac{5}{4}} \|\varphi_{\lambda_1}\|_{\SOne_{\lambda_1}}
	\end{align*}
	for the high modulation part. Inserting the last two estimates into~\eqref{eq:EstIIHH} and taking the $l^2$-norm in $\lambda$, we finally obtain
	\begin{align*}
		\Big( \sum_{\lambda \in 2^{\N_0}} (\lambda^{\frac{3}{4}} \|(\lambda + |\partial_t|)^{\frac{1}{4}} P_{\ll \lambda^2}^{(t)} P_\lambda (h(\overline{\varphi} \psi)_{HH}) \|_{L^1_t L^2_x})^2 \Big)^{\frac{1}{2}} \lesssim \|h\|_{L^6_t} \|\varphi\|_{\SOne} \|\psi\|_{\SOne},
	\end{align*}
	which is the claimed bound for the high-high part of $II$ and thus finishes the proof of~\eqref{eq:EstTrilinWaveTempLinfL2}.
\end{proof}

\begin{remark}
The application of Theorem~\ref{thm:TrilinearEstWave} in the proof of Theorem~\ref{thm:RegNoise} thus requires a global Besov bound for the geometric Brownian motion. For this reason we have shown Proposition~\ref{prop:GeometricBrownianMotionVp}, which implies the corresponding bound 
due to the following Besov embedding of $V^p$ spaces 
\begin{equation}
	\label{eq:BesovEmbVp}
	\dot{B}^{\frac{1}{p}}_{p,1} \subseteq V^p_0 \subseteq \dot{B}^{\frac{1}{p}}_{p,\infty},
\end{equation}
see~\cite[Section~5]{CH18}, where $\dot{B}^{\frac{1}{p}}_{p,1}$ and $\dot{B}^{\frac{1}{p}}_{p,\infty}$ are the standard homogeneous Besov spaces on $\R$. 
\end{remark}

\subsection{Proof of Theorem~\ref{thm:RegNoise}} 

We have now collected all the tools for the proof of the noise regularization effects in Theorem~\ref{thm:RegNoise}.

Let $V$ be a solution of the linear wave equation. 
Recall that $\cU_V$ and $\cI_V$ denote the homogeneous and inhomogeneous solution operators, respectively, of the Schr{\"o}dinger equation with potential $V$. Theorem~7.1 in~\cite{CHN23} yields that $\cU_V$ and $\cI_V$ are continuous linear operators from $H^1(\R^4)$ and $\NOne(I)$ to $\SOne(I)$, respectively, for any interval $I$ and that there exists a constant $C = C(V)$, independent of $I$, $u_0$, and $g$, such that
\begin{equation}
	\label{eq:BoundsPropOpPot}
	\|\cU_V[u_0]\|_{\SOne(I)} \leq C(V) \|u_0\|_{H^1}, \qquad \|\cI_{V}[g]\|_{\SOne(I)} \leq C(V) \|g\|_{\NOne(I)}.
\end{equation}

As in the proof of Theorem~\ref{thm:LocalWP} we rewrite the problem. Setting $v_L(t) = e^{\imu t |\nabla|} Y_0$ and $\rho = v - v_L$, $(z,v)$ is a solution of~\eqref{eq:RanZakNoncons} if and only if $(z,\rho)$ solves
\begin{equation}  \label{eq:RanZakNonconsPot}
	\left\{\aligned
		(\imu \partial_t + \Delta - \Re(v_L)) z &= \Re(\rho) z, \qquad &z(0) &= X_0, \\
		(\imu \partial_t + |\nabla|) \rho &= - h_c |\nabla| |z|^2, &\rho(0) &= 0,
    \endaligned
    \right.
\end{equation}
where we recall that $c = \Im \phi^{(1)}_1$ 
and $h_c$ is defined in \eqref{h-W1-def}.
Since $\rho(t) = -\cJ_0[h_c |\nabla| |z|^2]$, we obtain a solution of~\eqref{eq:RanZakNonconsPot} - and thus of~\eqref{eq:RanZakNoncons} - if and only if
\begin{equation}
\label{eq:RanZakNonconsFixedPoint}
	z(t) = \cU_{v_L}[X_0](t) - \cI_{v_L}[\Re(\cJ_0[h_c |\nabla| |z|^2]) z](t).
\end{equation}

Define the fixed point operator $\Phi(X_0, Y_0; z)$ by the right-hand side of~\eqref{eq:RanZakNonconsFixedPoint}. With $C = C(v_L)$ the constant from~\eqref{eq:BoundsPropOpPot}, we set $R = 2 C(v_L)\|X_0\|_{H^1}$ and for some stopping time $\tau$
\begin{align*}
	B_R(\tau) := \{z \in \SOne([0,\tau)) \colon \|z\|_{\SOne([0,\tau))} \leq R\},
\end{align*}
which is a complete metric space equipped with the metric induced by $\|\cdot\|_{\SOne([0,\tau))}$. Combining the estimates~\eqref{eq:BoundsPropOpPot}, \eqref{eq:Bilinvu} from Lemma~\ref{lem:BilinearEstimates}, and~Theorem~\ref{thm:TrilinearEstWave}, we infer
\begin{align}
	\| \Phi(X_0, Y_0; z)\|_{\SOne([0,\tau))} &\leq C(v_L) \|X_0\|_{H^1} + C(v_L) \|\cJ_0[h_c |\nabla| |z|^2] z\|_{\NOne([0,\tau))} \nonumber\\
	&\leq  C(v_L) \|X_0\|_{H^1} + C \cdot C(v_L) \|\cJ_0[h_c |\nabla| |z|^2]\|_{\WEner([0,\tau))} \| z\|_{\SOne([0,\tau))} \nonumber\\
	&\leq  C(v_L) \|X_0\|_{H^1} + C \cdot C(v_L) (\|h_c\|_{L^6_t} + \|h_c\|_{B^{\frac{1}{8}}_{6,\infty}([0,\tau))}) \| z\|_{\SOne([0,\tau))}^3 \label{eq:EstRegNoiseFixOp}
\end{align}
for all $z \in \SOne([0,\tau))$. Arguing in the same way, we also obtain
\begin{align}
	&\| \Phi(X_0, Y_0; z) - \Phi(X_0, Y_0; w) \|_{\SOne([0,\tau))} \nonumber\\
	&\leq C \cdot C(v_L) (\|h_c\|_{L^6_t} + \|h_c\|_{B^{\frac{1}{8}}_{6,\infty}([0,\tau))}) (\|z\|_{\SOne([0,\tau))}^2 + \|w\|_{\SOne([0,\tau))}^2) \| z - w \|_{\SOne([0,\tau))} \label{eq:EstRegNoiseDiffFixOp}
\end{align}
for all $z,w \in \SOne([0,\tau))$.  

Now, fix $C$ as the maximum of the generic constants on the right-hand sides of~\eqref{eq:EstRegNoiseFixOp} and~\eqref{eq:EstRegNoiseDiffFixOp} and define the stopping time $\tau$ as
\begin{align*}
	\tau_c := \inf\left\{t \geq 0 \colon 2 C \cdot C(v_L) (\|h_c\|_{L^6([0,t))} + \|h_c\|_{B^{\frac{1}{8}}_{6,\infty}([0,t))})R^2 > \frac{1}{2}\right\}.
\end{align*}
Then, it follows from~\eqref{eq:EstRegNoiseFixOp} and~\eqref{eq:EstRegNoiseDiffFixOp} that 
\begin{align*}
	&\| \Phi(X_0, Y_0; z)\|_{\SOne([0,\tau_c))} \leq \frac{R}{2} + \frac{R}{2} = R, \\
	&\| \Phi(X_0, Y_0; z) - \Phi(X_0, Y_0; w) \|_{\SOne([0,\tau_c))} \leq \frac{1}{2}\| z - w \|_{\SOne([0,\tau_c))}
\end{align*}
for all $z,w \in B_R(\tau)$. Consequently, $\Phi$ has a unique fixed point $z$ in $B_R(\tau)$. Uniqueness of $z$ in $\SOne([0,\tau_c))$ then follows from standard arguments. 

Below we will show that $\PP(\tau_c = \infty) \longrightarrow 1$ as $c \rightarrow \infty$. Note that the arguments which yield this statement also show that $\tau_c > 0$ a.s. for every $c$.

Define the event
\begin{align*}
	A_c :=\{\omega \in \Omega \colon \tau_c(\omega) = \infty\}.
\end{align*}
We first prove that 
\begin{equation}
\label{eq:ProbATo1}
	\PP(A_c) \longrightarrow 1 \quad \text{as } c \rightarrow \infty.
\end{equation} 

For this purpose, 
we define $C' := C'(X_0, Y_0)$ as $C' = (2 C \cdot C(v_L) R^2)^{-1}$. In view of the definition of $\tau_c$ it is thus sufficient to prove that
\begin{equation}
\label{eq:ProbAcTo0}
	\PP(\{\|h_c\|_{L^6([0,\infty))} + \|h_c\|_{B^{\frac{1}{8}}_{6,\infty}([0,\infty))} \geq C'\}) \longrightarrow 0 \quad \text{as } c \rightarrow \infty.
\end{equation}
To that purpose, we first extend $h_c$ to $\R$ by $h_c(t) = c^2 t + 1$ for $-\frac{1}{c^2} \leq t < 0$ and $h_c(t) = 0$ for $t < -\frac{1}{c^2}$.
By interpolation we have
\begin{align}
	\|h_c\|_{B^{\frac{1}{8}}_{6,\infty}} \lesssim \|h_c\|_{L^{15}}^{\frac{5}{8}} \|h_c\|_{B^{\frac{1}{3}}_{3,\infty}}^{\frac{3}{8}} \lesssim \|h_c\|_{L^{15}}^{\frac{5}{8}} (\|h_c\|_{L^3}^{\frac{3}{8}} + \|h_c\|_{\dot{B}^{\frac{1}{3}}_{3,\infty}}^{\frac{3}{8}})
	\lesssim \|h_c\|_{L^{15}}^{\frac{5}{8}} (\|h_c\|_{L^3}^{\frac{3}{8}} + |h_c|_{V^3}^{\frac{3}{8}}), \label{eq:InterpolBesovNorm}
\end{align}
where the norms are taken over $\R$. For every $c > 0$ we now define a map $\psi_c \colon C([0,\infty)) \rightarrow C(\R)$ by
\begin{align*}
    \psi_c[g](t) = \begin{cases}
        e^{- 2 g(t) - 2 c^2 t}, \quad &t \geq 0, \\
        c^2 t + 1, &-\frac{1}{c^2} \leq t < 0, \\
        0, &t < -\frac{1}{c^2},
    \end{cases}
\end{align*}
which is measurable when $C([0,\infty))$ and $C(\R)$ are equipped with their respective Borel-$\sigma$-algebras.
Using the scaling property of Brownian motion, i.e. $\PP \circ (c \beta_1^{(1)}(\cdot))^{-1} = \PP \circ (\beta_1^{(1)}((c^2 \cdot ))^{-1}$ on $C([0,\infty))$, and recalling that $| \cdot |_{V^3}$ is measurable on $C(\R)$ by Lemma~\ref{lem:MeasurabilityVpnorm}, we infer
\begin{align*}
	& \PP(\{\|h_c\|_{L^{15}_t(\R)}^{\frac{5}{8}} |h_c|_{V^3(\R)}^{\frac{3}{8}} \geq C'/3\})
 = \PP(\{\|\psi_c[c \beta_1^{(1)}] \|_{L^{15}_t(\R)}^{\frac{5}{8}} \big|\psi_c[c \beta_1^{(1)}]\big|_{V^3(\R)}^{\frac{3}{8}} \geq C'/3\}) \\
	&= \PP(\{\|\psi_c[\beta_1^{(1)}(c^2 \cdot)]\|_{L^{15}_t(\R)}^{\frac{5}{8}} |\psi_c[\beta_1^{(1)}(c^2 \cdot)]|_{V^3(\R)}^{\frac{3}{8}} \geq C'/3\}) =: I.
\end{align*}
Since $|\cdot|_{V^p}$ is invariant under rescaling, the definition of $\psi_c$ and Proposition~\ref{prop:GeometricBrownianMotionVp} imply that 
\begin{align*}
    |\psi_c[\beta_1^{(1)}(c^2 \cdot)]|_{V^3(\R)} = |\psi_1[\beta_1^{(1)}]|_{V^3([0,\infty))} < \infty \quad  \text{a.s.}
\end{align*}
Hence, there is a constant $C_2 = C_2(\omega)$ such that
\begin{align*}
	I &\leq \PP(\{C_2 \|\psi_c[\beta_1^{(1)}(c^2 \cdot)]\|_{L^{15}_t(\R)}^{\frac{5}{8}} \geq C'/3\})
 = \PP(\{C_2 \|\psi_1[\beta_1^{(1)}]\|_{L^{15}_t(\R)}^{\frac{5}{8}}  \geq c^{\frac{1}{12}} \cdot C'/3\}) \longrightarrow 0
\end{align*}
as $c \rightarrow \infty$. In the same way, we obtain
\begin{align*}
\PP(\{\|h_c\|_{L^{15}_t(\R)}^{\frac{5}{8}} \|h_c\|_{L^3_t(\R)}^{\frac{3}{8}} \geq C'/3\}) 
+ \PP(\{\|h_c\|_{L^6_t(\R)} \geq C'/3\}) \longrightarrow 0
\end{align*}
as $c \rightarrow \infty$. In view of~\eqref{eq:InterpolBesovNorm}, this implies~\eqref{eq:ProbAcTo0} and thus~\eqref{eq:ProbATo1}. 

Next, we show that for every $\omega \in A_c$ the solution $(z,v)$ of~\eqref{eq:RanZakNoncons} scatters, i.e., that the event $A_c$ coincides with $\Upsilon$ from~\eqref{eq:DefUpsilon}. In combination with~\eqref{eq:ProbATo1} this proves the assertion of the theorem.

For any $\omega \in A_c$, we have $\|z\|_{\SOne([0,\infty))} < \infty$ so that Theorem~7.1 in~\cite{CHN23} implies the existence of $z_+ \in H^1(\R^4)$ such that
\begin{align*}
	\lim_{t \rightarrow \infty} \|e^{- \imu t \Delta} z(t) - z_+\|_{H^1_x} = 0.
\end{align*}
To show that also $v$ scatters as $t \rightarrow \infty$, we employ estimate~\eqref{eq:EstTrilinWaveLinfL2} from the proof of Theorem~\ref{thm:TrilinearEstWave} to deduce for every $\omega \in A_c$, 
\begin{align*}
    \| e^{- \imu t |\nabla|} v(t) - e^{- \imu t' | \nabla |} v(t') \|_{L^2} 
    = \Big\| \int_{t'}^t e^{- \imu s |\nabla|} (h_c | \nabla | |z|^2)(s) \dd s \Big\|_{L^2}
    \lesssim \| h_c \|_{L^6([t',t])} \| z\|_{\SOne([t',t])}^2 \longrightarrow 0
\end{align*}
as $t',t \rightarrow \infty$. We conclude that $e^{-\imu t |\nabla|} v(t)$ converges in $L^2$, which finishes the proof of Theorem~\ref{thm:RegNoise}. 
\hfill \qed

\begin{remark}
We remark that the scattering behavior~\eqref{scatter-XY} also implies that the energy of the Schr\"odinger component vanishes, i.e.,
\begin{align}  \label{X-H1-0} 
  \lim\limits_{t\to \infty}  \|X(t)\|_{H^1}  = 0, \ \ \mathbb{P}\text{-a.s.} 
\end{align}  
In fact, since $\{e^{-\imu t \Delta}\}$ is unitary in $H^1$ 
and $e^{\wh \mu t - W_1(t)}$ is independent of the spatial variable, 
one has 
\begin{align}  \label{X-H1-vanish-exp}
 \|X(t)\|_{H^1} 
\leq& \|e^{-(\wh \mu t - W_1(t)) } (e^{-it \Delta} e^{\wh \mu t - W_1(t)} X(t) - z_+) \|_{H^1} 
          +  e^{-\Re(\wh \mu t - W_1(t)) } \|z_+\|_{H^1},  
\end{align}  
where $z_+$ is the scattering state as in \eqref{scatter-XY}. 
Since $ e^{-\Re(\wh \mu t - W_1(t)) } =  e^{- \Im \phi_1^{(1)} \beta^{(1)}_1(t) -  (\Im \phi_1^{(1)})^2 t} $ 
converges asymptotically exponentially fast to zero  $\mathbb{P}$-a.s., one thus obtains \eqref{X-H1-0}. 

\end{remark}

\appendix

\section{Decomposability} 
\label{Sec-Decom}

We prove the decomposability for the $\XOne$-space,
which is used in the gluing procedure
when extending local solutions to the maximal existence time.

\begin{lemma}  [Decomposability]
	\label{lem:DecompX}
	Let $I_1, I_2 \subseteq \R$ be open intervals such that $I_1 \cap I_2 \neq \emptyset$. If $u$ belongs to $\XOne(I_1) \cap \XOne(I_2)$, then $u \in \XOne(I_1 \cup I_2)$ and
	\begin{align*}
		\|u\|_{\XOne(I_1 \cup I_2)} \lesssim (1 + |I_1 \cap I_2|^{-\frac{1}{2}}) (\|u\|_{\XOne(I_1)} + \|u\|_{\XOne(I_2)}).
	\end{align*}
\end{lemma}
\begin{proof}
	We fix a function $\rho \in C^\infty(\R)$ with $\rho(t) = 1$ for $t \leq -1$ and $\rho(t) = 0$ for $t \geq 1$ such that
	\begin{align*}
		\rho(t) + \rho(-t) = 1
	\end{align*}
	for all $t \in \R$. After translation in time, we can assume that $I_1 \cap I_2 = (-\epsilon, \epsilon)$ for some $\epsilon > 0$. Moreover, we assume that $\inf I_1 \leq \inf I_2$.
    Let $\rho_1(t) := \rho(\epsilon^{-1} t)$ and $\rho_2(t) := \rho(-\epsilon^{-1}t)$.
    Let $u_k$ be extensions of $u_{|I_k}$ with $\|u_k\|_{\XOne(\R)} \sim \|u\|_{\XOne(I_k)}$ for $k = 1,2$. By construction we then have
	\begin{equation}
		\label{eq:DecompNormLSDecu}
		u = \rho_1 u_1 + \rho_2 u_2 \qquad \text{on } I_1 \cup I_2.
	\end{equation}

	The decomposability of the $\SOne$-component of the $\XOne$-norm was demonstrated in Lemma~2.8 in~\cite{CHN23}. In the proof of that lemma it was shown that
	\begin{align}
		\label{eq:DecompNormSComp}
		\|\rho_k u_k\|_{\SOne(\R)} \lesssim (1 + \epsilon^{-\frac{1}{2}})\|u_k\|_{\XOne(\R)}.
	\end{align}
	It remains to provide an analogous localizability estimate for the lateral Strichartz component of the norm.
	
	In order to estimate $P_{\lambda, \vece_j} C_{\leq(\frac{\lambda}{2^8})^2} P_\lambda (\rho_k u_k)$ in the $L^{\infty,2}_{\vece_j}$-norm,
we see that $P_{\lambda, \vece_j} C_{\leq(\frac{\lambda}{2^8})^2} P_\lambda$ is a convolution operator with kernel $\phi_\lambda$, where $\phi_\lambda(t,x) = \lambda^6 \phi(\lambda^2 t, \lambda x)$ for a Schwartz function $\phi \in \Schw(\R \times \R^4)$.

We write
	\begin{align}
	\label{eq:DecompNormLSIntroCommutator}
		P_{\lambda, \vece_j} C_{\leq(\frac{\lambda}{2^8})^2} P_\lambda (\rho_k u_k)
		= \(P_{\lambda, \vece_j} C_{\leq(\frac{\lambda}{2^8})^2} P_\lambda (\rho_k u_k) -  \rho_k  P_{\lambda, \vece_j} C_{\leq(\frac{\lambda}{2^8})^2} P_\lambda u_k \)
          + \rho_k  P_{\lambda, \vece_j} C_{\leq(\frac{\lambda}{2^8})^2} P_\lambda u_k.
	\end{align}
	The commutator term can be written as
	\begin{align*}
		 &P_{\lambda, \vece_j} C_{\leq(\frac{\lambda}{2^8})^2} P_\lambda (\rho_k u_k) -  \rho_k  P_{\lambda, \vece_j} C_{\leq(\frac{\lambda}{2^8})^2} P_\lambda u_k \notag \\
		 & = \int_{\R \times \R^4} (\rho_k(t-s) - \rho_k(t)) \phi_\lambda(s,y) \tilde{P}_\lambda u_k(t-s, x-y) \dd(s,y) \nonumber\\
		 &=\int_{\R \times \R^4} \int_0^1 \rho_k'(t-\eta s) \dd \eta \cdot (-s) \phi_\lambda(s,y) \tilde{P}_\lambda u_k(t-s, x-y) \dd(s,y),
	\end{align*}
	and we estimate
	\begin{align}
	\label{eq:DecompNormLSCommutatorEst}
	&\|P_{\lambda, \vece_j} C_{\leq(\frac{\lambda}{2^8})^2} P_\lambda (\rho_k u_k) -  \rho_k  P_{\lambda, \vece_j} C_{\leq(\frac{\lambda}{2^8})^2} P_\lambda u_k\|_{L^{\infty,2}_{\vece_j}} \nonumber\\
	&\lesssim \lambda^{\frac{1}{2}} \|P_{\lambda, \vece_j} C_{\leq(\frac{\lambda}{2^8})^2} P_\lambda (\rho_k u_k) -  \rho_k  P_{\lambda, \vece_j} C_{\leq(\frac{\lambda}{2^8})^2} P_\lambda u_k\|_{L^2_{t,x}} \nonumber\\
	&\lesssim \lambda^{\frac{1}{2}}\int_{\R \times \R^4} \int_0^1 \|\rho_k'(t-\eta s)\|_{L^2_t} \dd \eta \cdot |s| |\phi_\lambda(s,y)| \|\tilde{P}_\lambda u_k(t-s, x-y)\|_{L^\infty_t L^2_x} \dd(s,y) \nonumber\\
	&\lesssim \lambda^{\frac{1}{2}} \|\rho_k'\|_{L^2_t} \|\tilde{P}_\lambda u_k\|_{L^\infty_t L^2_x} \lambda^{-2} \int_{\R \times \R^4} \lambda^6  |\lambda^2 s| |\phi(\lambda^2 s, \lambda y)|  \dd(s,y)   \nonumber \\
	&\lesssim \lambda^{-\frac{3}{2}} \|\rho_k'\|_{L^2_t} \|\tilde{P}_\lambda u_k\|_{L^\infty_t L^2_x}
     \lesssim \lambda^{-\frac 32} \epsilon^{-\frac{1}{2}}   \|\tilde{P}_\lambda u_k\|_{L^\infty_t L^2_x},
	\end{align}
where in the last step we used the fact that
	\begin{align}
	\label{eq:DecompNormLSDerRhoEst}
		\|\rho_1'\|_{L^2_t} = \Big(\int_\R |\epsilon^{-1} \rho'(\epsilon^{-1} t)|^2 \dd t\ \Big)^{\frac{1}{2}} = \epsilon^{-1} \Big(\int_\R | \rho'(t)|^2 \epsilon \dd t\ \Big)^{\frac{1}{2}} \lesssim \epsilon^{-\frac{1}{2}}
	\end{align}
	and the same bound for $\rho_2'$.

Moreover, for the last term on the right-hand side of \eqref{eq:DecompNormLSIntroCommutator},
we have
	\begin{align*}
		\|\rho_k  P_{\lambda, \vece_j} C_{\leq(\frac{\lambda}{2^8})^2} P_\lambda u_k\|_{L^{\infty,2}_{\vece_j}}
        & \lesssim \|\rho_k\|_{L^\infty_t} \|P_{\lambda, \vece_j} C_{\leq(\frac{\lambda}{2^8})^2} P_\lambda u_k\|_{L^{\infty,2}_{\vece_j}} \notag \\
        & \lesssim  \|P_{\lambda, \vece_j} C_{\leq(\frac{\lambda}{2^8})^2} P_\lambda u_k\|_{L^{\infty,2}_{\vece_j}}.
	\end{align*}

Thus, combining the above estimates we get
	\begin{align}
		\label{eq:DecompNormLSEst}
		\Big(\sum_{\lambda \in 2^\N} \lambda^3 \|P_{\lambda, \vece_j} C_{\leq(\frac{\lambda}{2^8})^2} P_\lambda (\rho_k u_k)\|_{L^{\infty,2}_{\vece_j}}^2 \Big)^{\frac{1}{2}}
		&\lesssim \epsilon^{-\frac{1}{2}} \Big(\sum_{\lambda \in 2^\N} \|\tilde{P}_\lambda u_k\|_{L^\infty_t L^2_x}^2 \Big)^{\frac{1}{2}} \notag \\
         & \qquad + \Big(\sum_{\lambda \in 2^\N} \lambda^3 \|P_{\lambda, \vece_j} C_{\leq(\frac{\lambda}{2^8})^2} P_\lambda  u_k\|_{L^{\infty,2}_{\vece_j}}^2 \Big)^{\frac{1}{2}} \nonumber \\
		&\lesssim (1 + \epsilon^{-\frac{1}{2}}) \|u_k\|_{\XOne(\R)}.
	\end{align}
In view of \eqref{eq:DecompNormLSDecu}, we thus infer
	\begin{align*}
		\|u\|_{\XOne(\R)} \leq \sum\limits_{i=1}^2 \|\rho_k u_k\|_{\XOne(\R)}
       \lesssim (1 + \epsilon^{-\frac{1}{2}}) \sum\limits_{i=1}^2  \|u_k\|_{\XOne(\R)}
         \lesssim ( 1 + |I_1 \cap I_2|^{-\frac{1}{2}}) (\|u\|_{\XOne(I_1)} + \|u\|_{\XOne(I_2)}).
	\end{align*}
The definition of the $\XOne(I_1 \cup I_2)$-norm now implies the assertion of the lemma.
\end{proof}

\section{Improvement of regularity}
In this part of the appendix we prove an improvement of regularity result, which is used in the proof of the blow-up alternative in Theorem~\ref{Thm-LWP}.

\begin{lemma}
	\label{lem:ImprovementRegularity}
	Let $I \subseteq \R$ be a finite interval with $0 = \min I$. Let $(u,v) \in C(I, H^{1}(\R^4) \times L^2(\R^4))$ be a solution of~\eqref{eq:RanZakbc} with
	\begin{align*}
		\|u\|_{L^\infty_t H^{1}_x(I \times \R^4)}  + \|u\|_{L^2_t W^{\frac{1}{2},4}_x(I \times \R^4)} + \|v\|_{L^\infty_t L^2_x(I \times \R^4)} < \infty.
	\end{align*}
	Then $(u,v) \in \SHalf(I) \times W^{0,0,0}(I)$. In particular, $v$ can be continuously extended to $\overline{I}$ in $L^2(\R^4)$.
\end{lemma}

\begin{proof}
Without loss of generality we assume $I = [0,T_1)$ for some $T_1 > 0$.
Standard estimates and embeddings imply
	\begin{align*}
		\|\Re(v) u\|_{L^2_t H^{-\frac{1}{2}}_x(I \times \R^4)} \lesssim \|v\|_{L^\infty_t L^2_x (I \times \R^4)} \|u\|_{L^2_t W^{\frac{1}{2},4}_x(I \times \R^4)}.
	\end{align*}
	Moreover, we have
	\begin{align*}
		&\|  b \cdot \nabla u - c u + \Re(\cT_t(W_2)) u \|_{L^2_t H^{-\frac{1}{2}}_x(I \times \R^4)}  \\ 
        & \lesssim \|b \cdot \nabla u\|_{L^2_t L^2_x(I \times \R^4)} + \|c u \|_{L^2_t L^2_x(I \times \R^4)} + \|\Re(\cT_t(W_2)) u\|_{L^2_t L^2_x(I \times \R^4)} \\
		& \lesssim T_1^{\frac{1}{2}} \|b\|_{L^\infty_t H^3_x(I \times \R^4)} \|u\|_{L^\infty_t H^1_x(I \times \R^4)}+ (\|c\|_{L^\infty_t H^1_x(I \times \R^4)} + \|\Re(\cT_t(W_2))\|_{L^\infty_t H^1_x(I \times \R^4)}) \|u\|_{L^2_t L^4_x(I \times \R^4)}.
	\end{align*}
	In view of the regularity properties of $b$, $c$, and $\cT_t(W_2)$, the last two estimates imply
\begin{align}
		\label{eq:SchrOpApplImprReg}
		\| (\imu \partial_t + \Delta) u\|_{L^2_t H^{-\frac{1}{2}}_x(I \times \R^4)}
		= \| \Re(v) u - b \cdot \nabla u - c u + \Re(\cT_t(W_2)) u \|_{L^2_t H^{-\frac{1}{2}}_x(I \times \R^4)}  < \infty.
	\end{align}
 Similarly, we obtain
	\begin{align}
		\label{eq:WaveOpApplToExtension}
		\| (\imu \partial_t + |\nabla|) v \|_{L^2_t H^{-1}_x}
		&\lesssim \| |u|^2 \|_{L^2_t L^2_x(I \times \R^4)} \lesssim \|u\|_{L^\infty_t L^{\frac{8}{3}}_x(I \times \R^4)} \|u\|_{L^2_t L^8_x(I \times \R^4)} \nonumber\\
		&\lesssim \|u\|_{L^\infty_t H^{\frac{1}{2}}_x(I \times \R^4)} \| u \|_{L^2_t W^{\frac{1}{2},4}_x(I \times \R^4)} < \infty.
	\end{align}
 We next show that $u$ can be continuously extended to $\overline{I}$ in $H^{\frac{1}{2}}(\R^4)$. To that purpose, we define for any $t' \in I$ the extensions
\begin{align}
		u_{t'}(t) &:= \one_{(-\infty,0)}(t) e^{\imu t \Delta} u_0 + \one_{[0,t')}(t) u(t) + \one_{[t',\infty)}(t) e^{\imu(t-t')\Delta}u(t'), \label{eq:DefExtension}\\
		v_{t'}(t) &:= \one_{(-\infty,0)}(t) e^{\imu t |\nabla|} v_0 + \one_{[0,t')}(t) v(t) + \one_{[t',\infty)}(t) e^{\imu(t-t')|\nabla|}v(t'). \nonumber
	\end{align}
 Note that $(u_{t'}, v_{t'}) \in C(\R, H^1 \times L^2)$ with
 \begin{align}
 \label{eq:EstExtImprovReg}
     \| u_{t'} \|_{\SHalf_w} &:= \| u_{t'} \|_{L^\infty_t H^{\frac{1}{2}}_x} + \| u_{t'} \|_{L^2_t W^{\frac{1}{2},4}_x} + \| (\imu \partial_t + \Delta) u_{t'} \|_{L^2_t H^{-\frac{1}{2}}_x} \nonumber\\
     &\lesssim \| u \|_{L^\infty_t H^{\frac{1}{2}}_x(I \times \R^4)} + \| u \|_{L^2_t W^{\frac{1}{2},4}_x(I \times \R^4)} + \| (\imu \partial_t + \Delta) u \|_{L^2_t H^{-\frac{1}{2}}_x(I \times \R^4)}, \\
     \| v_{t'} \|_{\WEndp_w} &:= \| v_{t'} \|_{L^\infty_t L^2_x} + \| (\imu \partial_t + |\na|) v_{t'} \|_{L^2_t H^{-1}_x} \lesssim  \| v \|_{L^\infty_t L^2_x(I \times \R^4)} + \| (\imu \partial_t + |\na|) v \|_{L^2_t H^{-1}_x(I \times \R^4)}, \nonumber
 \end{align}
 where the right-hand sides are independent of $t'$. For $t_1 < t_2$ we compute
 \begin{align*}
     \| e^{- \imu t_2 \Delta} u(t_2) - e^{- \imu t_1 \Delta} u(t_1) \|_{H^{\frac{1}{2}}_x} &= \Big\| \int_{t_1}^{t_2} e^{-\imu s \Delta} (\Re(v) u - b \cdot \nabla u - c u + \Re(\cT_{\cdot}(W_2)) u) \dd s \Big\|_{H^{\frac{1}{2}}_x} \\
     &\leq \Big\| \int_{t_1}^{t} e^{\imu (t- s) \Delta} (\Re(v) u - b \cdot \nabla u - c u + \Re(\cT_{\cdot}(W_2)) u) \dd s \Big\|_{\SHalf([t_1,t_2])} \\
     &\lesssim \| \Re(v) u \|_{\NHalf([t_1,t_2])} + \| b \cdot \nabla u + c u - \Re(\cT_{\cdot}(W_2)) u \|_{\GHalf([t_1,t_2])},
 \end{align*}
 where we employed Lemma~\ref{lem:LinEstimates} in the last step. For the first summand on the right-hand side we apply Proposition~6.1 in the Corrigendum of~\cite{CHN23}, which yields
 \begin{align*}
     \| \Re(v) u \|_{\NHalf([t_1,t_2])} &\lesssim \| v \|_{\WEndp_w([t_1,t_2])} \| u \|_{L^2_t W^{\frac{1}{2},4}_x([t_1,t_2] \times \R^4)}^{\frac{1}{2}} \| u \|_{\SHalf_w([t_1,t_2])}^{\frac{1}{2}} \\
     &\lesssim \| v_{t_2} \|_{\WEndp_w} \| u \|_{L^2_t W^{\frac{1}{2},4}_x([t_1,t_2] \times \R^4)}^{\frac{1}{2}} \| u_{t_2} \|_{\SHalf_w}^{\frac{1}{2}}.
 \end{align*}
 For the second summand we combine Lemma~\ref{lem:BilinearLowerOrder}~\ref{it:ControlNoiseN12} with~\eqref{eq:S12LSInProof} and~\eqref{eq:EstbnuLHLS} to infer
 \begin{align*}
     &\| b \cdot \nabla u + c u - \Re(\cT_{\cdot}(W_2)) u \|_{\GHalf([t_1,t_2])} \\
     &\lesssim (t_2 - t_1)^{\frac{1}{2}} \Big(\sum_{j = 1}^4 \| b \|_{L^{1,\infty}_{\vece_j}(I \times \R^4)} + \| b \|_{L^\infty_t H^3(I \times \R^4)} + \| c \|_{L^\infty_t H^2_x(I \times \R^4)} + \| \cT_t(W_2) \|_{L^\infty_t H^2_x(I \times \R^4)} \Big) \| u \|_{L^\infty_t H^1_x(I \times \R^4)}.
 \end{align*}
 Using the estimates in~\eqref{eq:EstExtImprovReg} and dominated convergence, we obtain that
 \begin{align*}
     \| e^{- \imu t_2 \Delta} u(t_2) - e^{- \imu t_1 \Delta} u(t_1) \|_{H^{\frac{1}{2}}_x} \longrightarrow 0
 \end{align*}
 as $t_1, t_2 \rightarrow T_1$. Hence, $e^{-\imu t \Delta} u(t)$ is Cauchy as $t \rightarrow T_1$ and we conclude that $u(t)$ converges in $H^{\frac{1}{2}}(\R^4)$ as $t \rightarrow T_1$. We call the latter limit $u(T_1)$.

 Replacing $t'$ by $T_1$ in~\eqref{eq:DefExtension}, we obtain an extension $u'$ of $u$ in $C(\R, H^{\frac{1}{2}}(\R^4))$. As in~\eqref{eq:EstExtImprovReg}, we see that $\| u' \|_{\SHalf_w} < \infty$.
Setting $v'(t) = e^{\imu t |\nabla| } v_0 - \cJ_0[|\nabla| | u'|^2]$, we obtain an extension of $v$. Proposition~6.2 in the Corrigendum of~\cite{CHN23} implies $\| v' \|_{\WEndp} < \infty$. Arguing as above for the extension of $u$, we then also get that $v' \in C(\R, L^2(\R^4))$. This shows in particular that $v \in \WEndp(I)$ and that $v$ can be continuously extended to~$\overline{I}$. Finally, arguing as above, we find that $\Re(v) u - b \cdot \nabla u - c u + \Re(\cT_{\cdot}(W_2) u$ belongs to $\GHalf(I)$, which implies $u \in \SHalf(I)$.
\end{proof}

\section{Continuity of the restriction norm}
Finally, we show the continuity of  
the adapted spaces in the endpoint of the time interval. This result is employed in both the proof of
local well-posedness in Theorem~\ref{thm:LocalWP} and in
a continuity argument in the proof of Theorem~\ref{Thm-GWP-Ground}.

\begin{lemma}
    \label{lem:Continuity}
    Let $T > 0$.
    \begin{enumerate}
        \item \label{it:ContSHalf} If $u \in \SHalf([0,T))$, then $t \mapsto \| u \|_{\SHalf([0,t))}$ is continuous on $(0,T)$.
        \item \label{it:ContWEner} If $v \in \WEner([0,T))$, then $t \mapsto \| v \|_{\WEner([0,t))}$ is continuous on $(0,T)$.
        \item \label{it:ContSumY} If $v \in \Y([0,T))$, then $t \mapsto \| v \|_{\Y([0,t]) + L^2_t W^{1,4}_x([0,t] \times \R^4)}$ is continuous on $(0,T)$.
    \end{enumerate} 
\end{lemma}

\begin{proof}
    We first show part~\ref{it:ContSHalf} and we start with the right-continuity of the map. Let $t_0 \in (0,T)$ and $\epsilon > 0$. Let $u_\epsilon$ be an extension of $u$ from $[0,t_0)$ to $\R$ with 
    \begin{align*}
        \| u_\epsilon \|_{\SHalf(\R)} \leq \| u \|_{\SHalf([0,t_0))} + \epsilon. 
    \end{align*}
    For any $t_1 \in (t_0,T)$ we define
    \begin{align*}
        \tilde{u}_{t_1}(t) = u_\epsilon(t) + \one_{[t_0,t_1]}(t) (u(t) - u_\epsilon(t)) + \one_{(t_1, \infty)}(t) e^{\imu(t- t_1)\Delta}(u(t_1) - u_\epsilon(t_1)),
    \end{align*}
    which extends $u$ from $[0,t_1)$ to a function in $C(\R,H^{\frac{1}{2}}(\R^4))$. Using Strichartz estimates, we then estimate
    \begin{align*}
        \| \tilde{u}_{t_1} - u_\epsilon \|_{\SHalf(\R)} 
        &\lesssim \Big(\sum_{\lambda \in 2^{\N_0}} (\lambda^{\frac{1}{2}} \| P_\lambda( u - u_\epsilon) \|_{L^\infty_t L^2_x((t_0,t_1) \times \R^4) } + \lambda^{\frac{1}{2}} \| P_\lambda( u - u_\epsilon) \|_{L^2_t L^4_x((t_0,t_1) \times \R^4) } \\
        &\hspace{6em} + \lambda^{-\frac{1}{2}} \| (\imu \partial_t + \Delta) P_\lambda( u - u_\epsilon) \|_{L^2_t L^2_x((t_0,t_1) \times \R^4) })^2 \Big)^{\frac{1}{2}} + \| u(t_1) - u_\epsilon(t_1) \|_{H^{\frac{1}{2}}(\R^4)}
    \end{align*}
    Since $u, u_\epsilon \in \SHalf([0,T))$, continuity and dominated convergence imply that
    \begin{align*}
        \lim_{t_1 \downarrow t_0} \| \tilde{u}_{t_1} - u_\epsilon \|_{\SHalf(\R)} = 0.
    \end{align*}
    We thus obtain
    \begin{align*}
       | \| u \|_{\SHalf([0,t_1))} -  \| u \|_{\SHalf([0,t_0))} | \leq \| \tilde{u}_{t_1} \|_{\SHalf(\R)} - \| u_\epsilon \|_{\SHalf(\R)} + \epsilon \leq \| \tilde{u}_{t_1} - u_\epsilon \|_{\SHalf(\R)} + \epsilon \leq 2 \epsilon
    \end{align*}
    for $t_1 - t_0$ small enough, which shows the right-continuity of $t \mapsto \| u \|_{\SHalf([0,t))}$.

    To prove the left-continuity, let $\epsilon > 0$ and let $\tilde{u}$ be an extension of $u$ from $[0,T)$ in $\SHalf(\R)$. Then there exists $\lambda_0 \in 2^\N$ such that 
    \begin{align*}
        \| \tilde{u} - P_{\leq \lambda_0} \tilde{u} \|_{\SHalf(\R)} \leq \epsilon,
    \end{align*}
    implying
    \begin{align*}
        \| u - P_{\leq \lambda_0} u \|_{\SHalf([0,t))} \leq \| u - P_{\leq \lambda_0} u \|_{\SHalf([0,T))} \leq  \| \tilde{u} - P_{\leq \lambda_0} \tilde{u} \|_{\SHalf(\R)} \leq \epsilon
    \end{align*}
    for all $t \in (0,T)$ by the monotonicity of the norm. Hence, it is enough to show the assertion for $u \in \SHalf([0,T))$ for which there exists a $\lambda_0 \in 2^\N$ such that $P_\lambda u = 0$ for all $\lambda > \lambda_0$.
    
    So we assume that $u$ has this property in the following. Let $t_0 \in (0,T)$, $(t_n)$ be a monotonically increasing sequence in $(0,t_0)$ with $\lim_{n \rightarrow \infty} t_n = t_0$, and $\epsilon > 0$. Set $g = (\imu \partial_t + \Delta) u$ on $(0,t_0)$. For each $n \in \N$ we take an extension $u_n$ of $u$ from $[0,t_n)$ in $\SHalf(\R)$ such that
    \begin{align}
         \label{eq:PropertiesExtension}
         \| u_n \|_{\SHalf(\R)} \leq \| u \|_{\SHalf([0,t_n))} + \epsilon
     \end{align}
     and define $g_n := (\imu \partial_t + \Delta)u_n$.

     Since $\| u_n \|_{\SHalf(\R)} \leq \| u \|_{\SHalf([0,t_n))} + \epsilon \leq  \| u \|_{\SHalf([0,t_0))} + \epsilon$ for all $n \in \N$, $(u_n)$ is a bounded sequence in $\SHalf(\R)$. In particular, $(u_n)$ is a bounded sequence in $\ell^2 L^\infty_t H^{\frac{1}{2}}_x(\R \times \R^4) \cap L^2(\R, B^{\frac{1}{2}}_{4,2}(\R^4))$ and $(g_n)$ is a bounded sequence in $L^2(\R, H^{-\frac{1}{2}}(\R^4))$, where we write $\ell^2 L^q_t H^{s}_x(\R \times \R^4)$ for the set of tempered distributions with
     \begin{align*}
         \| w \|_{\ell^2 L^q_t H^{s}_x} := \Big( \sum_{\lambda \in 2^{\N_0}} \lambda^{2s} \|P_\lambda w\|_{L^q_t L^2_x}^2 \Big)^{\frac{1}{2}}  < \infty.
     \end{align*}
     Since $\ell^2 L^\infty_t H^{\frac{1}{2}}_x(\R \times \R^4)$ is the dual of the separable space $\ell^2 L^1_t H^{-\frac{1}{2}}_x(\R \times \R^4)$ and $L^2(\R, B^{\frac{1}{2}}_{4,2}(\R^4))$ and $L^2(\R, H^{-\frac{1}{2}}(\R^4))$ are reflexive, we find a subsequence, again denoted by $(u_n)$ and $(g_n)$, such that $u_n \overset{\ast}{\rightharpoonup} \tilde{u}$ in $\ell^2 L^\infty_t H^{\frac{1}{2}}_x(\R \times \R^4)$, $u_n \rightharpoonup \tilde{\tilde{u}}$ in $L^2(\R, B^{\frac{1}{2}}_{4,2}(\R^4))$ and $g_n \rightharpoonup \tilde{g}$ in $L^2(\R, H^{-\frac{1}{2}}(\R^4))$. Since both weak* convergence in $\ell^2 L^\infty_t H^{\frac{1}{2}}_x(\R \times \R^4)$ and weak convergence in $L^2(\R, B^{\frac{1}{2}}_{4,2}(\R^4))$ imply convergence in $\cS'(\R \times \R^4)$, we have $\tilde{u} = \tilde{\tilde{u}}$ in $\cS'(\R \times \R^4)$. Moreover, testing with $\phi \in \cS(\R \times \R^4)$, we also infer that
     \begin{align*}
         (\imu \partial_t + \Delta) \tilde{u} = \tilde{g}.
     \end{align*}
     Since $\tilde{g} \in L^2(\R, H^{-\frac{1}{2}}(\R^4))$, the latter identity implies that $\tilde{u}$ has a representative in $C(\R,H^{-\frac{1}{2}}(\R^4))$. In the following, we identify $\tilde{u}$ with this representative. Since $u_n = u$ on $[0,t_n)$ for all $n \in \N$, we also have that $(u_n)$ converges to $u$ in $C(I, H^{\frac{1}{2}}(\R^4))$ for every compact subinterval $I \subseteq [0,t_0)$. Testing with $\phi \in C^\infty_c(0,t_0)$, we thus obtain $\tilde{u} = u$ on $(0,t_0)$. By the continuity properties of $\tilde{u}$ and $u$, we conclude that $\tilde{u} = u$ on $[0,t_0)$. In particular, $\tilde{u}$ is an extension of $u$ from $[0,t_0)$ to $\R$.
    
    We next note that the above convergence properties of $(u_n)$ and $(g_n)$ also imply that $P_\lambda u_n \overset{\ast}{\rightharpoonup} P_\lambda \tilde{u}$ in $L^\infty(\R, L^2(\R^4))$, $P_\lambda u_n \rightharpoonup P_\lambda \tilde{u}$ in $L^2(\R, L^4(\R^4))$, and $P_\lambda g_n \rightharpoonup P_\lambda \tilde{g}$ in $L^2(\R, L^2(\R^4))$ for every $\lambda \in 2^{\N_0}$. We thus obtain
    \begin{align}
        \label{eq:LeftContEstExt}
        \| \tilde{u} \|_{\SHalf(\R)} &= \Big(\sum_{\lambda \in 2^{\N_0}} (\lambda^{\frac{1}{2}} \| P_\lambda \tilde{u} \|_{L^\infty_t L^2_x} + \lambda^{\frac{1}{2}} \| P_\lambda \tilde{u} \|_{L^2_t L^4_x} + \lambda^{-\frac{1}{2}} \| P_\lambda \tilde{g} \|_{L^2_t L^2_x})^2\Big)^{\frac{1}{2}} \nonumber\\
        &\leq  \Big(\sum_{\lambda \in 2^{\N_0}} (\lambda^{\frac{1}{2}} \liminf_{n \rightarrow \infty} \| P_\lambda u_n \|_{L^\infty_t L^2_x} + \lambda^{\frac{1}{2}} \liminf_{n \rightarrow \infty} \| P_\lambda u_n \|_{L^2_t L^4_x} + \lambda^{-\frac{1}{2}} \liminf_{n \rightarrow \infty} \| P_\lambda g_n \|_{L^2_t L^2_x})^2\Big)^{\frac{1}{2}} \nonumber \\
        &\leq \liminf_{n \rightarrow \infty} \| u_n \|_{\SHalf(\R)} \leq \liminf_{n \rightarrow \infty} (\| u \|_{\SHalf([0,t_n))} + \epsilon) = \lim_{n \rightarrow \infty} \| u \|_{\SHalf([0,t_n))} + \epsilon,
    \end{align}
    where the monotonicity of the norm implies the existence of the limit in the last step. Since $\| \tilde{u} \|_{\SHalf(\R)} < \infty$, we can find $\lambda_1 > \lambda_0$ such that 
    \begin{align}
        \label{eq:LeftContSecondFreqCutoff}
        \| \tilde{u} - P_{\leq \lambda_1} \tilde{u} \|_{\SHalf(\R)} \leq \epsilon.
    \end{align}
    Since $P_\lambda u = 0$ for all $\lambda > \lambda_0$, we have
    \begin{align*}
        P_{\leq \lambda_1} \tilde{u} = P_{\leq \lambda_1} u = u
    \end{align*}
    on $[0,t_0)$. Moreover, since $\tilde{u}$ is continuous in $H^{-\frac{1}{2}}(\R^4)$, we have $P_{\leq \lambda_1} \tilde{u} \in C(\R, H^{\frac{1}{2}}(\R^4))$. Hence, $P_{\leq \lambda_1} \tilde{u}$ is an extension of $u$ from $[0,t_0)$ in $\SHalf(\R)$. Employing~\eqref{eq:LeftContEstExt} and~\eqref{eq:LeftContSecondFreqCutoff}, we finally estimate
    \begin{align*}
        | \| u \|_{\SHalf([0,t_0))} - \| u \|_{\SHalf([0,t_n))} | &\leq \| P_{\leq \lambda_1} \tilde{u} \|_{\SHalf(\R)} - \| u \|_{\SHalf([0,t_n))} \\
        &\leq \| \tilde{u} \|_{\SHalf(\R)} + | \| \tilde{u} \|_{\SHalf(\R)} - \| P_{\leq \lambda_1} \tilde{u} \|_{\SHalf(\R)} | -  \| u \|_{\SHalf([0,t_n))} \\
        &\leq \lim_{k \rightarrow \infty}  \| u \|_{\SHalf([0,t_k))} -  \| u \|_{\SHalf([0,t_n))} + 2 \epsilon \leq 3 \epsilon
    \end{align*}
    for all large enough $n$. Since $\epsilon > 0$ was arbitrary, we infer
    \begin{align*}
        \lim_{n \rightarrow \infty}  \| u \|_{\SHalf([0,t_n))} =  \| u \|_{\SHalf([0,t_0))},
    \end{align*}
    which concludes the proof of the left-continuity of $t \mapsto \| u \|_{\SHalf([0,t))}$.

    Part~\ref{it:ContWEner} follows similarly as part~\ref{it:ContSHalf}.

    To prove~\ref{it:ContSumY}, we again start with the right-continuity. We first fix an extension $\tilde{v}$ of $v$ from $[0,T)$ in $\Y(\R)$. Let $t_0 \in (0,T)$ and $\epsilon > 0$.

    We take $v_1 \in \Y([0,t_0])$ and $v_2 \in L^2_t W^{1,4}_x([0, t_0] \times \R^4)$ such that $v_1 + v_2 = v$ on $[0,t_0]$ and
    \begin{align*}
        \| v_1 \|_{\Y([0,t_0])} + \| v_2 \|_{L^2_t W^{1,4}_x([0, t_0] \times \R^4)} \leq \| v \|_{\Y([0,t_0]) + L^2_t W^{1,4}_x([0, t_0] \times \R^4)} + \epsilon.
    \end{align*}
    Fix an extension $\tilde{v}_1$ of $v_1$ from $[0, t_0]$ in $ \Y(\R)$ and set $\tilde{v}_2 = \tilde{v} - \tilde{v}_1$. Note that $\tilde{v}_2$ is an extension of $v_2$. Moreover, since $\tilde{v}_2 \in \Y(\R)$, there exists $\lambda_0 \in 2^\N$ such that
    \begin{align}
        \label{eq:ProofContYChoicelambda0}
        \| P_{> \lambda_0} \tilde{v}_2 \|_{\Y(\R)} \leq \epsilon \qquad \text{and} \qquad
        \| v_2 - P_{\leq \lambda_0} v_2 \|_{L^2_t W^{1,4}_x([0, t_0] \times \R^4)} \leq \epsilon.
    \end{align}
    Note that this yields the decomposition $v = \tilde{v}_1 + P_{> \lambda_0} \tilde{v}_2 + P_{\leq \lambda_0} \tilde{v}_2$ on every interval $[0,t]$ with $t \in (t_0, T)$. As $t \mapsto \| v \|_{\Y([0,t]) + L^2_t W^{1,4}_x([0, t] \times \R^4)}$ is monotonically increasing, we thus obtain
    \begin{align*}
        &| \| v \|_{\Y([0,t]) + L^2_t W^{1,4}_x([0, t] \times \R^4)} - \| v \|_{\Y([0,t_0]) + L^2_t W^{1,4}_x([0, t_0] \times \R^4)} | \\
        &\leq \| \tilde{v}_1 + P_{> \lambda_0} \tilde{v}_2 \|_{\Y([0,t])} + \| P_{\leq \lambda_0} \tilde{v}_2 \|_{ L^2_t W^{1,4}_x([0, t] \times \R^4)} - \| v_1\|_{\Y([0,t_0])} - \| v_2 \|_{ L^2_t W^{1,4}_x([0, t_0] \times \R^4)} + \epsilon \\
        &\leq\| \tilde{v}_1\|_{\Y([0,t])} + \| P_{\leq \lambda_0} \tilde{v}_2 \|_{ L^2_t W^{1,4}_x([0, t] \times \R^4)} - \| v_1\|_{\Y([0,t_0])} - \| v_2 \|_{ L^2_t W^{1,4}_x([0, t_0] \times \R^4)} + 2 \epsilon,
    \end{align*}
    where we have used~\eqref{eq:ProofContYChoicelambda0} in the last step. Employing part~\ref{it:ContWEner} and dominated convergence, we derive
    \begin{align*}
        &\| \tilde{v}_1\|_{\Y([0,t])} \longrightarrow \| \tilde{v}_1\|_{\Y([0,t_0])} = \| v_1\|_{\Y([0,t_0])}, \\
        &\| P_{\leq \lambda_0} \tilde{v}_2 \|_{ L^2_t W^{1,4}_x([0, t] \times \R^4)} \longrightarrow \| P_{\leq \lambda_0} \tilde{v}_2 \|_{ L^2_t W^{1,4}_x([0, t_0] \times \R^4)} = \| P_{\leq \lambda_0} v_2 \|_{ L^2_t W^{1,4}_x([0, t_0] \times \R^4)}
    \end{align*}
    as $t \downarrow t_0$. In view of~\eqref{eq:ProofContYChoicelambda0}, we conclude that there is $\delta > 0$ such that
    \begin{align*}
        | \| v \|_{\Y([0,t]) + L^2_t W^{1,4}_x([0, t] \times \R^4)} - \| v \|_{\Y([0,t_0]) + L^2_t W^{1,4}_x([0, t_0] \times \R^4)} | \leq 5 \epsilon
    \end{align*}
    for all $t \in (t_0,T)$ with $|t - t_0| < \delta$, which implies the right-continuity of $t \mapsto \| v \|_{\Y([0,t]) + L^2_t W^{1,4}_x([0,t] \times \R^4)}$.

    The left-continuity of this map follows from ideas already used in this proof. As in part~\ref{it:ContSHalf}, using the monotonicity of the norm, it suffices to prove the assertion for $v \in \Y([0,T))$ for which there is $\lambda_0 \in 2^\N$ such that $P_\lambda v = 0$ for all $\lambda > \lambda_0$. Let $t_0 \in (0,T)$ and $(t_n)$ be a sequence in $(0,t_0)$ converging to $t_0$. Let $\epsilon > 0$. For every $n \in \N$, we take $v_{1,n} \in \Y([0,t_n])$ and $v_{2,n} \in L^2_t W^{1,4}_x([0,t_n] \times \R^4)$ such that $v = v_{1,n} + v_{2,n}$ on $[0,t_n]$ and
    \begin{align*}
        \| v_{1,n} \|_{\Y([0,t_n])} + \| v_{2,n} \|_{L^2_t W^{1,4}_x([0,t_n] \times \R^4)} \leq \| v \|_{\Y([0,t_n]) + L^2_t W^{1,4}_x([0,t_n] \times \R^4)} + \epsilon,
    \end{align*}
    as well as extensions $\tilde{v}_{1,n}$ of $v_{1,n}$ in $\Y(\R)$ satisfying
    \begin{align*}
        \| \tilde{v}_{1,n} \|_{\Y(\R)} \leq \| v_{1,n} \|_{\Y([0,t_n])} + \epsilon.
    \end{align*}
    Then $(\| \tilde{v}_{1,n} \|_{\Y(\R)})_n$ is bounded. Arguing as in the proof of part~\ref{it:ContWEner}, i.e., adapting the ideas from part~\ref{it:ContSHalf}, we obtain a subsequence, again denoted by $(\tilde{v}_{1,n})$, such that $\tilde{v}_{1,n} \overset{\ast}{\rightharpoonup} \tilde{v}_1$ in $\ell^2 L^\infty_t L^2_x(\R \times \R^4)$ with $\tilde{v}_1 \in C(\R, H^{-\frac{1}{2}}(\R^4))$ and
    \begin{align*}
        \| \tilde{v}_1 \|_{\Y(\R)} \leq \liminf_{n \rightarrow \infty} \| \tilde{v}_{1,n} \|_{\Y(\R)} \leq \liminf_{n \rightarrow \infty} \| v_{1,n} \|_{\Y([0,t_n])} + \epsilon.
    \end{align*}
    Moreover, $(v_{2,n})$ weakly converges to some $v_2$ in $L^2_t W^{1,4}_x([0,t_0] \times \R^4)$ with
    \begin{align*}
        \| v_2 \|_{L^2_t W^{1,4}_x([0,t_0] \times \R^4)} \leq \liminf_{n \rightarrow \infty} \| v_{2,n} \|_{L^2_t W^{1,4}_x([0,t_n] \times \R^4)},
    \end{align*}
    where we take the trivial extension of $v_{2,n}$ to $[0,t_0]$. Arguing as in part~\ref{it:ContSHalf}, we infer that $v = \tilde{v}_1 + v_2$ on $[0,t_0]$. Using that $P_\lambda v = 0$ for all $\lambda > \lambda_0$, we have
    \begin{align*}
        v = P_{\leq \lambda_1} v = P_{\leq \lambda_1} \tilde{v}_1 + P_{\leq \lambda_1} v_2
    \end{align*}
    on $[0,t_0]$ and
    \begin{align*}
        \| \tilde{v}_1 -  P_{\leq \lambda_1} \tilde{v}_1 \|_{\Y(\R)} + \|v_2 - P_{\leq \lambda_1} v_2 \|_{L^2_t W^{1,4}_x([0,t_0] \times \R^4)} \leq \epsilon
    \end{align*}
    for sufficiently large $\lambda_1$. Then $P_{\leq \lambda_1} \tilde{v}_1 \in \Y(\R)$, $P_{\leq \lambda_1} v_2 \in L^2_t W^{1,4}_x([0,t_0] \times \R^4)$ and, using the monotonicity of the norm once more, we arrive at
    \begin{align*}
        &| \| v \|_{\Y([0,t_0]) + L^2_t W^{1,4}_x([0,t_0] \times \R^4)} - \| v \|_{\Y([0,t_n]) + L^2_t W^{1,4}_x([0,t_n] \times \R^4)} | \\
        &\leq \| P_{\leq \lambda_1} \tilde{v}_1 \|_{\Y([0,t_0])} + \| P_{\leq \lambda_1} v_2 \|_{L^2_t W^{1,4}_x([0,t_0] \times \R^4)}
         - \| v_{1,n} \|_{\Y([0,t_n])} - \| v_{2,n} \|_{L^2_t W^{1,4}_x([0,t_n] \times \R^4)} + \epsilon \\
         &\leq \| \tilde{v}_1 \|_{\Y(\R)} + \|v_2 \|_{L^2_t W^{1,4}_x([0,t_0] \times \R^4)}  - \| v_{1,n} \|_{\Y([0,t_n])} - \| v_{2,n} \|_{L^2_t W^{1,4}_x([0,t_n] \times \R^4)} + 2 \epsilon \\
         &\leq \liminf_{k \rightarrow \infty} \| v_{1,k} \|_{\Y([0,t_k])} + \liminf_{k \rightarrow \infty} \| v_{2,k} \|_{L^2_t W^{1,4}_x([0,t_k] \times \R^4)} - \| v_{1,n} \|_{\Y([0,t_n])} - \| v_{2,n} \|_{L^2_t W^{1,4}_x([0,t_n] \times \R^4)} + 3 \epsilon \leq 4\epsilon
    \end{align*}
    for all large enough $n$. This implies the left-continuity of $t \mapsto \| v \|_{\Y([0,t]) + L^2_t W^{1,4}_x([0,t] \times \R^4)}$, completing the proof of the lemma.        
\end{proof}

\section*{Acknowledgements} 
Funded by the Deutsche Forschungsgemeinschaft (DFG, German Research Foundation) -- Project-ID 317210226 -- SFB 1283. 
D.\ Zhang is also grateful for the NSFC grants (No. 12271352, 12322108) 
and Shanghai Frontiers Science Center of Modern Analysis.

The authors would like to thank Timothy Candy and Kenji Nakanishi for helpful discussions.

\bibliographystyle{abbrv}
\bibliography{StZEnergyCritical}

\end{document}